\documentclass[10pt]{article}
\usepackage[left=0.8in,top=0.8in,right=0.8in,bottom=0.8in]{geometry}

\usepackage{graphicx}
\usepackage{mathtools}
\usepackage{amssymb}
\usepackage{amsthm}
\usepackage{latexsym}
\usepackage{bm}
\usepackage{color}
\usepackage{stmaryrd}
\usepackage{mathrsfs}
\usepackage{array}
\usepackage{algorithm}
\usepackage{algpseudocode}
\usepackage{multirow}
\usepackage{url}
\usepackage{mdframed}
\usepackage{longtable}

\makeatletter
\newcommand*{\rom}[1]{\expandafter\@slowromancap\romannumeral #1@}
\makeatother

\newenvironment{myenv}[1]
  {\mdfsetup{
    frametitle={\colorbox{white}{\space#1\space}},
    innertopmargin=10pt,
    frametitleaboveskip=-\ht\strutbox,
    frametitlealignment=\center
    }
  \begin{mdframed}
  }
  {\end{mdframed}}

\DeclareMathAlphabet\mathbfcal{OMS}{cmsy}{b}{n}

\SetSymbolFont{stmry}{bold}{U}{stmry}{m}{n}

\newtheorem{definition}{Definition}
\newtheorem{lemma}{Lemma}

\newtheorem{proposition}{Proposition}

\newtheorem{remark}{Remark}

\newcolumntype{P}[1]{>{\centering\arraybackslash}p{#1}}

\AfterEndEnvironment{proof}{\noindent\ignorespaces}

\numberwithin{equation}{section}

\begin{document}
\title{A continuum and computational framework for viscoelastodynamics: \rom{3}. A nonlinear theory}
\author{Ju Liu$^{\textup{ *}}$, Jiashen Guan, Chongran Zhao, Jiawei Luo \\
\textit{\small Department of Mechanics and Aerospace Engineering,}\\
\textit{\small Southern University of Science and Technology,}\\
\textit{\small 1088 Xueyuan Avenue, Shenzhen, Guangdong 518055, China}\\
$^{*}$ \small \textit{E-mail address:} liuj36@sustech.edu.cn, liujuy@gmail.com
}
\date{}
\maketitle

\section*{Abstract}
We continue our investigation of viscoelasticity by extending the Holzapfel-Simo approach discussed in Part \rom{1} to the fully nonlinear regime. By scrutinizing the relaxation property for the non-equilibrium stresses, it is revealed that a kinematic assumption akin to the Green-Naghdi type is necessary in the design of the potential. This insight underscores a link between the so-called additive plasticity and the viscoelasticity model under consideration, further inspiring our development of a nonlinear viscoelasticity theory. Our strategy is based on Hill's hyperelasticity framework and leverages the concept of generalized strains. Notably, the adopted kinematic assumption makes the proposed theory fundamentally different from the existing models rooted in the notion of the intermediate configuration. The computation aspects, including the consistent linearization, constitutive integration, and modular implementation, are addressed in detail. A suite of numerical examples is provided to demonstrate the capability of the proposed model in characterizing viscoelastic material behaviors at large strains.

\vspace{5mm}

\noindent \textbf{Keywords:} 
Continuum mechanics, Constitutive modeling, Viscoelasticity, Hyperelasticity of Hill's class, Generalized strains, Green-Naghdi plasticity

\section{Introduction}
\subsection{Motivation}
\label{subsec:motivation}
The motivation of this work is to generalize the finite deformation \textit{linear} viscoelasticity models to the nonlinear regime, and it came from observations of our previous studies. Throughout this work, we refer to Liu, Latorre, and Marsden \cite{Liu2021b} as Part \rom{1}, and it was revealed in Part \rom{1} that the finite deformation linear viscoelasticity models, under the isothermal condition, are characterized by the following form of the configurational free energy
\begin{align}
\label{eq:intro-Upsilon-form}
\Upsilon\left(\tilde{\bm C}, \bm \Gamma \right) = \frac{1}{4\mu} \left\lvert 2 \frac{\partial \mathcal G(\tilde{\bm C})}{\partial \tilde{\bm C}} - \mu \left( \bm \Gamma - \bm I \right)  - \hat{\bm S}_{0} \right\lvert^2.
\end{align}
The constant symmetric tensor $\hat{\bm S}_{0}$ determines the initial value of $\bm Q := -2 \partial \Upsilon / \partial \bm \Gamma$. The linearity of such type of models comes from the above quadratic form of the energy $\Upsilon$ and the Newtonian assumption on the viscous behavior (i.e., there exists a constant rank-four viscosity tensor relating the stress and strain rate). The models based on the identical-polymer-chain assumption are recovered by taking the energy-like function $\mathcal G$ in \eqref{eq:intro-Upsilon-form} as the hyperelastic free energy scaled by a non-dimensional parameter. Although such type of modeling approach was successfully adopted in a variety of scenarios \cite{Govindjee1992,Holzapfel1996,Holzapfel2001,Wollner2024}, its configurational free energy lacks a clear physical interpretation. Another model is coined as the Holzapfel-Simo-Saint Venant-Kirchhoff model in Part \rom{1}, which is instantiated by taking
\begin{align*}
\mathcal G = \mu \left\lvert \frac{\tilde{\bm C} - \bm I}{2}\right\lvert^2,
\end{align*}
and $\hat{\bm S}_0 = \bm O$. That leads to a simpler form of the configurational free energy, i.e.,
\begin{align}
\label{eq:intro-Upsilon-HSSK}
\Upsilon\left(\tilde{\bm C}, \bm \Gamma \right) = \mu \left\lvert \frac{ \bm \tilde{\bm C} - \bm \Gamma }{2} \right\lvert^2.
\end{align}
The significance of the above energy can be made clearer if we introduce the volume-preserving Lagrangian strain $\tilde{\bm E}$ and a viscous counterpart $\bm E^{\mathrm v}$ as
\begin{align*}
\tilde{\bm E} := \frac{\tilde{\bm C} - \bm I }{2}, \quad \mbox{ and } \quad \bm E^{\mathrm v} := \frac{ \bm \Gamma - \bm I }{2}. 
\end{align*}
Then the energy \eqref{eq:intro-Upsilon-HSSK} can be expressed as
\begin{align}
\label{eq:intro-Upsilon-E-Ev}
\Upsilon\left(\tilde{\bm C}, \bm \Gamma \right) = \mu \left\lvert \tilde{\bm E} - \bm E^{\mathrm v} \right\lvert^2.
\end{align}
One insight one may glean from the energy form \eqref{eq:intro-Upsilon-E-Ev} is that the two strain-like tensors enter into the energy function in the form of  $\tilde{\bm E} - \bm E^{\mathrm{v}}$. This closely resembles the kinematic assumption made by Green and Naghdi in the realm of elastoplasticity modeling \cite{Green1965,Naghdi1990,Simo1985a,Xiao2006}. If one views the tensor $\tilde{\bm E} - \bm E^{\mathrm{v}}$ as the strain of the non-equilibrium elastic part, the reason of why we put ``Saint Venant-Kirchhoff" in the name of this model is self-explanatory. This perhaps is not so surprising as the original model introduced by Simo in 1987 \cite{Simo1987} is a geometrically nonlinear extension of the small-strain viscoelasticity theory, which also relies on an additive decomposition of the strain \cite[Chapter~10]{Simo2006}. Nevertheless, adopting the Saint Venant-Kirchhoff model is restrictive in material modeling and may suffer from mathematical issues \cite{Raoult1986}. In this work, we extend the model to the nonlinear regime by adopting the kinematic decomposition introduced initially by Green and Naghdi in elastoplasticity \cite{Green1965,Naghdi1990} as the outset.

\begin{table}[htbp]
\begin{center}
\tabcolsep=0.19cm
\renewcommand{\arraystretch}{2.6}
\begin{tabular}{ >{\centering\arraybackslash}m{3.0cm} | >{\centering\arraybackslash}m{5.0cm} | >{\centering\arraybackslash}m{5.0cm} }
 & Multiplicative decomposition  & Additive decomposition   \\
\hline
Elastoplasticity & Lee 1969 \cite{Lee1969} &  Green \& Naghdi 1965 \cite{Green1965}, Miehe 1998 \cite{Miehe1998} \\
\hline
Viscoelasticity & Sidoroff 1974 \cite{Sidoroff1974}, Reese \& Govindjee 1998 \cite{Reese1998} & Simo 1987 \cite{Simo1987}, Holzapfel \& Simo 1996 \cite{Holzapfel1996}
\end{tabular}
\end{center}
\caption{Elastoplastic and viscoelasticity models based on two kinematic assumptions. The viscoelasticity model based on the additive kinematic assumption (i.e., the lower right corner of the table) is also referred to as the finite deformation linear viscoelasticity \cite{Liu2021b}, meaning its nonlinearity is only geometrical.}
\label{table:inelasticity-modeling}
\end{table}

\subsection{Kinematic assumptions in finite inelasticity}
The above discussion brings us to a fundamental issue in inelasticity modeling, that is, how to quantify the inelastic kinematic response. There have been long-lasting debates on this issue, and there is no commonly-accepted approach to date. A widely adopted approach is through the multiplicative decomposition of the deformation gradient $\bm F=\bm F^{\mathrm{e}}\bm F^{\mathrm{i}}$, in which the elastic and inelastic deformations $\bm F^{\mathrm{e}}$ and $\bm F^{\mathrm{i}}$ are related to an imaginary stress-free intermediate configuration. Although this approach is nowadays often credited to Lee and his associates for their pioneering work in the late 1960s \cite{Lee1967,Lee1969}, this idea can be traced back to earlier works of Eckart \cite{Eckart1948} and Kr\"{o}ner \cite{Kroener1959}, among others \cite{Bruhns2020,Sadik2017}. Mathematically, the plastic deformation gradient belonging to $GL(3)_{+}$ is used as the internal state variable for characterizing inelastic deformation. Physically, its microscopic interpretation is due to the crystallographic slip in single-crystal metal plasticity, according to the continuum slip theory \cite{Asaro1983,Reina2016}. This strategy has been found to be applicable in thermomechanics \cite{Stojanovic1964}, growth mechanics \cite{Rodriguez1994}, etc. In a broader context, this modeling framework is mature, with computational approaches and engineering applications extensively investigated \cite{Simo1998,Simo2006,SouzaNeto2011}. The adoption of the multiplicative decomposition for viscoelasticity modeling can be traced back to the work of Sidoroff in 1974 \cite{Sidoroff1974}. Building upon this work, Reese and Govindjee systematically constructed a viscoelasticity model and thoroughly discussed the corresponding finite element modeling \cite{Reese1998,Reese1998b}. In particular, their algorithmic treatment of the constitutive integration exploited the techniques in computational plasticity \cite{Cuitino1992,Simo1992b,Simo1992d,Weber1990}. That viscoelasticity model, derived from thermomechanical principles, allows the use of general nonlinear potential to address the non-equilibrium aspects. Those attributes make it appealing, and thus the multiplicative viscoelasticity model has since then been generalized to modeling additional material behaviors, especially for soft materials. Representative applications include electroactive polymers \cite{Hong2011,Sharma2019}, anisotropic materials \cite{Latorre2015,Latorre2016,Liu2019b,Nguyen2007}, shape memory polymers \cite{Nguyen2008}, and hydrogels \cite{Mao2017}, to list a few.

Nevertheless, the concept of the intermediate configuration is not without shortcomings, and its validity has been questioned since its inception. A primary concern comes with the existence and uniqueness of the intermediate configuration \cite{Casey1992,Dafalias1987,Naghdi1990}. Moreover, the intermediate configuration can be arbitrarily rotated without affecting the multiplicative decomposition. Therefore, the theory is incomplete unless the description of the rotational parts of $\bm F^{\mathrm{e}}$ and $\bm F^{\mathrm{i}}$ are specified. Furthermore, the objectivity of the constitutive models essentially poses a rather involved constraint governing $\bm F^{\mathrm{e}}$ and $\bm F^{\mathrm{i}}$, which is sometimes known as the invariance requirements \cite{Casey1980,Green1971,Dafalias1998}. There are several works striving to resolve the aforementioned issues by proposing proper kinematic laws \cite{Boyce1989,Gurtin2005}. However, they are more or less ad hoc and can be difficult to be justified from physical experiments. In our opinion, those ad hoc laws should be placed together with the multiplicative decomposition as the primary modeling assumptions for specific models. In anisotropic cases, this is particularly critical as one relies on the specific assumption to define the evolution of the material symmetry groups on the intermediate configurations \cite{Nguyen2007,Latorre2015,Ciambella2021}. 

Almost parallel to the development of the theory based on the multiplicative decomposition, Green and Naghdi developed an alternative strategy for elastoplasticity based on an additive decomposition \cite{Green1965}. They introduced a symmetric tensor $\bm E^{\mathrm{p}}$ and referred to it as the plastic strain. In their theory, the difference $\bm E - \bm E^{\mathrm{p}}$ enters into the energy function as a tensorial variable. Recognizing the potential incompatibility issue, both authors emphasized that $\bm E - \bm E^{\mathrm{p}}$ shall not be interpreted as the elastic strain and thus avoided the terminology ``additive decomposition". Despite that, the interpretation of additive decomposition has been used by others since then \cite{Simo1985a,Xiao2006}. An extension of this concept was made by Miehe, who invoked the additive split of the Hencky strain as the point of departure \cite{Miehe1998,Miehe1998a,Miehe2000,Miehe2002}. Schr{\"o}der et al. considered anisotropic elastoplasticity within this framework \cite{Schroeder2002}. Papadopoulos and Lu generalized the additive plasticity theory by exploiting the Seth-Hill strain family \cite{Papadopoulos1998,Papadopoulos2001}. Computational aspects of this type of model have also been investigated \cite{Meng2002,Miehe2001b,Lewandowski2023}. Despite the accumulated studies in the additive theory, research efforts have been made primarily in the realm of elastoplasticity. As was discussed at the beginning of this article, the model of Simo \cite{Simo1987,Holzapfel1996} implicitly adopted the additive decomposition. We summarize representative inelasticity models in Table \ref{table:inelasticity-modeling} based on their kinematic assumptions. By noticing that the current additive viscoelasticity model is only geometrically nonlinear, it is necessary and possible to systematically extend it to the fully nonlinear regime and develop corresponding computation methods.

\subsection{Hyperelasticity of Hill's class and generalized strains}
In the phenomenological modeling of elastic materials, it is nowadays fairly common to design nonlinear functions written in terms of the principal stretches or invariants to characterize nonlinear hyperelastic behaviors \cite{Dal2021,Xiang2020}. Alternatively, in a different approach proposed by Hill \cite{Hill1979}, the structure of Hooke's law is maintained, and the nonlinearity is manifested through generalized strains. In \cite{Hill1968}, Hill proposed the concept of general strains by adopting a smooth monotone scale function of the principal stretches. The first instantiation of this concept is the Seth-Hill \cite{Hill1968, Hill1979, Seth1964} or Doyle-Ericksen strains \cite{Doyle1956}, which constitute a strain family parameterized by a real number. With this strain family, Hill's hyperelasticity framework involves the Saint Venant-Kirchhoff model \cite[pp.~250-251]{Holzapfel2000} as well as the Hencky elasticity model \cite{Neff2016,Xiao2002}. This strain family also serves as the backbone for the aforementioned additive plasticity \cite{Papadopoulos1998,Papadopoulos2001,Lewandowski2023,Miehe2002,Miehe1998,Menzel2006}. Nevertheless, an apparent drawback of the Seth-Hill strain family, except the Hencky strain, is that they are not coercive, meaning that the strains remain finite for extreme deformations \cite{Korobeynikov2019}. To rectify this issue, coercivity is supplemented as an additional requirement for the scale function of generalized strains. In 1991, Curnier and Rakotomanana proposed a two-integer-parameter strain family that incorporates the strains (e.g. the Pelzer, Mooney, Wall, and Rivlin strains) not belonging to the Seth-Hill family \cite{Curnier1991}. This strain family was later extended by Darigani and Naghdabadi by allowing the parameters to take real values \cite{Darijani2010b}. Around the end of the twentieth century, Ba\v{z}ant and Itskov each individually designed a single-parameter strain family with motivations from different fields \cite{Bazant1998,Itskov2004}. In 2006, Curnier and Zysset proposed a two-parameter strain family that is easy to compute and delivers good approximations of the Seth-Hill family \cite{Curnier2006}. Later in 2013, Darijani and Naghdabadi introduced a two-parameter strain family represented in the exponential form \cite{Darijani2013}, which can be beneficial for describing the strain hardening in very large deformations. 

Furthermore, Hill’s modeling strategy has evolved in another aspect. Traditionally, a single pair of work-conjugate stress and strain is utilized, meaning the strain energy involves a single quadratic term of strain. Considering the Ogden model on the other hand, three terms, or six parameters, are needed to achieve satisfactory fitting results for realistic materials \cite{Dal2021,Destrade2022,Holzapfel2000}. Notably, the invariants of the Seth-Hill strains resemble the energy form of the Ogden model \cite{Moerman2016}, leading to the introduction of additional quadratic terms of different strains into the potential function for material modeling purposes \cite{Beex2019,Darijani2010b}.  Over the past several decades, we have witnessed the progressive enrichment of Hill's hyperelasticity and generalized strains. This ever-increasing modeling strategies offers new possibilities in modeling both elastic and inelastic materials.

\subsection{Structure and content of the article}
The body of this work is organized as follows. In Section \ref{sec:theory}, the visco-hyperelasticity theory is methodically developed. The emphases are placed on the generalized strains, kinematic assumptions of Green-Naghdi type, constitutive relations based on the Helmholtz and Gibbs potentials. Section \ref{sec:numerics} offers comprehensive discussions on consistent linearization, constitutive integration, and modular implementation. Section \ref{sec:examples} showcases numerical examples that underscore the efficacy of our proposed model in characterizing viscoelastic material properties. We draw conclusions and discuss directions for future studies in Section \ref{sec:conclusion}.

\section{Theory}
\label{sec:theory}
In this section, we construct a nonlinear visco-hyperelasticity theory based on the kinematic assumptions of the Green-Naghdi type. The generalized strains are utilized to treat the material nonlinearity within the framework of Hill's hyperelasticity. Within this framework, the existing finite linear viscoelasticity is recovered when the strain is specialized to the Green-Lagrange strain. At the outset of our presentation, the following clarifications are made. 
\begin{enumerate}
\item The summation convention is \textit{not} adopted, and all summations are expressed explicitly with a summation sign.
\item The norms of a vector $\bm W$ and a rank-two tensor $\bm A$ are defined as 
\begin{align*}
\left \lvert \bm W \right\rvert := \left( \bm W \cdot \bm W \right)^{\frac12} \quad \mbox{ and } \quad \left \lvert \bm A \right\rvert := \left( \mathrm{tr}[\bm A \bm A^T] \right)^{\frac12} = \left( \bm A : \bm A \right)^{\frac12}.
\end{align*}
\item The rank-two (rank-four) identity and zero tensors are denoted by $\bm I$ ($\mathbb I$) and $\bm O$ ($\mathbb O$), respectively. 
\end{enumerate}

\subsection{Kinematics and generalized strains}
\label{sec:kinematics-and-gen-strains}
We adopt the initial configuration of the body $\Omega_{\bm X} \subset \mathbb R^3$ as the reference configuration and assume it to be stress-free. The motion of the body is described by a family of diffeomorphisms $\bm x = \bm \varphi(\bm X, t) =  \bm \varphi_t(\bm X)$ parameterized by the time coordinate $t$. For a given time coordinate, $\bm \varphi_t$ maps a material point initially at $\bm X \in \Omega_{\bm X}$ to the location $\bm x$. We denote the current configuration of the body as $\Omega_{\bm x}^{t} := \bm \varphi(\Omega_{\bm X}, t)$. The displacement is defined as $\bm U(\bm X, t) := \bm \varphi(\bm X, t) - \bm X$, and the velocity is defined as $\bm V := d\bm U/dt$, where $d(\cdot)/dt$ designates the total time derivative. The deformation gradient and the Jacobian are defined as
\begin{align*}
\bm F := \frac{\partial \bm \varphi_t}{\partial \bm X} \quad \mbox{and} \quad J:=\mathrm{det}(\bm F).
\end{align*}
To distinguish bulk and shear responses, the deformation gradient is multiplicatively decomposed into volumetric and volume-preserving parts as
\begin{align}
\label{eq:Flory_decomposition}
\bm F = \bm F_{\mathrm{vol}} \tilde{\bm F}, \qquad \bm F_{\mathrm{vol}} := J^{1/3}\bm I, \qquad \tilde{\bm F} := J^{-1/3}\bm F.
\end{align}
The right Cauchy-Green tensor $\bm{C} := \bm{F}^T\bm{F}$ and its unimodular counterpart $\tilde{\bm C} := \tilde{\bm F}^{T} \tilde{\bm F}$ enjoy the following spectral representations,
\begin{align}
\label{eq:C_spectral}
\bm C = \sum_{a=1}^{3} \lambda_a^2 \bm M_a \quad \mbox{and} \quad \tilde{\bm C} = \sum_{a=1}^{3} \tilde{\lambda}_a^2 \bm M_a,
\end{align}
in which $\lambda_a$ ($\tilde{\lambda}_a := J^{-1/3} \lambda_a$) are the (modified) principal stretches, and $\bm M_a := \bm{N}_a \otimes \bm{N}_a$ are the self-dyads of the principal referential directions $\bm N_a$. In particular, the set of self-dyads $\{\bm M_a\}$ forms the basis for the tensor space of commutators of $\bm C$. Recall that
\begin{align*}
\bm F \bm N_a = \lambda_a \bm n_a,
\end{align*}
where the unit-length spatial vectors $\bm n_a$ are known as the principal spatial directions. With the above, the spectral representations of $\bm{F}$ and $\tilde{\bm F}$ take the form
\begin{align*}
\bm{F} = \sum_{a=1}^3 \lambda_a \bm{n}_a \otimes \bm{N}_a \quad \mbox{and} \quad \tilde{\bm{F}} = \sum_{a=1}^3 \tilde{\lambda}_a \bm{n}_a \otimes \bm{N}_a.
\end{align*}

In this study, strain plays a pivotal role in characterizing nonlinear material behaviors, and we consider a general Lagrangian strain represented as 
\begin{align}
\label{eq:Hill_strain}
\bm{E} := \sum_{a=1}^{3} E (\lambda_a) \bm{M}_a,
\end{align}
with $E: (0,\infty) \rightarrow \mathbb R$ being a scale function. The generalized strain $\tilde{\bm E}$ associated with the isochoric part of the deformation is analogously defined as
\begin{align}
\label{eq:iso_Hill_strain}
\tilde{\bm{E}} := \sum_{a=1}^{3} E (\tilde{\lambda}_a) \bm{M}_a.
\end{align}
The Eulerian counterparts of the above two Lagrangian strains can be defined analogously by replacing $\bm M_a$ with $\bm{n}_a \otimes \bm{n}_a$. It can be verified that the generalized strains defined above are objective and isotropy \cite{Holzapfel2000}. It is demanded by Hill \cite{Hill1979} that, to make the above two definitions suitable, the scale function $E$ has to be twice differentiable\footnote{Since $E$ is a univariate function, we use $E'$ and $E''$ to denote its first and second derivatives.}, monotonically increasing (i.e., $E' > 0$), and satisfy
\begin{align}
\label{eq:E_property}
E(1) = 0 \quad \mbox{and} \quad E'(1) = 1.
\end{align}
The monotonicity is physically reasonable and ensures the inverse of the function $E$ exists. It is also termed as the regularity condition \cite{Curnier1991}. The condition \eqref{eq:E_property}$_1$ guarantees the measure vanishes for the undeformed state. The condition \eqref{eq:E_property}$_2$ renders the measure compatible with the infinitesimal strain and is known as the normality condition \cite{Ogden1997}. The well-known Seth-Hill strain family satisfies the above attributes. Nevertheless, it is found that, except for the Hencky strain \cite{Neff2016,Bruhns2020}, the rest members of the strain family exhibit erratic behaviors when the deformation ratio asymptotically approaches infinity or zero. Therefore, two additional conditions were supplemented, demanding the scale function $E$ to be coercive \cite{Korobeynikov2019}, i.e.,
\begin{align}
\label{eq:physical_property}
\lim_{\lambda \rightarrow \infty} E(\lambda) \rightarrow +\infty \quad \mbox{and} \quad \lim_{\lambda \rightarrow 0} E(\lambda) \rightarrow -\infty.
\end{align}
They require the strain to approach infinity when extreme deformation presents. Furthermore, the strain is tension-compression symmetric if the scale function $E$ satisfies
\begin{align}
\label{eq:tension-compression-symmetry}
E(\lambda^{-1}) = - E(\lambda).
\end{align}
While the symmetry condition \eqref{eq:tension-compression-symmetry} is mathematically elegant \cite{Bazant1998,Korobeynikov2019}, it is important to note that materials typically behave asymmetrically in tension and compression \cite{Du2020,Latorre2017,Moerman2016}. In this regard, certain two-parameter generalized strains enable one to flexibly control the model behaviors in the tension and compression branches separately, which can be desirable in constitutive modeling. In the following, we present representative examples of the generalized strains listed in chronological order.

\paragraph{\textbf{Seth-Hill strain family}}
The Seth-Hill, or Doyle-Ericksen, strain family \cite{Doyle1956,Seth1964} is defined as
\begin{align*}
E(\lambda) = E^{(m)}_{\mathrm{SH}}(\lambda) := 
\begin{cases}
\frac{1}{m}\left( \lambda^m -1\right),&m\neq 0,\\[1.0em]
\mathrm{ln}\lambda ,&m=0.
\end{cases}
\end{align*} 
It is the first instantiation of Hill's generalized strain definition and is parameterized by a single parameter. This strain family incorporates several classical strains, including the Almansi ($m=-2$), Hill ($m=-1$), Hencky or logarithmic ($m=0$), Biot ($m=1$), and Green-Lagrange ($m=2$) strains. Except for the Hencky strain, the Seth-Hill strain family does not satisfy the coercive property \eqref{eq:physical_property} (see Figure \ref{fig:strain_stretch}), which is viewed as its major deficiency and inspires the subsequent development of new generalized strains.

\paragraph{\textbf{Curnier-Rakotomanana strain family}}
\label{ex:Two-parameter_CR_strains}
In the early 1990s, Curnier and Rakotomanana \cite{Curnier1991} further generalized the Seth-Hill strain family as
\begin{align*}
E(\lambda) = E_{\mathrm{CR}}^{(m, n)}(\lambda) := \frac{1}{m+n} \left( \lambda^m -\lambda^{-n} \right) = \frac{m}{m+n}E_{\mathrm{SH}}^{(m)}(\lambda) + \frac{n}{m+n}E_{\mathrm{SH}}^{(-n)}(\lambda).
\end{align*}
The authors considered $m$ and $n$ as integer parameters, taking the value of $1$ or $2$ \cite{Curnier1991}. This essentially integrates the Pelzer ($m=n=1$), Mooney ($m=n=2$), Wall ($m=1$, $n=2$), and Rivlin ($m=2$, $n=1$) strains into a unified format. Since the newly incorporated strains are named after rubber elasticians, it is also termed as the `rubber' family. Later, Darijani and Naghdabadi extended the definition by allowing the two parameters to take non-negative real values with $mn>0$ \cite{Darijani2010}. The Hencky strain is supplemented as a natural consequence of L'H\^{o}pital's rule when $m=n=0$. It can be shown that the scale function satisfies the coercive property \eqref{eq:physical_property}. Also, the two parameters allow one to separately describe the tension and compression material behavior. Notably, studies show that this generalized strain family leads to satisfactory quality of fit for several soft materials \cite{Darijani2010b}.
 
\paragraph{\textbf{Ba\v{z}ant-Itskov strain family}}
In the late 1990s, Ba\v{z}ant introduced a family of generalized strains which can now be represented as
\begin{align*}
E(\lambda) = E_{\mathrm{BI}}^{(m)}(\lambda) := \frac{1}{2m}\left( \lambda^m -\lambda^{-m}\right) \quad \mbox{with} \quad m \neq 0.
\end{align*}
His original goal was to improve the asymptotic approximation of the Hencky strain over the Seth-Hill strains \cite{Bazant1998}. With a different motivation, Itskov derived the same form based on an analysis of an elastoplasticity problem \cite{Itskov2004}. Again, the Hencky strain can be supplemented to complete the definition when $m=0$. This family possesses the coercive and symmetry properties. Apparently, it is a sub-family of the previous Curnier-Rakotomanana strain family when the two parameters are identical. It is also worth mentioning that $\sum_{a=1}^{3}E_{\mathrm{BI}}^{(1)}(\lambda_a)\bm{n}_a \otimes \bm{N}_a$ leads to the two-point tensor introduced by B\"{o}ck and Holzapfel \cite{Boeck2004}. Despite several appealing mathematical attributes, the lack of objectivity limits its applicability in constitutive modeling \cite{Ciarletta2011}.

\paragraph{\textbf{Curnier-Zysset strain family}}
In 2006, Curnier and Zysset proposed a strain family that is easy to compute and approximate the Seth-Hill family \cite{Curnier2006}. It can be fitted into the format of the generalized strain \eqref{eq:Hill_strain} as
\begin{align*}
E(\lambda)  = E_{\mathrm{CZ}}^{(m)}(\lambda)  := \frac{2+m}{8} \lambda^2 - \frac{2-m}{8} \lambda^{-2} - \frac{m}{4}, \quad \mbox{with} \quad -2 \leq m \leq 2.
\end{align*}
In their study, the well-posedness of the hyperelastic model based on this strain family was analyzed by delimiting the stretch region where the rank-one convexity holds. It was found that this generalized strain family may substantially extend the region for the model to remain rank-one convex, in comparison with the Green-Lagrange strain. This manifests the theoretical significance of developing new concepts of strains. 

\paragraph{\textbf{Darijani-Naghdabadi exponential strain family}}
In 2013, Darijani and Naghdabadi \cite{Darijani2013} introduced a strain family as
\begin{align*}
E(\lambda) = E_{\mathrm{DN}}^{(m,n)}(\lambda) := \frac{1}{m+n}\left( e^{m(\lambda-1)} - e^{n(\lambda^{-1}-1)} \right), \quad \mbox{with} \quad m,n > 0.
\end{align*}
It can be viewed as an extension of the two-parameter Curnier-Rakotomanana strain family using an exponential function form. This choice can be propitious for characterizing the non-Gaussian region of rubber-like materials \cite{Dal2021} and the strain hardening in biological tissues \cite{Chagnon2015}.

\begin{figure}
\begin{center}
\begin{tabular}{c}
\includegraphics[angle=0, trim=120 160 120 120, clip=true, scale = 0.28]{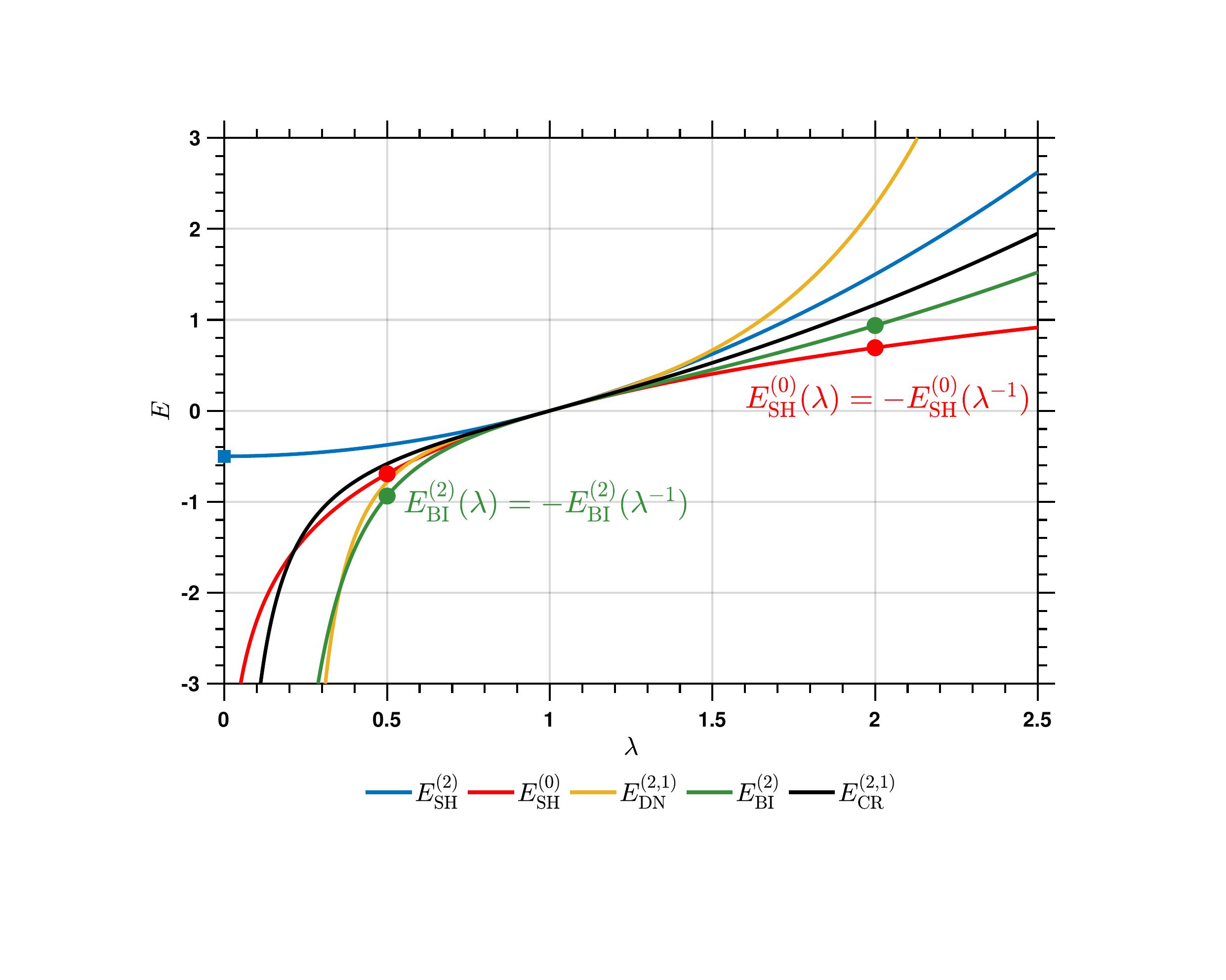}
\end{tabular}
\end{center}
\caption{Illustration of different scale functions. The Green-Lagrange strain (blue curve) gives a finite value when the stretch approaches zero. The Hencky scale function (red curve) and the Ba\v{z}ant-Itskov scale function (green curve) satisfy the symmetry property \eqref{eq:tension-compression-symmetry}, as is indicated by the red and green dots, respectively. }
\label{fig:strain_stretch}
\end{figure}

In our following discussion, we use $\bm E$ as a generic strain without specifying its type. When it becomes necessary, a subscript and superscript are used to indicate the specific strain. For example, $\bm E^{(2)}_{\mathrm{SH}}$ designates the strain of the Seth-Hill family with the parameter $m=2$ (a.k.a. the Green-Lagrange strain). In Figure \ref{fig:strain_stretch}, five representative scale functions are depicted, with the symmetry and coercive properties indicated by markers.

\begin{remark}
The definition \eqref{eq:Hill_strain} given here is based on the scale function $E$ of the principal stretch. The generalized strain can be equivalently represented in terms of the stretch tensor $\bm U := \sqrt{\bm C}$, which may be convenient in computation. For example, the Curnier-Zysset strain family was originally defined as
\begin{align*}
\bm E^{(m)}_{\mathrm{CZ}} := \frac{2+m}{8} \bm C  - \frac{2-m}{8} \bm C^{-1} - \frac{m}{4}\bm I,
\end{align*}
and their motivation for the design partly comes from avoiding spectral decompositions in practical calculations.
\end{remark}

\subsection{Projection tensors}
\label{sec:Projection_tensors}
In this section, we introduce three projection tensors that are necessary for constitutive modeling. Considering three distinct principal stretches $\left \{ \lambda_a \right \}$ for $a=1,2,3$, the derivative of $\bm{E}$ with respect to $\bm{C}$ leads to a rank-four tensor as
\begin{align}
\label{eq:Q}
\mathbb{Q} := 2\frac{\partial \bm{E}}{\partial \bm{C}} = \sum_{a=1}^{3} d_{a} \bm{M}_a \otimes \bm{M}_a + \sum_{a=1}^3 \sum_{b\neq a}^3 \vartheta_{ab}  \bm M_a \odot \bm M_b,
\end{align} 
in which
\begin{align}
\label{eq:d_theta}
d_a := 2 \frac{\partial E}{\partial (\lambda^2)}(\lambda_a) = \frac{1}{\lambda_a} E'(\lambda_a), \quad \vartheta_{ab} := 2\frac{ E(\lambda_a) - E(\lambda_b) }{\lambda_a^2 - \lambda_b^2},
\end{align}
and the symbol $\odot$ represents the following tensorial operation (see, e.g., \cite[p.~254]{Holzapfel2000})
\begin{align}
\label{eq:odot}
\bm M_a \odot \bm M_b := \frac12 \left( \bm N_a \otimes \bm N_b \otimes \bm N_a \otimes \bm N_b + \bm N_a \otimes \bm N_b \otimes \bm N_b \otimes \bm N_a  \right).
\end{align}
When two or three principal stretches are equal, the difference quotient in \eqref{eq:d_theta}$_2$ is replaced by
\begin{align*}
\lim\limits_{\lambda_b \to \lambda_a} \vartheta_{ab} = d_a,
\end{align*}
with the aid of L'H\^{o}pital's rule. Here we state and prove a useful result related to $\mathbb{Q}$.
\begin{lemma}
\label{lemma:Q_inv}
There exists a rank-four tensor $\mathbb Q^{-1}$ such that $\mathbb Q : \mathbb Q^{-1} = \mathbb Q^{-1} : \mathbb Q = \mathbb I$.
\end{lemma}
\begin{proof}
The monotonicity and smoothness of the scale function $E$ guarantee that its inverse $E^{-1}$ exists and is differentiable, due to the inverse function theorem. This indicates that the derivative $\partial \bm C/\partial \bm E$ is well-defined. We may explicitly construct $\mathbb Q^{-1}$ as
\begin{align*}
\mathbb Q^{-1} = \sum_{a=1}^{3} \frac{\lambda_a}{E'(\lambda_a)} \bm{M}_a \otimes \bm{M}_a + \sum_{a=1}^3 \sum_{b\neq a}^3 \frac{\lambda_a^2 - \lambda_b^2}{ 2(E(\lambda_a) - E(\lambda_b)) } \bm{M}_a \odot \bm{M}_b.
\end{align*}
When there are stretches taking identical values, the difference quotient term in the above formula needs to be replaced by $1/d_a$. One may conveniently verify the following identities,
\begin{align*}
& \left( \bm M_a \otimes \bm M_a \right) : \left( \bm M_b \otimes \bm M_b \right) = 
\begin{cases}
\mathbb O, & a \neq b, \\[0.5em]
\bm M_a \otimes \bm M_a, & a=b,
\end{cases} \displaybreak[2] \\
& \left( \bm M_a \odot \bm M_b \right) : \left( \bm M_c \odot \bm M_d \right) = 
\begin{cases}
\frac12 \bm M_a \odot \bm M_b, & a =c \neq b = d, \\[0.5em]
\frac12 \bm M_a \odot \bm M_b, & a = d \neq b = c, \\[0.5em]
\mathbb O, & \mbox{otherwise with } a \neq b,  c \neq d,
\end{cases} \displaybreak[2] \\
& \left( \bm M_a \otimes \bm M_a \right) : \left( \bm M_b \odot \bm M_c \right) = \mathbb O, \quad b \neq c.
\end{align*}
With the above results, direct calculation of the contraction gives
\begin{align*}
\mathbb Q : \mathbb Q^{-1} = \mathbb Q^{-1} : \mathbb Q = \sum_{a=1}^{3} \bm{M}_a \otimes \bm{M}_a + \sum_{a=1}^3 \sum_{b\neq a}^3 \bm{M}_a \odot \bm{M}_b = \mathbb I,
\end{align*}
which completes the proof.
\end{proof}

The second derivative of $\bm{E}$ with respect to $\bm{C}$ yields a rank-six tensor
\begin{align}
\label{eq:L}
\mathbb{L} :=& 2\frac{\partial \mathbb{Q}}{\partial \bm{C}} = 4\frac{\partial^2 \bm{E}}{\partial \bm{C} \partial \bm{C}} \nonumber \\
=& \sum^{3}_{a=1} f_{a} \bm{M}_a\otimes \bm{M}_a\otimes \bm{M}_a + \sum^{3}_{a=1} \sum^{3}_{b \neq a} \xi_{ab} \left( \mathbb{H}_{abb} + \mathbb{H}_{bab} + \mathbb{H}_{bba} \right) + \sum^{3}_{a=1} \sum^{3}_{b \neq a} \sum^{3}_{\substack{ c \neq a \\  c \neq b} } \eta \mathbb{H}_{abc},
\end{align}
with
\begin{align}
\label{eq:fa}
f_a :=& 4\frac{\partial^2 E}{\partial (\lambda^2) \partial (\lambda^2)}(\lambda_a) = -\frac{1}{\lambda_a^3}E'(\lambda_a) + \frac{1}{\lambda_a^2} E''(\lambda_a), \displaybreak[2] \\
\label{eq:xiab}
 \xi_{ab}:=& \frac{ \vartheta_{ab}- d_b }{2(\lambda_a^2 - \lambda_b^2)}, \displaybreak[2] \\
\label{eq:eta}
\eta :=& \sum_{a=1}^{3} \sum_{b \neq a}^{3} \sum_{\substack{ c \neq a \\  c \neq b} }^3 \frac{E(\lambda_a)}{2(\lambda_a^2-\lambda_b^2)(\lambda^2_a - \lambda^2_c)}, \displaybreak[2] \\
\mathbb H_{abc} :=& \bm N_a \otimes \bm N_b \otimes \bm N_c \otimes \bm N_a \otimes \bm N_b \otimes \bm N_c + \bm N_a \otimes \bm N_b \otimes \bm N_c \otimes \bm N_a \otimes \bm N_c \otimes \bm N_b \nonumber \displaybreak[2] \\
&+ \bm N_a \otimes \bm N_b \otimes \bm N_a \otimes \bm N_c \otimes \bm N_b \otimes \bm N_c + \bm N_a \otimes \bm N_b \otimes \bm N_a \otimes \bm N_c \otimes \bm N_c \otimes \bm N_b \nonumber \displaybreak[2] \\
& + \bm N_b \otimes \bm N_a \otimes \bm N_c \otimes \bm N_a \otimes \bm N_b \otimes \bm N_c + \bm N_b \otimes \bm N_a \otimes \bm N_c \otimes \bm N_a \otimes \bm N_c \otimes \bm N_b \nonumber \displaybreak[2] \\
\label{eq:H}
& + \bm N_b \otimes \bm N_a \otimes \bm N_a \otimes \bm N_c \otimes \bm N_b \otimes \bm N_c + \bm N_b \otimes \bm N_a \otimes \bm N_a \otimes \bm N_c \otimes \bm N_c \otimes \bm N_b.
\end{align}
When there exist identical stretches, the coefficients in \eqref{eq:xiab} and \eqref{eq:eta} are replaced by the following
\begin{align*}
\lim\limits_{\lambda_b \to \lambda_a} \xi_{ab} = \frac18 f_a, \quad \lim\limits_{\substack{\lambda_b \to \lambda_a \\ \lambda_c \neq \lambda_a}} \eta = \xi_{ca}, \quad \mbox{and} \quad \lim\limits_{\substack{\lambda_b \to \lambda_a \\ \lambda_c \to \lambda_a}} \eta = \frac{1}{8} f_{a}.
\end{align*}
The formulas of $\mathbb{Q}$ and $\mathbb{L}$ can be derived based on the principal axes method \cite{Chadwick1971}. In that method, a rate formulation is established and the components of the projection tensors are identified by equating coefficients \cite[pp.~258-259]{Holzapfel2000}. In this work, we use an alternative approach and derive the formulas by straightforwardly considering the derivatives of the principal axes with respect to the deformation tensor. Our derivations are detailed in \ref{ap:Q} and \ref{ap:L}, respectively. The first and second derivatives of $\tilde{\bm{E}}$ with respect to $\tilde{\bm{C}}$ are denoted as 
\begin{align}
\label{eq:tilde_Q_L}
\tilde{\mathbb{Q}} := 2\frac{\partial \tilde{\bm{E}}}{\partial \tilde{\bm{C}}} \quad \mbox{and} \quad \tilde{\mathbb{L}} := 4\frac{\partial^2 \tilde{\bm{E}}}{\partial \tilde{\bm{C}} \partial \tilde{\bm{C}}} = 2\frac{\partial \tilde{\mathbb{Q}}}{\partial \tilde{\bm{C}}}.
\end{align}
The projections $\tilde{\mathbb Q}$ and $\tilde{\mathbb L}$ can be explicitly represented by the formulas \eqref{eq:Q} and \eqref{eq:L} with the coefficients evaluated at $\tilde{\lambda}_a$. 

The partial derivative of $\tilde{\bm{C}}$ with respect to $\bm{C}$  reads
\begin{align}
\label{eq:deviatoric_projection}
\frac{\partial \tilde{\bm{C}}}{\partial \bm{C}} := J^{-\frac23} \mathbb{P}^T, \quad \mbox{with} \quad \mathbb{P} := \mathbb{I} - \frac{1}{3} \bm{C}^{-1} \otimes \bm{C}.
\end{align}
The projection tensor $\mathbb{P}$ furnishes the deviatoric behavior in the Lagrangian description \cite[p.~230]{Holzapfel2000}.

\subsection{Hyperelasticity of Hill's class}
\label{sec:Hyperelasticity}
The design of the strain energy function is the crux of material modeling. Different from phenomenological models based on invariants or principal stretches, the hyperelastic model of Hill's class maintains the quadratic functional form while describing nonlinear response by using a generalized strain. Based on this philosophy, the hyperelastic strain energy $\Psi(\bm C)$ can be represented as
\begin{align}
\label{eq:hill-conv-hyperelasticity}
\Psi(\bm C) = \frac12 \bm E : \mathbb{E} : \bm E,
\end{align}
where $\bm{E}$ is a generalized Lagrangian strain, and $\mathbb E$ is a rank-four stiffness tensor. The second Piola-Kirchhoff stress is given by
\begin{align*}
\bm S = 2\frac{\partial \Psi}{\partial \bm C} =  \bm T : \mathbb Q, \quad \mbox{with} \quad \bm T = \frac{\partial \Psi}{\partial \bm E} = \mathbb E : \bm E.
\end{align*}
If the material is isotropic, the stiffness tensor $\mathbb E$ takes the form
\begin{align*}
\mathbb E = 2 \mu \left( \mathbb I - \frac13 \bm I \otimes \bm I \right) + \kappa \bm I \otimes \bm I.
\end{align*}
Consequently, one has 
\begin{align}
\label{eq:hill-hyperelasticity-type-1}
\Psi(\bm C) = \mu \left\lvert \mathrm{dev}\bm E \right\rvert^2 + \frac{\kappa}{2} \left( \mathrm{tr} \bm E \right)^2 \quad \mbox{and} \quad \bm T = 2\mu \mathrm{dev} \bm E + \kappa (\mathrm{tr}\bm E) \bm I,
\end{align}
where $\mathrm{tr} \left( \cdot \right) := \left( \cdot \right) : \bm I$ and $\mathrm{dev} \left( \cdot \right) := \left( \cdot \right) - \mathrm{tr}\left( \cdot \right) \bm I / 3$. Due to the quadratic form of the energy, the conjugate pair of $\bm T$ and $\bm E$ maintains the Hookean stress-strain relationship. The second Piola-Kirchhoff stress is obtained from the stress-like tensor $\bm T$ by invoking the projection tensor $\mathbb Q$ previously defined in \eqref{eq:Q}, and this procedure is sometimes referred to as the geometric postprocessing \cite{Miehe2002}. With a specific choice of the strain $\bm E$, the above furnishes a specific hyperelasticity model of Hill's class.

It has also been recommended to modify the strain energy \eqref{eq:hill-hyperelasticity-type-1}$_1$ by replacing $\mathrm{tr}\bm E$ with $\ln(J)$ \cite{Xiao2011}, leading to the following potential,
\begin{align}
\label{eq:hill-hyperelasticity-type-2}
\Psi(\bm C) = \mu \left\lvert \tilde{\bm E} \right\rvert^2 + \Psi_{\mathrm{vol}}(J),
\end{align}
with $\Psi_{\mathrm{vol}}(J) = \kappa \ln(J)^2 / 2$. This modified model has been recognized to enjoy improved material stability \cite[p.~251]{Holzapfel2000}. If the strain in \eqref{eq:hill-hyperelasticity-type-2} is specialized to one of the Seth-Hill strain family, the first term in \eqref{eq:hill-hyperelasticity-type-2} can be expressed as
\begin{align}
\label{eq:hill-hyperelasticity-connection-ogden}
\mu \left\lvert \tilde{\bm E}_{\mathrm{SH}}^{(m)} \right\rvert^2 = \frac{\mu}{m^2} \sum_{a=1}^{3}\left( \tilde{\lambda}_a^{m}-1\right)^2,
\end{align}
which closely resembles the isochoric elastic response function of the Ogden model \cite[p.~237]{Holzapfel2000}.

\begin{figure}
\begin{center}
\includegraphics[angle=0, trim=10 210 40 50, clip=true, scale = 0.38]{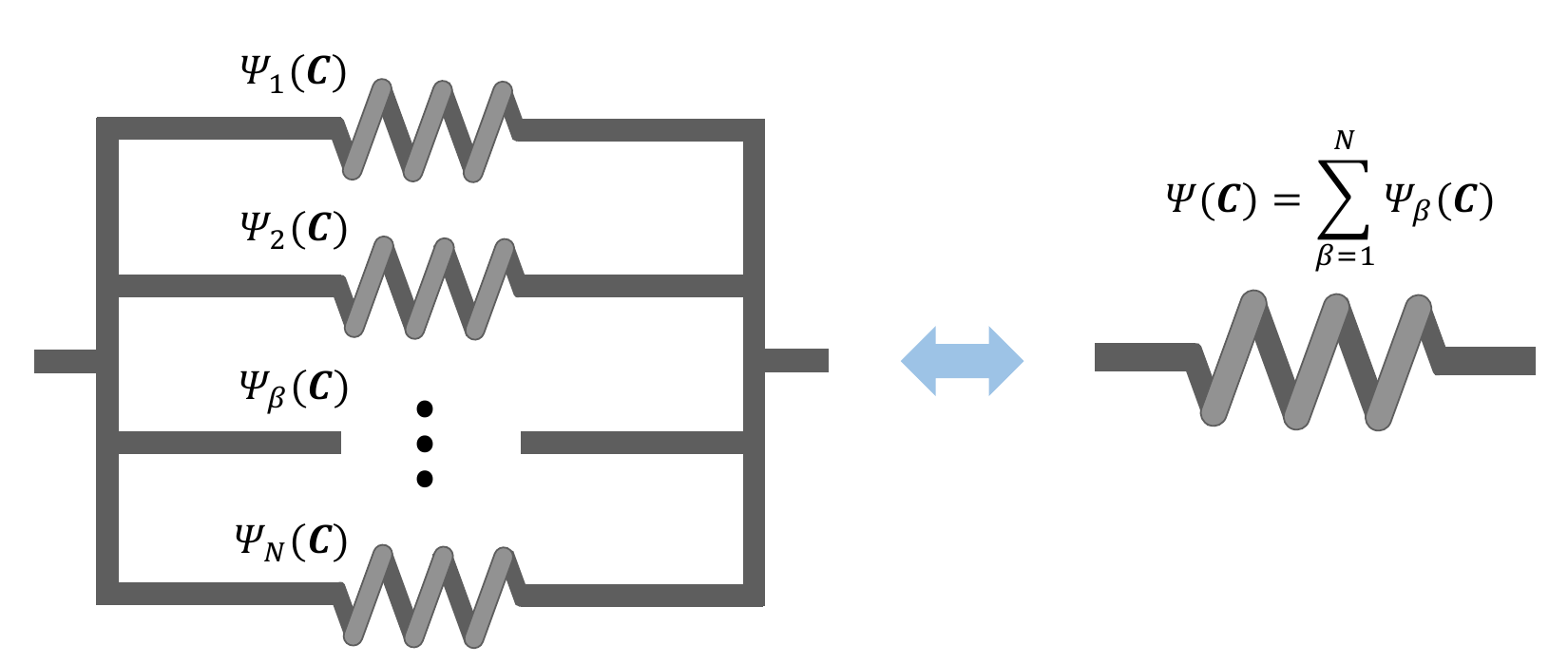}
\end{center}
\caption{The interpretation of the generalized hyperelastic strain energy function \eqref{eq:psi-def-fused-type-1} as a device of multiple springs of Hookean type in parallel. Each one of the Hookean spring is described by a single quadratic energy of a specific strain. Their combined mechanical behavior can be viewed as a single spring whose energy is the sum of quadratic strain energies of different strains.}
\label{fig:generalized_spring}
\end{figure}

We notice that there is only one term, or two material parameters, involved in defining the potential \eqref{eq:hill-hyperelasticity-connection-ogden}. Considering the Ogden model used for practical data fitting, three terms, or six parameters, are typically needed to provide high quality of fit to experimental data \cite{Dal2021,Ogden2004}. Thereby, there is a need to enrich the form of the aforementioned Hill's hyperelasticity model by introducing additional quadratic strain energies of different strains into the free energy. This strategy has been explored by using multiple Curnier-Rakotomanana strains \cite{Darijani2010b} or Seth-Hill strains \cite{Beex2019}, both of which demonstrated improved performance in the nonlinear data fitting. Here, we consider an energy represented as
\begin{align}
\label{eq:psi-def-fused-type-1}
\Psi(\bm C)=\sum_{\beta=1}^{N}\Psi_{\beta}(\bm C), \quad \mbox{with} \quad \Psi_{\beta}(\bm C) := \frac12 \bm{E}_{\beta} :\mathbb E_{\beta} : \bm{E}_{\beta}.
\end{align}                                                                                                                                                                                                                                                                                                      
In the above, the strain energy $\Psi$ is constructed as the sum of multiple quadratic energies whose associated strains $\bm E_{\beta}$ can be any member of any generalized strain family outlined in Section \ref{sec:kinematics-and-gen-strains}. The number of quadratic energy terms is $N$. An interpretation of this generalization is made in Figure \ref{fig:generalized_spring}. As a generalization of the modified strain energy \eqref{eq:hill-hyperelasticity-type-2}, we consider the potential as
\begin{align}
\label{eq:psi-def-fused-type-2}
\Psi(\bm C) = \sum_{\beta=1}^{N} \mu_{\beta} \left\lvert \tilde{\bm E}_{\beta} \right\rvert^2 + \Psi_{\mathrm{vol}}(J).
\end{align}
With \eqref{eq:hill-hyperelasticity-connection-ogden}, we may view the above strain energy formally similar to the $N$-term Ogden model in the sense that $N$ different nonlinear functions of stretch constitute the overall potential. Henceforth, we will refer to the model given by \eqref{eq:hill-conv-hyperelasticity} or \eqref{eq:hill-hyperelasticity-type-2} as the conventional Hill's hyperelasticity, and the elastic model given by \eqref{eq:psi-def-fused-type-1} or \eqref{eq:psi-def-fused-type-2} will be referred to as the generalized Hill's hyperelasticity with multiple terms.

\begin{figure}
\begin{center}
\begin{tabular}{cc}
\includegraphics[angle=0, trim=120 90 120 120, clip=true, scale = 0.2]{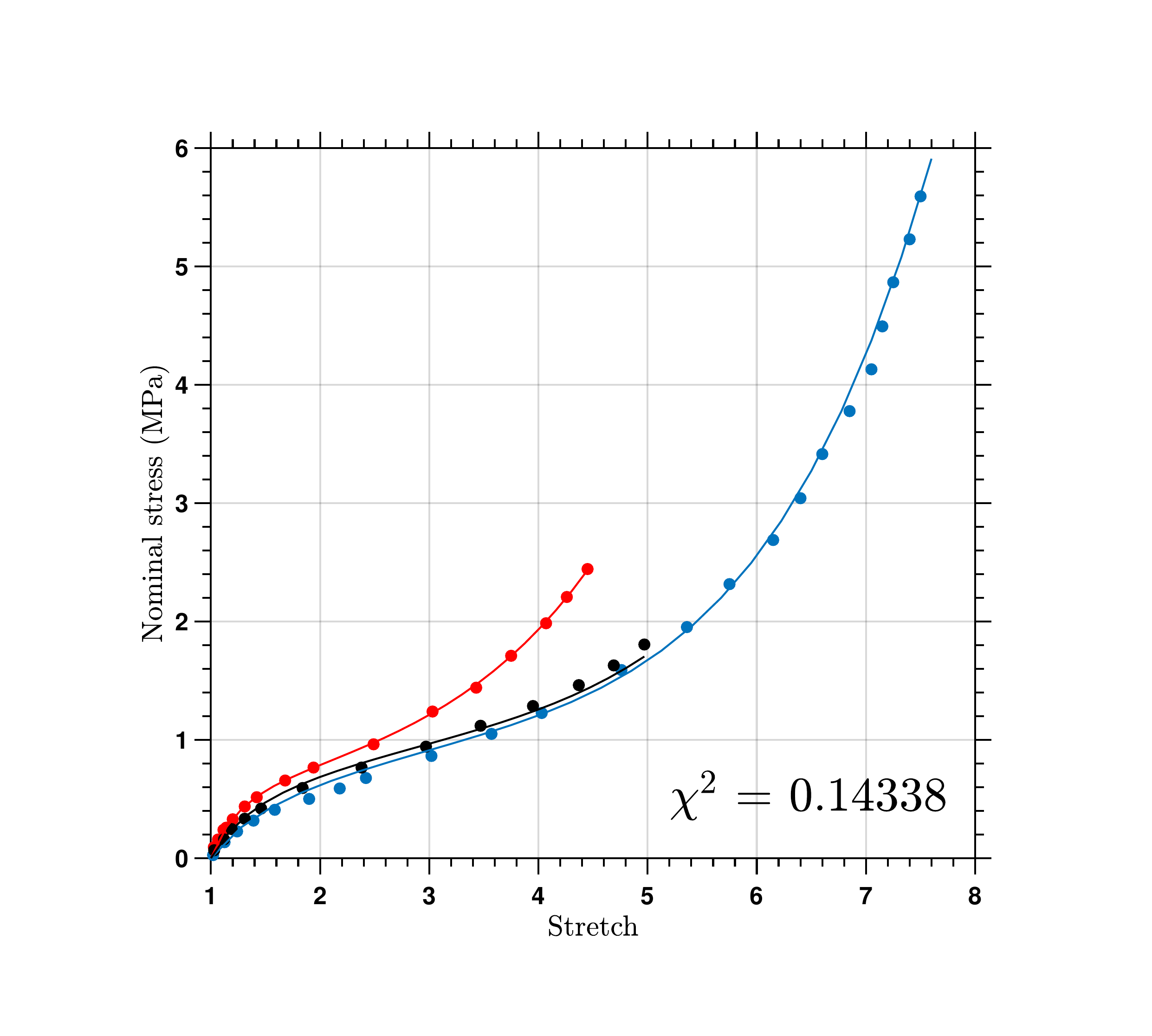} &
\includegraphics[angle=0, trim=120 90 120 120, clip=true, scale = 0.2]{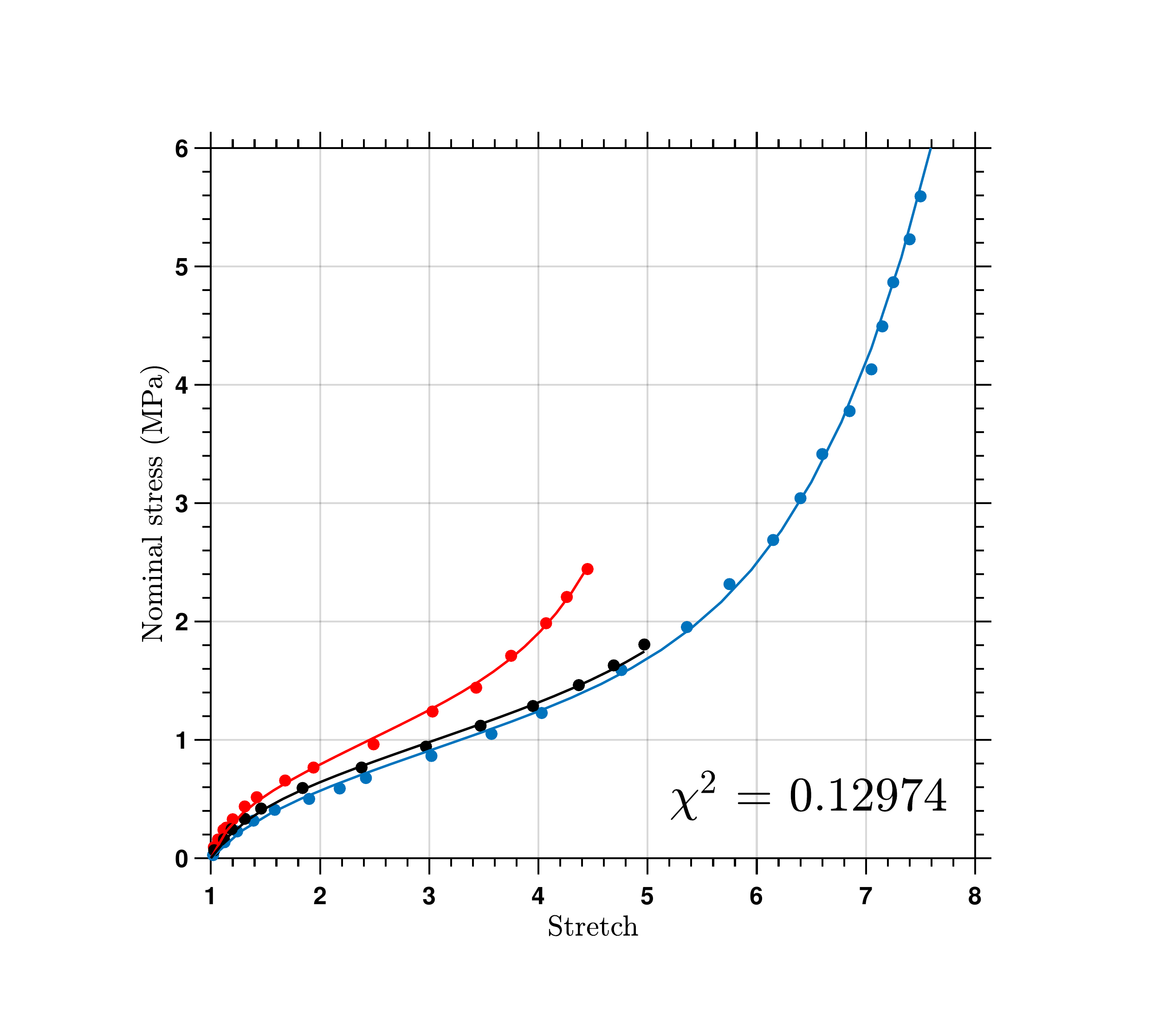}
\end{tabular}
\end{center}
\caption{Fitting of the Ogden model (left) and the model of Hill's class using two Curnier-Rakotomanana strains (right). The dots and lines represent the experimental data and model predictions, respectively. The blue, red, and black colors represent the uniaxial tensile, equi-biaxial tensile, and pure shear responses, respectively. The quality of fit metric $\chi^2$ is adopted from the work of Dal \cite[Eqn.~(179)]{Dal2021}. A smaller value of $\chi^2$ indicates a better model fit.}
\label{fig:curve_fitting}
\end{figure}

\begin{remark}
The efficacy of the proposed model is assessed through the simultaneous fitting of material parameters from a combined suite of uniaxial tensile, equi-biaxial tensile, and pure shear tests of vulcanized rubber, and the experimental data are adopted from the seminal work of Treloar \cite{Treloar1944}. The model we considered is based on the strain energy of the form \eqref{eq:psi-def-fused-type-1} with $N=2$, and the two strains are both taken from the Curnier-Rakotomanana family. The three-term Ogden model is considered for comparison purposes. Both models involve six parameters. The fitting results are illustrated in Figure \ref{fig:curve_fitting} and demonstrate that the newly proposed model exhibits superior quality of fit to the Ogden model. This indicates that the generalization of Hill's hyperelasticity using multiple terms indeed possesses the potential to improve the model capability in describing elastic behavior under large deformations.
\end{remark}

\subsection{Kinematic assumptions of the Green-Naghdi type}
\label{sec:Kinematic assumption}
In the original work of Green and Naghdi \cite{Green1965,Green1971,Naghdi1990}, a fundamental assumption is the existence of a symmetric Lagrangian tensor known as the plastic strain, and the difference between the total strain and the plastic strain enters into the stress response function. This strategy was later generalized with the invocation of Seth-Hill strains. Inspired by this approach and new strain families developed recently, we construct a nonlinear viscoelasticity theory based on similar assumptions. Our \textit{first} assumption is the existence of a rank-two Lagrangian tensor $\bm{\Gamma}$ $\in \mathrm{Sym}(3)_+$. It is a primitive variable characterizing the viscous deformation and is analogous to the deformation tensor $\bm{C}$ used for describing elastic deformation. As a symmetric tensor, the spectral decomposition of $\bm \Gamma$ can be performed, which leads to
\begin{align}
\label{eq:Gamma}
\bm{\Gamma} := \sum_{a=1}^{3} \Gamma^2_a \bar{\bm{M}}_a, \qquad \bar{\bm{M}}_a := \bar{\bm{N}}_a \otimes \bar{\bm{N}}_a,
\end{align}
where $\lbrace\Gamma^2_a\rbrace$ and $\lbrace \bar{\bm N}_a\rbrace$, for ${a=1,2,3}$, are the eigenvalues and eigenvectors of $\bm{\Gamma}$, respectively. We mention that $\bm \Gamma$ is not necessarily coaxial with $\bm C$, and $\lbrace \bar{\bm N}_a\rbrace$ is used to distinguish it from $\bm N_a$. With the decomposition \eqref{eq:Gamma}, we may introduce the corresponding ``viscous strain" $\bm E^{\mathrm{v}}\in \mathrm{Sym}(3)_+$ as
\begin{align}
\label{eq:Ev-def}
\bm E^{\mathrm{v}} := \sum_{a=1}^{3} E^{\mathrm v}(\Gamma_a) \bar{\bm{M}}_a.
\end{align}
We mention that $E^{\mathrm v}$ in \eqref{eq:Ev-def} is a generic scale function and does not necessarily coincide with the scale function used in the total strain $\bm E$ of the body. The \textit{second} kinematic assumption of this work is that $\bm E - \bm E^{\mathrm{v}}$ enters into the constitutive formulation. We further introduce the following tensors
\begin{align}
\label{eq:QQ_LL_v}
\mathbb{Q}^{\mathrm{v}}:= 2\frac{\partial \bm{E}^{\mathrm{v}}}{\partial \bm{\Gamma}} \quad \mathrm{and} \quad  \mathbb{L}^{\mathrm{v}} := 4 \frac{\partial^2 \bm{E}^{\mathrm{v}}}{\partial \bm{\Gamma} \partial \bm{\Gamma}} = 2 \frac{\partial \mathbb{Q}^{\mathrm{v}}}{\partial \bm{\Gamma}}, 
\end{align}
which are the viscous analogues of $\mathbb Q$ and $\mathbb L$. The explicit representations of $\mathbb Q^{\mathrm v}$  and $\mathbb L^{\mathrm{v}}$ follow the formula \eqref{eq:Q} and \eqref{eq:L} with the coefficients  evaluated at $\Gamma_a$. We immediately have the following result, whose proof follows that of Lemma \ref{lemma:Q_inv}. 
\begin{lemma}
\label{lemma:Q_v_inv}
There exists a rank-four tensor $\mathbb Q^{\mathrm v \: -1}$ such that $\mathbb Q^{\mathrm v} : \mathbb Q^{\mathrm v \: -1} = \mathbb Q^{\mathrm v \: -1} : \mathbb Q^{\mathrm v} = \mathbb I$.
\end{lemma}
The internal state variable $\bm \Gamma$ is introduced to describe the non-eqilibrium behavior of the material. Conceptually, it can be viewed as a viscous deformation-like tensor associated with the dashpot when using a spring-dashpot rheological model for interpretation (see Figure \ref{fig:standard_model}). For typical viscoelastic solids, their non-equilibrium behavior may involve multiple relaxation processes, meaning there can be multiple Maxwell elements arranged in parallel in the rheological model. It thus requires multiple independent internal state variables to adequately model the material behavior. In our discussion, we construct the constitutive theory based on a single relaxation process using $\bm \Gamma$, and the extension to $M$ relaxation processes can be made with a set of internal state variables $\{\bm \Gamma^{\alpha}\}_{\alpha=1}^{M}$. Without considering thermal effects, the \textit{equilibrium state} can be characterized kinematically by imposing that the rates of change of the internal state variables are identically zero. Here, we provide a rigorous definition for this notion.
\begin{definition}
For a viscoelastic material characterized by $M$ internal state variables $\{\bm \Gamma^{\alpha}\}_{\alpha=1}^{M}$, it is in the equilibrium state if
\begin{align}
\label{eq:def-equilibrium-state}
\frac{d}{dt}\bm \Gamma^{\alpha} = \bm O,
\end{align}
for $\alpha=1,\cdots, M$.
\end{definition}
\noindent We adopt the notation $|_{\mathrm{eq}}$ introduced in Part \rom{1} to indicate quantities evaluated at the equilibrium state. In an alternative way to characterize the equilibrium state, the force driving the viscous deformation is demanded to vanish in the equilibrium limit \cite[p.~345]{Simo2006}. We mention that the two conditions are not necessarily equivalent in the \textit{finite strain} setting, as we did encounter non-vanishing non-equilibrium stresses in certain models that cause pathological behavior previously (see Sec. 4.1 in Part \rom{1}). Nevertheless, we will show in the subsequent discussion that, with the kinematic assumption of the Green-Naghdi type, the non-equilibrium stress can be rigorously proved to be fully relaxed in the equilibrium state.

\begin{remark}
It is well-known that the difference $\bm E - \bm E^{\mathrm{v}}$ is not necessarily an elastic strain. Here, we use $\bm E^{\mathrm e}$ simply as a notation for this difference without providing additional interpretations. The normality condition \eqref{eq:E_property}$_2$ of the generalized strain guarantees that it can be approximated by an infinitesimal strain in the small strain limit. The second kinematic assumption, together with the generalized strain concept, naturally recovers the decomposition used in the infinitesimal strain inelasticity theory, that is, $\bm{\varepsilon} = \bm{\varepsilon}^{\mathrm{v}} + \bm{\varepsilon}^{\mathrm{e}}$ \cite{Simo2006}, in which $\bm{\varepsilon}$, $\bm{\varepsilon}^{\mathrm{e}}$ and $\bm{\varepsilon}^{\mathrm{v}}$ are the infinitesimal approximations of $\bm E$, $\bm E^{\mathrm{v}}$, and $\bm E^{\mathrm e}$, respectively. 
\end{remark}

\begin{remark}
When it becomes necessary, we use $\bm E^{\mathrm v}(\bm \Gamma^{\alpha})$ to clarify the dependency of a generalized strain on the deformation-like tensor (the $\alpha$-th internal state variable $\bm \Gamma^{\alpha}$ here). In practice, we leverage a robust and accurate algorithm  \cite{Scherzinger2008,Guan2023,PERIGEE_system} to perform spectral decompositions for $\bm C$ and $\{ \bm \Gamma^{\alpha} \}^{M}_{\alpha=1}$, which are then utilized to construct the generalized strains based on the corresponding scale functions.
\end{remark}

\begin{remark}
In a different approach, the kinematic assumption postulates the multiplicative decomposition $\bm{F} = \bm{F}^{\mathrm{e}} \bm{F}^{\mathrm{v}}$ based on the concept of an intermediate configuration \cite{Lee1969,Sidoroff1974,Reese1998}. One may establish the equivalence between the proposed theory and the multiplicative theory using specific strain or under specific deformation states. For example, if the internal state variable $\bm \Gamma$ is interpreted as $\bm C^{\mathrm{v}}:= \bm{F}^{\mathrm{v} \: T} \bm{F}^{\mathrm{v}}$, there is an equivalence of the elastic strain invariants using the Green-Lagrange strain. Further elaborations on these insights can be found in \ref{Connections between the two kinematic assumptions}, shedding lights on the connections between the two fundamentally distinct kinematic assumptions.
\end{remark}

\subsection{Constitutive theory based on the Helmholtz free energy}
In this section, we derive constitutive relations based on the Helmholtz free energy to provide a theory that works within the classical pure displacement formulation. Our discussion starts by considering a single relaxation process and using conventional Hill's hyperelasticity. Under the isothermal condition, the Helmholtz free energy $\Psi$ is postulated as a function of $\bm C$ and $\bm \Gamma$, that is, $\Psi = \Psi( \bm{C}, \bm{\Gamma})$. Due to the definitions of generalized strains, the potential $\Psi$ can be equivalently expressed as a function of  $\bm E = \bm E(\bm C)$ and $\bm E^{\mathrm{v}}= \bm E^{\mathrm{v}}(\bm \Gamma)$. A number of experiments suggest the existence of distinct hyperelastic and time-dependent contributions of the energy \cite{Bergstroem1998,Wang2018}, leading to the following form of the free energy 
\begin{align}
\label{eq:Helmholtz_eq_neq}
\Psi(\bm C, \bm \Gamma) = \Psi^{\infty}(\bm E) + \Upsilon(\bm E - \bm E^{\mathrm v}).
\end{align}
The superscript `$\infty$' represents the contribution arising from the material's equilibrium response. The potential $\Upsilon$ is associated with the dissipative behavior and is known as the configurational free energy \cite{Holzapfel2000,Liu2021b}. Sometimes, it is also referred to as the non-equilibrium part of the free energy \cite{Reese1998}, and we use the superscript `neq' to indicate stresses derived from $\Upsilon$. We mention that hyperelasticity is modeled following Hill's framework, and the energies $\Psi^{\infty}$ and $\Upsilon$ are, respectively, quadratic functions of $\bm E$ and $\bm E - \bm E^{\mathrm v}$. They may take the functional form of either \eqref{eq:hill-conv-hyperelasticity} or \eqref{eq:hill-hyperelasticity-type-2}, and we do not specify their explicit forms at this stage. After function composition, we may express the two potentials as $\Psi^{\infty}(\bm C)$ and $\Upsilon(\bm C, \bm \Gamma)$ without distinguishing the functional form. The significance of \eqref{eq:Helmholtz_eq_neq} is that the two deformation tensors $\bm C$ and $\bm \Gamma$ enter into the potential $\Upsilon$ through $\bm E(\bm C) - \bm E^{\mathrm v}(\bm \Gamma)$, exemplifying the second kinematic assumption of the Green-Naghdi type. The Clausius-Planck inequality reads
\begin{align}
\label{eq:Clausius_Plank_inequality}
\mathcal D := \bm S : \frac12 \frac{d\bm C}{dt} - \frac{d\Psi}{dt} \geq 0,
\end{align}
where $\mathcal D$ represents the internal dissipation. With the postulated potential form \eqref{eq:Helmholtz_eq_neq}, the inequality can be expanded as
\begin{align*}
\mathcal{D} = \left( \bm S - 2\frac{\partial \Psi}{\partial \bm C} \right) : \frac12 \frac{d\bm C}{dt} +\bm Q  : \frac12 \frac{d\bm \Gamma}{dt} \geq 0,
\end{align*}
with
\begin{align}
\label{eq:def_Q}
\bm Q := - 2\frac{\partial \Psi}{\partial \bm \Gamma}.
\end{align}
To ensure the satisfaction of the inequality for arbitrary kinematic processes, we make the following choices,
\begin{align}
\label{eq:constitutive_S_Q}
\bm S := 2 \frac{\partial \Psi}{\partial \bm C}, \quad \bm Q = \mathbb V : \left( \frac12 \frac{d\bm \Gamma}{dt} \right),
\end{align}
with $\mathbb V$ being a positive-definite rank-four viscosity tensor. It is additionally assumed that there exists $\mathbb V^{-1}$ such that the constitutive relation \eqref{eq:constitutive_S_Q}$_2$ can be equivalently expressed as
\begin{align}
\label{eq:evolution_equation}
\frac12 \frac{d\bm \Gamma}{dt} = \mathbb V^{-1} : \bm Q.
\end{align}
With the above constitutive relations, the internal dissipation becomes
\begin{align}
\label{eq:helmholtz-dissipation}
\mathcal D = \left( \frac12 \frac{d\bm \Gamma}{dt} \right) : \mathbb V : \left( \frac12 \frac{d\bm \Gamma}{dt} \right) = \bm Q : \mathbb V^{-1} : \bm Q,
\end{align}
which remains non-negative due to the positive-definiteness of the viscosity tensor $\mathbb V$. An illustration of this nonlinear viscoelasticity model is made through the spring-dashpot device in Figure \ref{fig:standard_model}.

\begin{figure}
\begin{center}
\includegraphics[angle=0, trim=30 130 200 0, clip=true, scale = 0.35]{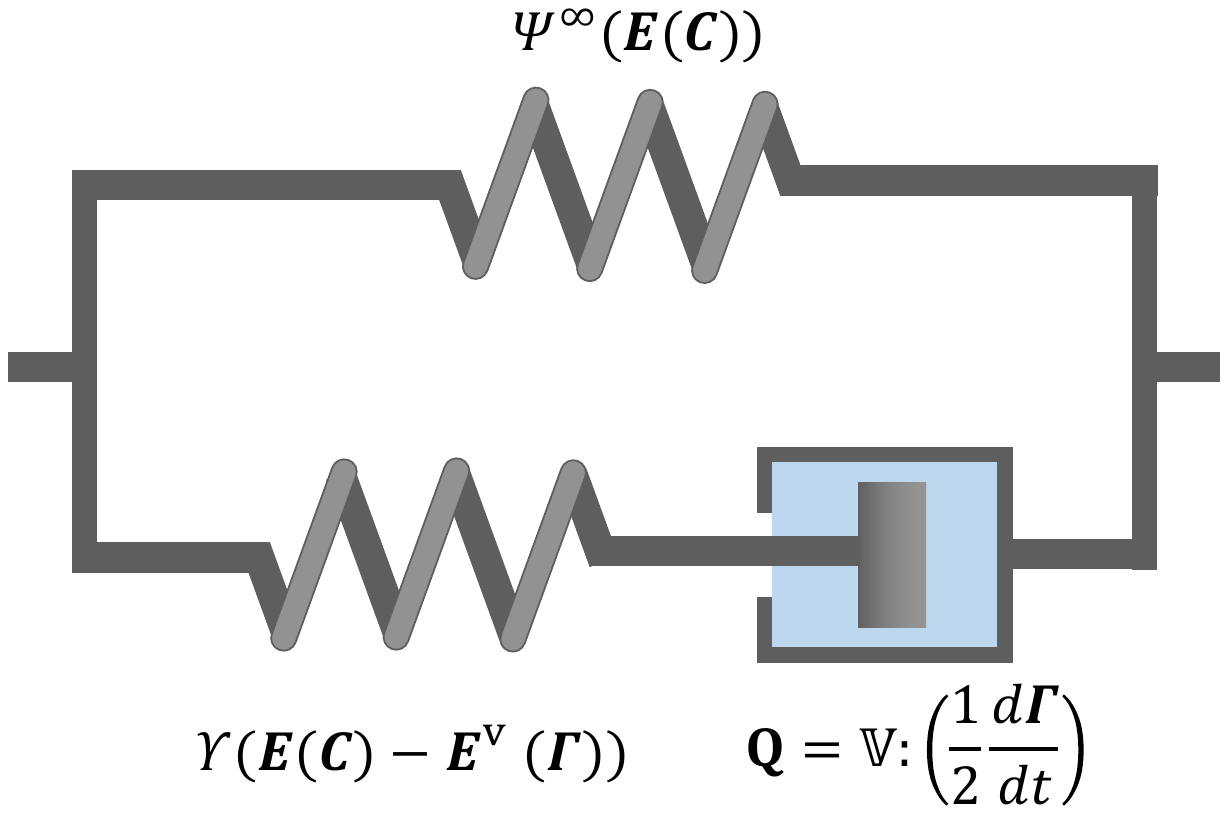}
\end{center}
\caption{An illustration of the theory based on the Helmholtz free energy using the spring-dashpot rheological model.}
\label{fig:standard_model}
\end{figure}

Leveraging the split of the Helmholtz free energy \eqref{eq:Helmholtz_eq_neq}, the second Piola-Kirchhoff stress $\bm{S}$ can be decomposed into equilibrium and non-equilibrium parts,
\begin{align}
\label{eq:S_eq_neq}
\bm S = \bm S^{\infty} + \bm S^{\mathrm{neq}},
\end{align}
with
\begin{align}
\label{eq:def_S_T_eq_neq}
\bm S^{\infty} := \bm T^{\infty} : \mathbb Q, \quad
\bm S^{\mathrm{neq}} := \bm T^{\mathrm{neq}} : \mathbb Q, \quad \bm{T}^{\infty}:= \frac{\partial \Psi^{\infty}}{\partial \bm{E}}, \quad \mbox{and} \quad \bm{T}^{\mathrm{neq}}:= \frac{\partial \Upsilon}{\partial \bm{E}} = -\frac{\partial \Upsilon}{\partial \bm{E}^{\mathrm{v}}}.
\end{align}
In the above, $\bm T^{\infty}$ and $\bm T^{\mathrm{neq}}$ are the work conjugate stresses to the generalized strain $\bm E$ with respect to the equilibrium and non-equilibrium potentials. We also recognize that the definition \eqref{eq:def_Q} of $\bm Q$ can be refined as
\begin{align}
\label{definition_Q}
\bm{Q} = -2\frac{\partial \Upsilon}{\partial \bm \Gamma} = -\frac{\partial \Upsilon}{\partial \bm E^{\mathrm{v}}} : 2\frac{\partial \bm E^{\mathrm{v}}}{\partial \bm \Gamma} = \bm{T}^{\mathrm{neq}}:\mathbb{Q}^{\mathrm{v}}.
\end{align}

With the constitutive relation \eqref{eq:constitutive_S_Q}$_2$, it can be readily shown that  $\bm Q |_{\mathrm{eq}} = \bm O$. A well-posed model needs to ensure $\bm{S}^{\mathrm{neq}} |_{\mathrm{eq}} = \bm O$ in the equilibrium state (see also Remark 1 of \cite{Reese1998}). An intuitive interpretation is that the force driving the dashpot vanishes in the equilibrium state (refer to Figure \ref{fig:standard_model}). In the above derived theory, $\bm{S}^{\mathrm{neq}} \neq \bm{Q}$ simply because the two projection tensors $\mathbb{Q}$ and $\mathbb{Q}^{\mathrm{v}}$ are not necessarily identical. Nevertheless, we still can show the vanishment of $\bm{S}^{\mathrm{neq}}$ in the following proposition.

\begin{proposition}
\label{eq:proposition_thermodynamic_limit_Hel}
In the equilibrium state, the non-equilibrium stress vanishes, i.e., 
\begin{align*}
\bm{S}^{\mathrm{neq}}\Big|_{\mathrm{eq}} = \bm{O}.
\end{align*} 
\end{proposition}

\begin{proof}
According to Lemma \ref{lemma:Q_inv}, there exists a rank-four tensor $\mathbb{Q}^{\mathrm{v} \: -1}$ such that $\mathbb{Q}^{\mathrm{v}} : \mathbb{Q}^{\mathrm{v} \: -1} = \mathbb{I} $. Therefore, one has the following relation
\begin{align*}
\bm{Q} : \mathbb{Q}^{\mathrm{v} \: -1} = \bm{T}^{\mathrm{neq}} : \mathbb{Q}^{\mathrm{v}}: \mathbb{Q}^{\mathrm{v} \: -1} = \bm{T}^{\mathrm{neq}}.
\end{align*}
It can be shown that $\bm Q|_{\mathrm{eq}} = \bm O$ due to \eqref{eq:constitutive_S_Q}$_2$ and the definition of the equilibrium state given in \eqref{eq:def-equilibrium-state}. Consequently one has $\bm{T}^{\mathrm{neq}}|_{\mathrm{eq}} = \bm O$, which further leads to $\bm{S}^{\mathrm{neq}}|_{\mathrm{eq}} = \bm{O}$.
\end{proof}

\begin{remark}
If the viscosity tensor $\mathbb{V}$ is isotropic, it can be represented in the following form,
\begin{align*}
\mathbb{V}= 2\eta_{\mathrm{d}} \left( \mathbb{I} - \frac{1}{3} \bm{I} \otimes \bm{I} \right) + \frac{2}{3} \eta_{\mathrm{v}} \bm{I} \otimes \bm{I},
\end{align*}
with $\eta_{\mathrm{d}}$ and $\eta_{\mathrm{v}}$ being the deviatoric and volumetric viscosities, respectively. If the two coefficients are identical (i.e., $\eta = \eta_{\mathrm{d}} = \eta_{\mathrm{v}}$), the constitutive relation \eqref{eq:constitutive_S_Q}$_2$ and the internal dissipation get further simplified as
\begin{align*}
\frac{d\bm \Gamma}{dt} = \frac{1}{\eta}\bm Q \quad \mbox{and} \quad \mathcal D = \frac{\eta}{2} \left\lvert \frac{d \bm \Gamma }{dt} \right\rvert^2 = \frac{1}{2\eta} \left\lvert \bm Q \right\rvert^2. 
\end{align*}
\end{remark}

\subsubsection{Generalization: Hill's hyperelasticity of multiple terms and multiple relaxation processes}
\label{sec:helmholtz-gen-hyperelasticity}
As was discussed in Section \ref{sec:Hyperelasticity}, the capability of conventional Hill's hyperelasticity in characterizing nonlinear material behavior can be improved by adopting multiple quadratic terms. Moreover, assume that there are $M$ independent internal state variables $\{ \bm{\Gamma}^{\alpha} \}_{\alpha=1}^{M}$. Each one of them is associated with a viscosity tensor $\mathbb{V}^{\alpha}$.  Based on this generalization strategy, the two potentials $\Psi^{\infty}$ and $\Upsilon$ can be conveniently modified into the sums of multiple quadratic terms,
\begin{align}
\label{eq:Hill_Psi_Upsilon}
\Psi^{\infty}(\bm C) = \sum_{\beta = 1}^{N}  \Psi^{\infty}_{\beta}(\bm C), \quad \Upsilon(\bm C, \bm \Gamma^1, \cdots, \bm \Gamma^{M}) = \sum_{\alpha = 1}^{M} \Upsilon^{\alpha}(\bm C, \bm \Gamma^{\alpha}),
\end{align}
with $\Psi^{\infty}_{\beta}$ and $\Upsilon^{\alpha}$ being the quadratic function of $\bm{E}^{\infty}_{\beta}(\bm{C})$ and $\bm{E}^{\alpha} (\bm{C}) - \bm{E}^{\mathrm{v} \: \alpha}(\bm{\Gamma}^{\alpha})$, respectively.  The stresses are correspondingly expressed as
\begin{align}
\label{eq:gen_S_inf_S_neq_Q_alpha}
\bm{S}^{\infty} = \sum^{N}_{\beta=1} \bm{T}^{\infty}_{\beta} : \mathbb{Q}^{\infty}_{\beta}, \quad \bm{S}^{\mathrm{neq}} = \sum^{M}_{\alpha=1} \bm{T}^{\alpha} : \mathbb{Q}^{\alpha}, \quad \bm{Q}^{\alpha} := -2 \frac{\partial \Upsilon}{\partial \bm \Gamma^{\alpha}} = -2 \frac{\partial \Upsilon^{\alpha}}{\partial \bm \Gamma^{\alpha}} = \bm{T}^{\alpha} : \mathbb{Q}^{\mathrm{v} \: \alpha},
\end{align}
with 
\begin{gather}
\label{eq:T_eq_neq_beta}
\bm{T}^{\infty}_{\beta} := \frac{\partial \Psi^{\infty}_{\beta} }{\partial \bm{E}^{\infty}_{\beta}}, \quad \bm{T}^{\alpha} := \frac{\partial \Upsilon^{\alpha} }{\partial \bm{E}^{\alpha}}  = - \frac{\partial \Upsilon^{\alpha} }{\partial \bm{E}^{\mathrm{v} \: \alpha}}, \displaybreak[2] \\
\label{eq:Q_inf_beta_alpha}
\mathbb{Q}^{\infty}_{\beta}:= 2\frac{\partial \bm{E}^{\infty}_{\beta}}{\partial \bm C}, \quad \mathbb{Q}^{\alpha}:= 2\frac{\partial \bm{E}^{\alpha}}{\partial \bm C}, \quad \mbox{and} \quad \mathbb{Q}^{\mathrm{v} \: \alpha} := 2\frac{\partial \bm{E}^{\mathrm{v} \: \alpha}}{\partial \bm{\Gamma}^{\alpha}}.
\end{gather}
For each one of the relaxation processes, the evolution equation reads
\begin{align}
\label{eq:Q_alpha_V_alpha_evo_eqn}
\bm Q^{\alpha} = \mathbb V^{\alpha} : \left( \frac12 \frac{d\bm \Gamma^{\alpha}}{dt} \right) \quad \mbox{or} \quad \frac12 \frac{d\bm \Gamma^{\alpha}}{dt} = (\mathbb V^{\alpha})^{-1} : \bm Q^{\alpha}.
\end{align}
The internal dissipation becomes
\begin{align*}
\mathcal D = \sum_{\alpha=1}^{M}\left( \frac12 \frac{d\bm \Gamma^{\alpha}}{dt} \right) : \mathbb V^{\alpha} : \left( \frac12 \frac{d\bm \Gamma^{\alpha}}{dt} \right)  = \sum_{\alpha = 1}^{M} \bm Q^{\alpha} : (\mathbb V^{\alpha})^{-1} : \bm Q^{\alpha},
\end{align*}
involving the dissipation from the $M$ different processes. For the generalized case, we show that the property given by Proposition \ref{eq:proposition_thermodynamic_limit_Hel} still holds.
\begin{proposition}
\label{eq:proposition_thermodynamic_limit_Hel_gen}
For the generalized model with the potential given by \eqref{eq:Hill_Psi_Upsilon}, the non-equilibrium stress vanishes in the equilibrium state, i.e., 
\begin{align*}
\bm{S}^{\mathrm{neq}}\Big|_{\mathrm{eq}} = \bm{O}.
\end{align*} 
\end{proposition}
\begin{proof}
Due to the definition of the equilibrium state, we have $\bm Q^{\alpha}|_{\mathrm{eq}} = \bm O$. According to Lemma \ref{lemma:Q_inv}, one has the following relation,
\begin{align*}
\bm{Q}^{\alpha} : \mathbb{Q}^{\mathrm{v} \: \alpha \: -1} = \bm{T}^{\alpha} : \mathbb{Q}^{\mathrm{v} \: \alpha}: \mathbb{Q}^{\mathrm{v} \: \alpha \: -1} = \bm{T}^{\alpha},
\end{align*}
which implies $\bm T^{\alpha}|_{\mathrm{eq}} = \bm O$. Following the expression for the non-equilibrium stress, it can be shown that
\begin{align*}
\bm{S}^{\mathrm{neq}}|_{\mathrm{eq}} = \sum^{M}_{\alpha=1} \bm{T}^{\alpha}|_{\mathrm{eq}} : \mathbb{Q}^{\alpha} = \bm O,
\end{align*}
which completes the proof.
\end{proof}
\begin{remark}
It remains an interesting topic to further generalize the theory by allowing multiple quadratic terms in the definition of $\Upsilon^{\alpha}$. However, we found that the result of Proposition \ref{eq:proposition_thermodynamic_limit_Hel_gen} will not hold with that generalization.
\end{remark}

\subsection{Constitutive theory based on the Gibbs free energy}
\label{sec:Gibbs_constitutive}
There have been experimental evidences suggesting that the shear viscoelastic effects are more significant than the bulk viscoelasticity for most materials \cite[Chapter 18]{Ferry1980}. Following those observations, we consider the decomposition of the Helmholtz free energy into the volumetric and isochoric parts,
\begin{align}
\label{eq:vol_iso-Helmholtz-free-energy}
\Psi(\bm C, \bm \Gamma) = \Psi_{\mathrm{vol}}^{\infty}(J) + \Psi_{\mathrm{iso}}(\tilde{\bm{E}}, \bm{E}^{\mathrm{v}}).
\end{align}
based on the decomposition of the Flory type \eqref{eq:Flory_decomposition}. Following the idea of energy split introduced in \eqref{eq:Helmholtz_eq_neq}, the isochoric part of the energy $\Psi_{\mathrm{iso}}$ gets further decomposed into the equilibrium and non-equilibrium parts,
\begin{align}
\label{eq:decomposition_of_Psi_iso}
\Psi_{\mathrm{iso}}(\tilde{\bm{E}}, \bm{E}^{\mathrm{v}}) = \Psi^{\infty}_{\mathrm{iso}}(\tilde{\bm{E}}) + \Upsilon(\tilde{\bm{E}} - \bm{E}^{\mathrm{v}}).
\end{align}
For the Helmholtz free energy, an issue is the singular behavior as the model approaches the incompressible limit. To circumvent this issue, a Legendre transformation can be performed on the volumetric part of the energy $\Psi^{\infty}_{\mathrm{vol}}$,
\begin{align}
\label{eq:Legendre_transformation}
G^{\infty}_{\mathrm{vol}}(P) := \inf_{J} \left \{ \Psi^{\infty}_{\mathrm{vol}}(J) + PJ \right \},
\end{align}
where $P$ is the thermodynamic pressure defined on the Lagrangian configuration. The resulting Gibbs free energy is defined as
\begin{align}
\label{eq:Gibbs}
G(\tilde{\bm C}, P, \bm \Gamma) :=  G^{\infty}_{\mathrm{vol}}(P) +  G_{\mathrm{iso}}(\tilde{\bm{E}}, \bm{E}^{\mathrm{v}})  = G^{\infty}_{\mathrm{vol}}(P) + G^{\infty}_{\mathrm{iso}}(\tilde{\bm{E}}) + \Upsilon(\tilde{\bm{E}}, \bm{E}^{\mathrm{v}}),
\end{align}
with
\begin{align}
\label{eq:G_iso_inf_Upsilon}
G^{\infty}_{\mathrm{iso}}(\tilde{\bm{E}}) = \Psi^{\infty}_{\mathrm{iso}}(\tilde{\bm{E}}) = \mu^{\infty} \left \lvert \tilde{\bm{E}} \right \rvert^2 \quad \mbox{and} \quad
\Upsilon(\tilde{\bm{E}}- \bm{E}^{\mathrm{v}}) = \mu^{\mathrm{neq}} \left \lvert \tilde{\bm{E}} - \bm{E}^{\mathrm{v}} \right \rvert^2.
\end{align}
In the above, we replaced the notation $\Psi^{\infty}_{\mathrm{iso}}$ by $G^{\infty}_{\mathrm{iso}}$ to indicate that it constitutes a part of the Gibbs free energy. The Legendre transformation does not affect the isochoric part, meaning $G^{\infty}_{\mathrm{iso}}$ and $\Upsilon$ are directly inherited from the Helmholtz free energy. Also, the design of the two potentials here follows the modified quadratic form \eqref{eq:psi-def-fused-type-2}, with $\mu^{\infty}$ ($\mu^{\mathrm{neq}}$) representing the shear modulus for the (non-)equilibrium part. For notational simplicity, we momentarily consider a single quadratic term for the elastic model here. Due to the Legendre transformation \eqref{eq:Legendre_transformation}, the pressure $P$ enters into the theory as a primitive variable, and the subsequent formulation becomes of the saddle-point nature, which is well-behaved in both compressible and incompressible regimes \cite{Liu2018,Liu2021b} (see Fig. 1 in Part \rom{1} for an interpretation). In fact, one may show that the Gibbs free energy \eqref{eq:Gibbs} leads to a generalization of the Herrmann variational formulation to the finite strain regimes \cite{Herrmann1965,Reissner1984,Shariff1997,Liu2019a}. With the Gibbs free energy \eqref{eq:Gibbs}, the Clausius-Plank inequality is modified to
\begin{align*}
\mathcal{D} &= \bm{S} : \frac12 \frac{d\bm{C}}{dt} - \frac{dG}{dt} + \frac{dP}{dt}J + P\frac{dJ}{dt} \nonumber \displaybreak[2] \\
&= \left( \bm S - 2 \frac{\partial G^{\infty}_{\mathrm{iso}}}{\partial \tilde{\bm C}} : \frac{\partial \tilde{\bm C}}{\partial \bm C} - 2 \frac{\partial \Upsilon}{\partial \tilde{\bm C}} : \frac{\partial \tilde{\bm C}}{\partial \bm C} + 2 P \frac{\partial J}{\partial \bm C} \right) : \frac12 \frac{d\bm C}{dt} +  \left( J - \frac{d{G^{\infty}_{\mathrm{vol}}}}{d P} \right) \frac{dP}{dt} - 2 \frac{\partial \Upsilon}{\partial \bm \Gamma} : \frac12 \frac{d \bm \Gamma}{dt} \nonumber \displaybreak[2] \\
&= \left( \bm S - J^{-\frac23} \mathbb P : \left( \tilde{\bm S}^{\infty}_{\mathrm{iso}} + \tilde{\bm S}^{\mathrm{neq}}_{\mathrm{iso}} \right) +  P J \bm C^{-1} \right) : \frac12 \frac{d\bm C}{dt} +  \left( J - \frac{d{G^{\infty}_{\mathrm{vol}}}}{d P} \right) \frac{dP}{dt} + \bm Q : \frac12 \frac{d \bm \Gamma}{dt}.
\end{align*}
In the last equality, we invoked the definition of $\mathbb P$ given in \eqref{eq:deviatoric_projection} and introduced three stress-like tensors as
\begin{align}
\label{eq:tilde_stress}
\tilde{\bm{S}}^{\infty}_{\mathrm{iso}} := 2\frac{\partial G_{\mathrm{iso}}^{\infty}}{\partial \tilde{\bm{C}}},\quad \tilde{\bm{S}}_{\mathrm{iso}}^{\mathrm{neq}} := 2\frac{\partial \Upsilon}{\partial \tilde{\bm{C}}}, \quad \mbox{and} \quad \bm{Q} := - 2\frac{\partial \Upsilon}{\partial \bm{\Gamma}}.
\end{align}
Based on the above relation for the internal dissipation $\mathcal D$, the following choices are made,
\begin{gather}
\label{eq:Gibbs-constitutive-rho-S-Q}
\rho = \rho_0 \left( \frac{d{G^{\infty}_{\mathrm{vol}}}}{d P} \right)^{-1} \circ \bm \varphi_{t}^{-1}, \quad
\bm S = \bm S_{\mathrm{vol}} + \bm S_{\mathrm{iso}} = \bm S_{\mathrm{vol}} + \bm S_{\mathrm{iso}}^{\infty} + \bm S_{\mathrm{iso}}^{\mathrm{neq}}, \quad \bm Q = \mathbb V : \left( \frac12 \frac{d\bm \Gamma}{dt} \right),  \displaybreak[2] \\
\label{eq:Gibbs-constitutive-S}
\bm S_{\mathrm{vol}} = -J P \bm C^{-1}, \quad
\bm S_{\mathrm{iso}}^{\infty} = J^{-\frac{2}{3}} \mathbb{P} : \tilde{\bm{S}}^{\infty}_{\mathrm{iso}}, \quad
\bm S_{\mathrm{iso}}^{\mathrm{neq}} = J^{-\frac{2}{3}} \mathbb{P} : \tilde{\bm{S}}^{\mathrm{neq}}_{\mathrm{iso}},
\end{gather}
in which $\mathbb V$ is the rank-four viscosity tensor that is assumed to be positive-definite. Noticing that $\rho$ is typically defined on the current configuration, we invoke $\bm \varphi_t^{-1}$ in \eqref{eq:Gibbs-constitutive-rho-S-Q}$_1$ to make a proper definition of it. With the above constitutive relations, the internal dissipation becomes
\begin{align*}
\mathcal{D} = \left( \frac12 \frac{d\bm \Gamma}{dt} \right) : \mathbb V : \left( \frac12 \frac{d\bm \Gamma}{dt} \right) = \bm Q : \mathbb V^{-1} : \bm Q
\end{align*}
and remains non-negative for arbitrary kinematic processes. An intuitive interpretation of the isochoric part of this model is made through the spring-dashpot device in Figure \ref{fig:standard_model_Gibbs}. Since the potentials are expressed in terms of $\tilde{\bm E}$, the stress-like tensors \eqref{eq:tilde_stress} are defined in the following refined manner,
\begin{align}
\label{eq:refined-def-stress-like-tensors}
\tilde{\bm{S}}^{\infty}_{\mathrm{iso}} = \tilde{\bm T}^{\infty} : \tilde{\mathbb Q}, \quad \tilde{\bm{S}}_{\mathrm{iso}}^{\mathrm{neq}} = \tilde{\bm{T}}^{\mathrm{neq}} : \tilde{\mathbb{Q}}, \quad \bm{Q} = \tilde{\bm{T}}^{\mathrm{neq}} : \mathbb{Q}^{\mathrm{v}},
\end{align}
with
\begin{align}
\label{eq:tilde_T_equilibrated}
\tilde{\bm{T}}^{\infty} := \frac{\partial G^{\infty}_{\mathrm{iso}}}{\partial \tilde{\bm{E}}}=2 \mu^{\infty} \tilde{\bm{E}} \quad \mbox{and} \quad \tilde{\bm{T}}^{\mathrm{neq}} := \frac{\partial \Upsilon}{\partial \tilde{\bm{E}}} = 2 \mu^{\mathrm{neq}} (\tilde{\bm{E}} - \bm{E}^{\mathrm{v}}).
\end{align}

\begin{figure}
\begin{center}
\includegraphics[angle=0, trim=30 120 200 0, clip=true, scale = 0.35]{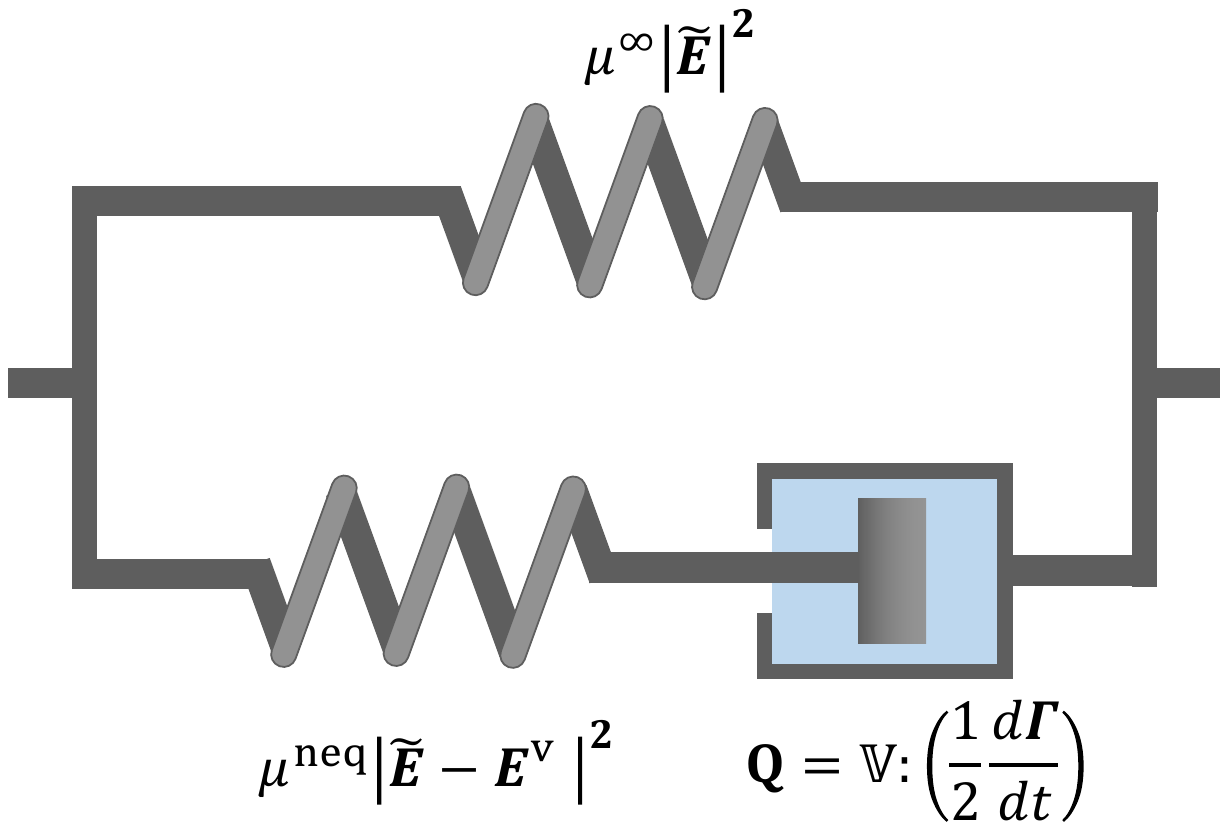}
\end{center}
\caption{An illustration of the theory based on the Gibbs free energy using the spring-dashpot rheological model.}
\label{fig:standard_model_Gibbs}
\end{figure}

Similar to the result of Proposition \ref{eq:proposition_thermodynamic_limit_Hel}, we prove the relaxation property of the stress associated with the potential $\Upsilon$ in the equilibrium state. This property is not obvious because $\bm{S}_{\mathrm{iso}}^{\mathrm{neq}}$ given in \eqref{eq:Gibbs-constitutive-S}$_3$ and $\bm Q$ given in \eqref{eq:Gibbs-constitutive-rho-S-Q}$_3$ are not necessarily identical.
\begin{proposition}
\label{pro:equilibrium}
In the equilibrium state, the non-eqilibrium stress $\bm S_{\mathrm{iso}}^{\mathrm{neq}}$ vanishes, that is,
\begin{align*}
\bm{S}_{\mathrm{iso}}^{\mathrm{neq}} \Big|_{\mathrm{eq}} = \bm{O}.
\end{align*} 
\end{proposition}
\begin{proof}
Due to Lemma \ref{lemma:Q_v_inv}, multiplying $\mathbb{Q}^{\mathrm{v} \: -1}$ at both sides of \eqref{eq:refined-def-stress-like-tensors}$_3$ leads to
\begin{align*}
\bm{Q} : \mathbb{Q}^{\mathrm{v} \: -1} = \tilde{\bm{T}}^{\mathrm{neq}} : \mathbb{Q}^{\mathrm{v}} : \mathbb{Q}^{\mathrm{v} \: -1} = \tilde{\bm{T}}^{\mathrm{neq}}.
\end{align*}
It is known that $\bm Q|_{\mathrm{eq}} = \bm O$. This implies $\tilde{\bm{T}}^{\mathrm{neq}}|_{\mathrm{eq}} = \bm{O}$, which leads to $\tilde{\bm{S}}_{\mathrm{iso}}^{\mathrm{neq}}|_{\mathrm{eq}} =\bm{S}_{\mathrm{iso}}^{\mathrm{neq}}|_{\mathrm{eq}} = \bm{O}$.
\end{proof}

\begin{remark}
\label{remark:configurational-free-energy}
If the strains are taken as the Green-Lagrange strain, we have $\tilde{\mathbb Q} = \mathbb Q^{\mathrm{v} } = \mathbb I$. Consequently, one has 
\begin{align*}
\tilde{\bm S}^{\mathrm{neq}}_{\mathrm{iso}} = \bm{Q} = \mu^{\mathrm{neq}} \left( \tilde{\bm {C}}  - \bm{\Gamma} \right),
\end{align*}
and the configurational free energy $\Upsilon$ is identical to the energy \eqref{eq:intro-Upsilon-HSSK} given at the beginning of this article. If the viscosity tensor $\mathbb V$ is isotropic, the derived constitutive model recovers the finite linear viscoelasticity model \cite{Simo1987,Holzapfel1996} discussed in Part \rom{1}. In this regard, we may view the model proposed here as the nonlinear generalization of the Holzapfel-Simo-Saint Venant-Kirchhoff model.
\end{remark}

\begin{remark}
In this work, the evolution equation is obtained by combining \eqref{eq:Gibbs-constitutive-rho-S-Q}$_3$ and \eqref{eq:refined-def-stress-like-tensors}$_3$, which results in
\begin{align*}
\bm Q = \mathbb V : \left( \frac12 \frac{d\bm \Gamma}{dt} \right) = 2\mu^{\mathrm{neq}} \left( \tilde{\bm E} - \bm E^{\mathrm v} \right) : \mathbb Q^{\mathrm v}.
\end{align*}
Noticing that $\bm E^{\mathrm v}$ generally depends on $\bm \Gamma$ in a nonlinear manner, the above is a nonlinear evolution equation for $\bm \Gamma$. In our previous study, we name it as a strain-driven format because the primitive unknown $\bm \Gamma$ is deformation-like \cite{Liu2023a}. If both $\tilde{\bm E}$ and $\bm E^{\mathrm v}$ are chosen as the Green-Lagrange strain and $\mathbb V = 2\eta \mathbb I$, we have the simplified evolution equation as 
\begin{align*}
\bm Q = \eta : \frac{d\bm \Gamma}{dt} = \mu^{\mathrm{neq}} \left( \tilde{\bm C} - \bm \Gamma \right).
\end{align*}
Taking time derivative at both sides of the above leads to
\begin{align*}
\frac{d\bm Q}{dt} + \frac{\mu^{\mathrm{neq}}}{\eta} \bm Q = \mu^{\mathrm{neq}} \frac{d\tilde{\bm C}}{dt},
\end{align*}
which becomes an evolution equation for the stress-like tensor $\bm Q$. The above stress-driven evolution equation is in fact the most considered type for the finite deformation linear viscoelasticity models \cite{Simo1987,Simo2006,Holzapfel1996a,Holzapfel2001,Liu2021b}. Nevertheless, deriving an evolution equation for $\bm Q$ in the nonlinear model becomes non-trivial. More importantly, working on the evolution equation for the deformation-like tensor $\bm \Gamma$ may lead to constitutive integration algorithms that preserve critical physical and mathematical structures \cite{Liu2023a}. In this study, we do not attempt to derive an evolution equation for $\bm Q$. Instead, our constitutive integration is based on the strain-driven evolution equation.
\end{remark}

\begin{remark}
The volumetric energy $\Psi^{\infty}_{\mathrm{vol}}$ is typically convex, guaranteeing the validity of the Legendre transformation. Readers may refer to \cite[Sec.~2.4]{Liu2018} for a discussion of some commonly used volumetric energies $\Psi^{\infty}_{\mathrm{vol}}$ and their counterparts $G^{\infty}_{\mathrm{vol}}$.
\end{remark}

\begin{figure}
	\begin{center}
		\includegraphics[angle=0, trim=180 100 180 70, clip=true, scale = 0.5]{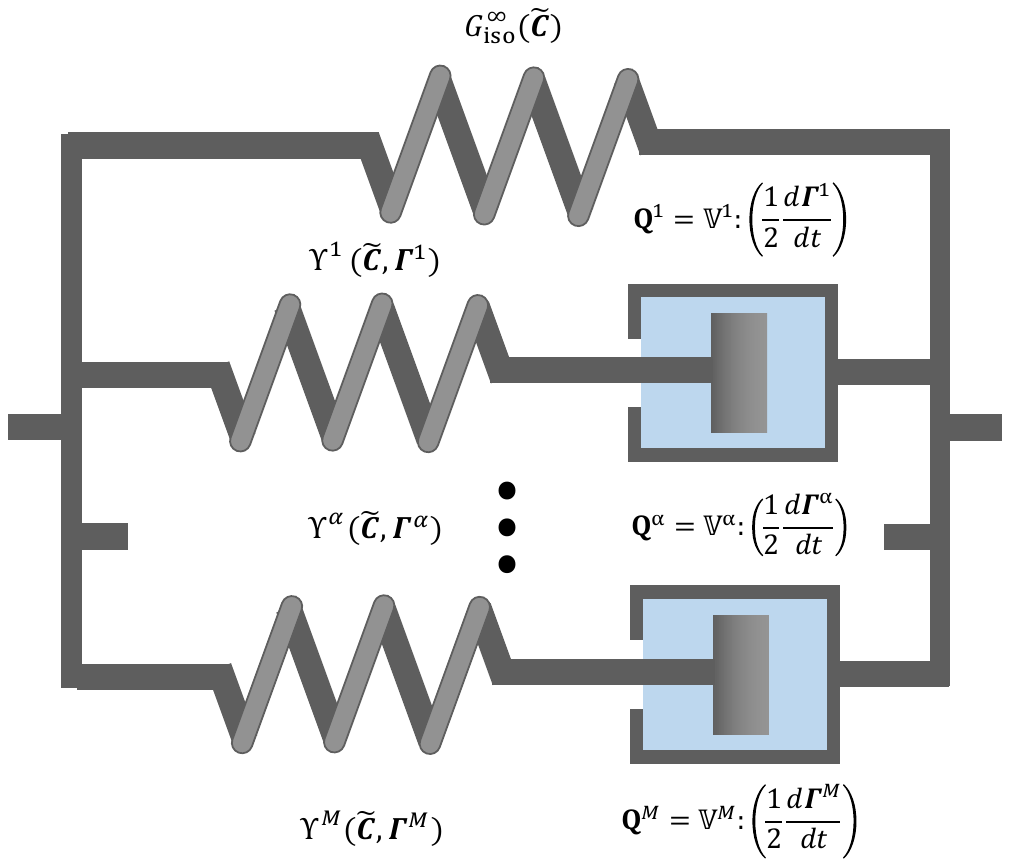}
	\end{center}
	\caption{An illustration of the multiple relaxation processes based on the Gibbs free energy by the spring-dashpot device.}
	\label{fig:multiple_relaxation}
\end{figure}

\subsubsection{Generalization: Hill's hyperelasticity of multiple terms and multiple relaxation processes}
\label{sec:gen-gibbs-multiple-terms-multiple-relax}
Similar to the generalization made in Section \ref{sec:helmholtz-gen-hyperelasticity}, we consider the hyperelasticity of Hill's class with multiple terms. The free energies $G^{\infty}_{\mathrm{iso}}$ and $\Upsilon$ are modified by using multiple quadratic terms,
\begin{align}
\label{eq:Gibbs_potential_multiple_terms_multiple_relax}
G^{\infty}_{\mathrm{iso}}( \tilde{\bm{C}} ) = \sum_{\beta = 1}^{N}  G^{\infty}_{\mathrm{iso} \: \beta}( \tilde{\bm{C}} ), \quad \Upsilon( \tilde{\bm{C}}, \bm{\Gamma}^1, \cdots, \bm \Gamma^{M} ) =  \sum_{\alpha = 1}^{M} \Upsilon^{\alpha}( \tilde{\bm{C}}, \bm{\Gamma}^{\alpha} ),
\end{align}
with
\begin{align}
\label{eq:Hill_Gibbs_Upsilon}
G^{\infty}_{\mathrm{iso} \: \beta}(\tilde{\bm C}) := \mu^{\infty}_{\beta} \left | \tilde{\bm{E}}^{\infty}_{\beta}(\tilde{\bm{C}}) \right |^2, \quad \Upsilon^{\alpha}(\tilde{\bm C}, \bm \Gamma^{\alpha}) := \mu^{\alpha} \left |\tilde{\bm{E}}^{\alpha} (\tilde{\bm{C}}) - \bm{E}^{\mathrm{v} \: \alpha}(\bm{\Gamma}^{\alpha}) \right |^2.
\end{align}
The values of $N$ and $M$ represent the number of quadratic terms used in the equilibrium energy and the number of relaxation processes, respectively. A rheological model is depicted in Figure \ref{fig:multiple_relaxation} to illustrate this generalization. The stresses are modified as 
\begin{align}
\label{eq:stress-like-tensors}
\tilde{\bm{S}}^{\infty}_{\mathrm{iso}} = \sum^{N}_{\beta=1} \tilde{\bm{T}}^{\infty}_{\beta} : \tilde{\mathbb{Q}}^{\infty}_{\beta}, \quad \tilde{\bm{S}}^{\mathrm{neq}}_{\mathrm{iso}} = \sum^{M}_{\alpha=1} \tilde{\bm{T}}^{\alpha} : \tilde{\mathbb{Q}}^{\alpha}, \quad \bm{Q}^{\alpha} := -2\frac{\partial \Upsilon}{\partial \bm \Gamma^{\alpha}} = -2 \frac{\partial \Upsilon^{\alpha}}{\partial \bm \Gamma^{\alpha}} = \tilde{\bm{T}}^{\alpha} : \mathbb{Q}^{\mathrm{v} \: \alpha},
\end{align}
with 
\begin{gather}
\label{eq:tilde-stress-like-projection-tensors}
\tilde{\bm{T}}^{\infty}_{\beta} := \frac{\partial G^{\infty}_{\mathrm{iso} \: \beta} }{\partial \tilde{\bm{E}}^{\infty}_{\beta}} = 2\mu^{\infty}_{\beta} \tilde{\bm{E}}^{\infty}_{\beta}, \quad \tilde{\bm{T}}^{\alpha} := \frac{\partial \Upsilon^{\alpha} }{\partial \tilde{\bm{E}}^{\alpha}} = - \frac{\partial \Upsilon^{\alpha} }{\partial \bm{E}^{\mathrm{v} \: \alpha}} = 2\mu^{\alpha} \left( \tilde{\bm{E}}^{\alpha} - \bm{E}^{\mathrm{v} \: \alpha} \right), \displaybreak[2] \\
\label{eq:tilde-stress-like-projection-tensors-2}
\tilde{\mathbb Q}^{\infty}_{\beta} := 2 \frac{\partial \tilde{\bm E}^{\infty}_{\beta}}{\partial  \tilde{\bm C}}, \quad \tilde{\mathbb Q}^{\alpha} := 2 \frac{\partial \tilde{\bm E}^{\alpha}}{\partial \tilde{\bm C}}, \quad \mbox{and} \quad \mathbb Q^{\mathrm v \: \alpha} := 2 \frac{\partial \bm E^{\mathrm v \: \alpha}}{\partial \bm \Gamma^{\alpha}}.
\end{gather}
There are then $M$ evolution equations, which are written as
\begin{align}
\label{eq:evolution-gibbs-generalized}
\bm Q^{\alpha} = \mathbb V^{\alpha} : \left( \frac12 \frac{d\bm \Gamma^{\alpha}}{dt} \right) \quad \mbox{or} \quad \frac12 \frac{d\bm \Gamma^{\alpha}}{dt} = (\mathbb V^{\alpha})^{-1} : \bm Q^{\alpha}, \quad \alpha = 1, \cdots, M.
\end{align}
The internal dissipation of the viscoelastic solid becomes
\begin{align*}
\mathcal D = \sum_{\alpha=1}^{M}\left( \frac12 \frac{d\bm \Gamma^{\alpha}}{dt} \right) : \mathbb V^{\alpha} : \left( \frac12 \frac{d\bm \Gamma^{\alpha}}{dt} \right)  = \sum_{\alpha = 1}^{M} \bm Q^{\alpha} : (\mathbb V^{\alpha})^{-1} : \bm Q^{\alpha}.
\end{align*}
Now we show that the generalization made here does not disturb the result of Proposition \ref{pro:equilibrium}.
\begin{proposition}
\label{pro:equilibrium_gen}
For the generalized model with the potential given in \eqref{eq:Gibbs_potential_multiple_terms_multiple_relax}, the non-eqilibrium stress $\bm S_{\mathrm{iso}}^{\mathrm{neq}}$ vanishes in the equilibrium state, that is,
\begin{align*}
\bm{S}_{\mathrm{iso}}^{\mathrm{neq}} \Big|_{\mathrm{eq}} = \bm{O}.
\end{align*} 
\end{proposition}
\begin{proof}
According to Lemma \ref{lemma:Q_v_inv}, multiplying $\mathbb{Q}^{\mathrm{v} \: \alpha \: -1}$ at both sides of \eqref{eq:stress-like-tensors}$_3$ leads to
\begin{align*}
\bm{Q}^{\alpha} : \mathbb{Q}^{\mathrm{v} \: \alpha \: -1} = \tilde{\bm{T}}^{\alpha} : \mathbb{Q}^{\mathrm{v} \: \alpha} : \mathbb{Q}^{\mathrm{v}  \: \alpha\: -1} = \tilde{\bm{T}}^{\alpha}.
\end{align*}
It is known that $\bm Q^{\alpha}|_{\mathrm{eq}} = \bm O$ due to the evolution equations \eqref{eq:evolution-gibbs-generalized}. This implies $\tilde{\bm{T}}^{\alpha}|_{\mathrm{eq}} = \bm{O}$, which leads to 
\begin{align*}
\tilde{\bm S}^{\mathrm{neq}}_{\mathrm{iso}}|_{\mathrm{eq}} = \sum^{M}_{\alpha=1} \tilde{\bm{T}}^{\alpha}|_{\mathrm{eq}} : \tilde{\mathbb{Q}}^{\alpha} = \bm O,
\end{align*}
and hence one has $\bm{S}_{\mathrm{iso}}^{\mathrm{neq}}|_{\mathrm{eq}} = J^{-\frac23}\mathbb P : \tilde{\bm S}^{\mathrm{neq}}_{\mathrm{iso}}|_{\mathrm{eq}}  = \bm{O}$.
\end{proof}
We summarize the constitutive relations for finite viscohyperelastic materials discussed in Section \ref{sec:Gibbs_constitutive} as follows.

\begin{myenv}{Constitutive laws for finite viscoelasticity}
  \noindent Gibbs free energy:
  \begin{gather*}
  G(\tilde{\bm C}, P, \bm \Gamma) = G^{\infty}_{\mathrm{vol}}(P) + G^{\infty}_{\mathrm{iso}}(\tilde{\bm{C}}) + \Upsilon(\tilde{\bm{C}}, \bm \Gamma^1, \cdots, \bm \Gamma^{M}), \displaybreak[2] \\
  G^{\infty}_{\mathrm{iso}}(\tilde{\bm{C}})  = \sum_{\beta = 1}^{N} G^{\infty}_{\mathrm{iso} \: \beta}(\tilde{\bm C}), \quad G_{\mathrm{iso}\:\beta}^{\infty}(\tilde{\bm C}) = \mu_{\beta}^{\infty} \left | \tilde{\bm{E}}^{\infty}_{\beta} (\tilde{\bm{C}}) \right |^2, \displaybreak[2] \\
  \Upsilon(\tilde{\bm{C}}, \bm \Gamma^1, \cdots, \bm \Gamma^{M}) = \sum_{\alpha = 1}^{M}  \Upsilon^{\alpha}(\tilde{\bm C}, \bm\Gamma^{\alpha}), \quad  \Upsilon^{\alpha}_{\beta}(\tilde{\bm C}, \bm\Gamma^{\alpha}) = \mu^{\alpha} \left |\tilde{\bm{E}}^{\alpha}(\tilde{\bm{C}})  - \bm{E}^{\mathrm{v} \: \alpha}(\bm{\Gamma}^{\alpha}) \right |^2.
  \end{gather*}
  
  \noindent Projection tensors:
  \begin{gather*}
  \tilde{\mathbb{Q}}^{\infty}_{\beta} = 2\frac{\partial \tilde{\bm{E}}^{\infty}_{\beta}}{\partial \tilde{\bm{C}}}, \quad  \mathbb{P} = \mathbb{I} - \frac{1}{3} \bm{C}^{-1} \otimes \bm{C}, \quad \tilde{\mathbb Q}^{\alpha} := 2\frac{\partial \tilde{\bm E}^{\alpha}}{\partial \tilde{\bm C}}, \quad \mathbb{Q}^{\mathrm{v} \: \alpha} = 2\frac{\partial \bm{E}^{\mathrm{v} \: \alpha}}{\partial \bm{\Gamma}^{\alpha}}.
  \end{gather*}
  
  \noindent Stress:
  \begin{gather*}
  \bm S = \bm S_{\mathrm{iso}} + \bm S_{\mathrm{vol}}, \quad \bm S_{\mathrm{vol}} = -J P \bm C^{-1}, \quad
  \bm{S}_{\mathrm{iso}} = \bm{S}_{\mathrm{iso}}^{\infty} + \bm{S}_{\mathrm{iso}}^{\mathrm{neq}}, \displaybreak[2] \\
  \bm{S}^{\infty}_{\mathrm{iso}} = J^{-\frac23} \mathbb{P} : \tilde{\bm{S}}_{\mathrm{iso} }^{\infty}, \quad \tilde{\bm{S}}_{\mathrm{iso}}^{\infty} = \sum_{\beta=1}^{N} \tilde{\bm{T}}^{\infty}_{\beta} : \tilde{\mathbb{Q}}^{\infty}_{\beta}, \quad  \tilde{\bm{T}}^{\infty}_{\beta} = 2 \mu_{\beta}^{\infty} \tilde{\bm{E}}^{\infty}_{\beta}(\tilde{\bm C}), \displaybreak[2] \\ 
  \bm{S}^{\mathrm{neq}}_{\mathrm{iso}} = J^{-\frac23} \mathbb{P} : \tilde{\bm{S}}_{\mathrm{iso}}^{\mathrm{neq}}, \quad \tilde{\bm{S}}_{\mathrm{iso}}^{\mathrm{neq}} = \sum_{\alpha=1}^{M} \tilde{\bm{T}}^{\alpha} : \tilde{\mathbb{Q}}^{\alpha}, \quad \tilde{\bm{T}}^{\alpha} = 2 \mu^{\alpha} \left( \tilde{\bm{E}}^{\alpha}(\tilde{\bm C}) - \bm{E}^{\mathrm{v} \: \alpha} ( \bm{\Gamma}^{\alpha}) \right).
  \end{gather*}

  \noindent Constitutive relations for $\bm{Q}^{\alpha}$:
  \begin{gather*}
  \bm{Q}^{\alpha} = \tilde{\bm{T}}^{\alpha} : \mathbb{Q}^{\mathrm{v} \: \alpha} \quad \mbox{and} \quad \mathbb Q^{\mathrm v \: \alpha} := 2 \frac{\partial \bm E^{\mathrm v \: \alpha}}{\partial \bm \Gamma^{\alpha}}.
  \end{gather*}

  \noindent Dissipation:
  \begin{align*}
  \mathcal{D} = \sum_{\alpha=1}^{M}\left( \frac12 \frac{d\bm \Gamma^{\alpha}}{dt} \right) : \mathbb V^{\alpha} : \left( \frac12 \frac{d\bm \Gamma^{\alpha}}{dt} \right)  = \sum_{\alpha = 1}^{M} \bm Q^{\alpha} : (\mathbb V^{\alpha})^{-1} : \bm Q^{\alpha}.
  \end{align*}

  \noindent Evolution equations:
  \begin{align*}
  \frac{d\bm{\Gamma}^{\alpha}}{dt} = (\mathbb V^{\alpha})^{-1} : \bm{Q}^{\alpha}.
  \end{align*} 
\end{myenv}

\section{Numerical aspects}
\label{sec:numerics}
In this section, we design numerical methods for the viscoelasticity model based on the Gibbs free energy. The numerical design involves the spatiotemporal discretization of the balance equations and the integration of the constitutive equation. The former follows the numerical scheme introduced in Part \rom{1}, and our discussion focuses on the consistent linearization of the system as well as the constitutive integration.

\subsection{Consistent linearization}
Consistent linearization is the cornerstone in the Newton-Raphson type algorithm \cite{Simo2006}. In the following, we provide the closed-form formula for the elasticity tensor associated with the model outlined in the previous section. We introduce the isochoric elasticity tensor as
\begin{align}
\label{eq:definition_CC_iso}
\mathbb{C}_{\mathrm{iso}} := 2 \frac{\partial \bm{S}_{\mathrm{iso}}}{\partial \bm{C}} = \mathbb{C}_{\mathrm{iso}}^{\infty} + \mathbb{C}_{\mathrm{iso}}^{\mathrm{neq}}, \quad \mbox{with} \quad \mathbb{C}_{\mathrm{iso}}^{\infty} := 2 \frac{\partial \bm{S}_{\mathrm{iso}}^{\infty}}{\partial \bm{C}} \quad \mbox{and} \quad \mathbb{C}_{\mathrm{iso}}^{\mathrm{neq}} := 2 \frac{\partial \bm{S}_{\mathrm{iso}}^{\mathrm{neq}}}{\partial \bm{C}},
\end{align}
due to the the fact that $\bm{S}_{\mathrm{iso}} = \bm{S}^{\infty}_{\mathrm{iso}} + \bm{S}_{\mathrm{iso}}^{\mathrm{neq}}$. With the representation of $\bm{S}^{\infty}_{\mathrm{iso}}$ given in \eqref{eq:Gibbs-constitutive-S}$_2$, the equilibrium part of $\mathbb C_{\mathrm{iso}}$ can be represented as
\begin{align}
\label{eq:CC_iso_equi}
\mathbb{C}_{\mathrm{iso}}^{\infty} =  \mathbb{P} :  \tilde{\mathbb{C}}_{\mathrm{iso}}^{\infty}  :\mathbb{P}^T + \frac23  \mathrm{Tr} \left(J^{-\frac23}\tilde{\bm{S}}^{\infty}_{\mathrm{iso}} \right) \tilde{\mathbb{P}} -  \frac23 \left( \bm{C}^{-1}\otimes \bm{S}_{\mathrm{iso}}^{\infty} + \bm{S}_{\mathrm{iso}}^{\infty}\otimes \bm{C}^{-1} \right), 
\end{align}
with the following notations,
\begin{align}
\label{eq:tilde_CC_iso_equi}
\tilde{\mathbb{C}}_{\mathrm{iso} }^{\infty} := 2J^{-\frac43} \frac{\partial \tilde{\bm{S}}_{\mathrm{iso} }^{\infty}}{\partial \tilde{\bm{C}}}, \quad
\quad \mathrm{Tr}(\cdot) := (\cdot):\bm{C}, \quad \tilde{\mathbb{P}} := \bm{C}^{-1} \odot \bm{C}^{-1} - \frac13 \bm{C}^{-1} \otimes \bm{C}^{-1}. 
\end{align}
Readers may refer to \cite[p.~255]{Holzapfel2000} for a detailed derivation of the formula \eqref{eq:CC_iso_equi}. Moreover, due to the constitutive model of $\tilde{\bm S}^{\infty}_{\mathrm{iso}}$ given in \eqref{eq:stress-like-tensors}$_1$, we have the following explicit formula for $\tilde{\mathbb{C}}_{\mathrm{iso}}^{\infty}$,
\begin{align}
\tilde{\mathbb{C}}_{\mathrm{iso}}^{\infty} = 2J^{-\frac43}\frac{\partial \tilde{\bm{S}}_{\mathrm{iso} }^{\infty}}{\partial \tilde{\bm{C}}} = J^{-\frac43} \sum_{\beta = 1}^{N} 2 \frac{\partial (\tilde{\bm{T}}_{\beta}^{\infty} : \tilde{\mathbb{Q}}^{\infty}_{\beta} ) }{\partial \tilde{\bm C}} &= J^{-\frac43} \sum_{\beta = 1}^{N} \left( \tilde{\mathbb{Q}}_{\beta}^{\infty \: T} : 2\frac{\partial \tilde{\bm{T}}_{\beta}^{\infty}}{\partial \tilde{\bm{C}}} + \tilde{\bm{T}}_{\beta}^{\infty} : 2\frac{\partial \tilde{\mathbb{Q}}^{\infty}_{\beta}}{\partial \tilde{\bm{C}}} \right) \nonumber \displaybreak[2] \\
\label{eq:S_eq_C}
&= J^{-\frac43} \sum_{\beta = 1}^{N} \left( 2\mu_{\beta}^{\infty} \tilde{\mathbb{Q}}_{\beta}^{\infty \: T} : \tilde{\mathbb{Q}}^{\infty}_{\beta} + \tilde{\bm{T}}_{\beta}^{\infty} : \tilde{\mathbb{L}}^{\infty}_{\beta} \right) , 
\end{align}
with
\begin{align}
\label{eq:tilde_L_beta}
\tilde{\mathbb L}^{\infty}_{\beta} := 4 \frac{\partial^2 \tilde{\bm E}^{\infty}_{\beta}}{\partial \tilde{\bm C} \partial \tilde{\bm C}} = 2 \frac{\partial \tilde{\mathbb Q}^{\infty}_{\beta}}{\partial \tilde{\bm C}}.
\end{align}
In the above, we have used $2 \partial \tilde{\bm{T}}_{\beta}^{\infty}/\partial \tilde{\bm{E}}_{\beta} = 2 \mu_{\beta}^{\infty} \mathbb{I}$ and the definition of $\tilde{\mathbb Q}^{\infty}_{\beta}$  given in \eqref{eq:tilde-stress-like-projection-tensors-2}$_1$. 

The derivation and final expression of $\mathbb{C}_{\mathrm{iso}}^{\mathrm{neq}}$ in \eqref{eq:definition_CC_iso} are analogous to those of $\mathbb{C}_{\mathrm{iso}}^{\infty}$ above. We give the final closed-form formula of $\mathbb{C}_{\mathrm{iso}}^{\mathrm{neq}}$ as 
\begin{align}
\label{eq:CC_iso_nonequi}
\mathbb{C}_{\mathrm{iso}}^{\mathrm{neq}} =\mathbb{P} : \tilde{\mathbb{C}}_{\mathrm{iso}}^{ \mathrm{neq} }:\mathbb{P}^T + \frac23  \mathrm{Tr} \left(J^{-\frac23} \tilde{\bm{S}}^{\mathrm{neq} }_{\mathrm{iso}} \right) \tilde{\mathbb{P}} -  \frac23 \left( \bm{C}^{-1}\otimes \bm{S}_{\mathrm{iso}}^{\mathrm{neq}} + \bm{S}_{\mathrm{iso}}^{\mathrm{neq}}\otimes \bm{C}^{-1} \right),
\end{align}
with
\begin{align}
\label{eq:elements_in_CC_neq}
\tilde{\mathbb{C}}_{\mathrm{iso}}^{\mathrm{neq}} := \sum_{\alpha = 1}^{M} J^{-\frac43} \left( 2\mu^{\alpha} \tilde{\mathbb{Q}}^{\alpha \: T} : \tilde{\mathbb{Q}}^{\alpha} + \tilde{\bm{T}}^{\alpha} : \tilde{\mathbb{L}}^{\alpha}\right), \quad \mbox{and} \quad \tilde{\mathbb L}^{ \alpha} := 4 \frac{\partial^2 \tilde{\bm E}^{\alpha}}{\partial \tilde{\bm C} \partial \tilde{\bm C}} = 2 \frac{\partial \tilde{\mathbb Q}^{\alpha}}{\partial \tilde{\bm C}}.
\end{align}
In the derivation of the above, we have used the fact that $2 \partial \tilde{\bm{T}}^{\alpha } / \partial \tilde{\bm{E}}^{\alpha} = 2 \mu^{\alpha} \mathbb{I}$ and the definitions of $\tilde{\mathbb Q}^{\alpha}$ and $\tilde{\bm{S}}^{\mathrm{neq} }_{\mathrm{iso}}$ given in the previous section.

\subsection{Evolution equation integration}
\label{sec:evolution_equation_integration}
The algorithmic procedure for evaluating $\bm Q^{\alpha}$ is based on the evolution equations for $\bm \Gamma^{\alpha}$, which reads 
\begin{align}
\label{eq:alpha-th-evolution-eqn}
\frac{d\bm{\Gamma}^{\alpha}}{dt} = (\mathbb V^{\alpha})^{-1} : \bm{Q}^{\alpha}.
\end{align}
We assume material is Newtonian, meaning $\mathbb V^{\alpha}$ is a constant tensor. Integrating the nonlinear evolution equations entails a local Newton-Raphson iterative procedure at each quadrature point. During the linearization process, it is necessary to invoke a set of rank-four tensors defined as follows, 
\begin{align}
\label{eq:definition_K}
 \mathbb{K}^{\alpha} :=  2 \frac{\partial \bm{Q}^{\alpha}} {  \partial \bm{\Gamma}^{\alpha} }  =   2 \frac{\partial (\tilde{\bm{T}}^{\alpha }: \mathbb{Q}^{\mathrm{v} \: \alpha })}{\partial \bm{\Gamma}^{\alpha}} &= \left( \mathbb{Q}^{\mathrm{v} \: \alpha \: T} : 2\frac{\partial \tilde{\bm{T}}^{\alpha}}{\partial \bm{\Gamma}^{\alpha}} + \tilde{\bm{T}}^{\alpha} : 2\frac{\partial \mathbb{Q}^{\mathrm{v} \: \alpha}}{\partial \bm{\Gamma}^{\alpha}} \right)  \nonumber \displaybreak[2] \\
&  = - 2 \mu^{\alpha} \mathbb{Q}^{\mathrm{v} \: \alpha \:  T} :  \mathbb{Q}^{\mathrm{v} \: \alpha} + \tilde{\bm{T}}^{\alpha} :\mathbb{L}^{\mathrm{v} \: \alpha}, 
\end{align}
with 
\begin{align}
\label{eq:Lv_alpha_beta}
\mathbb L^{\mathrm{v} \: \alpha} :=  4 \frac{\partial^2 \bm E^{\mathrm{v} \: \alpha}}{\partial \bm \Gamma^{\alpha} \partial \bm \Gamma^{\alpha}} = 2 \frac{\partial \mathbb Q^{\mathrm{v} \: \alpha}}{\partial \bm \Gamma^{\alpha}}.
\end{align}
In the above, we have used $2 \partial \tilde{\bm{T}}^{\alpha} / \partial \bm{E}^{\mathrm{v} \: \alpha} = -2 \mu^{\alpha} \mathbb{I}$ and the definition of $\mathbb{Q}^{\mathrm{v} \: \alpha}$ given in Seciton \ref{sec:gen-gibbs-multiple-terms-multiple-relax}. 

Let the time interval $(0, T)$ be divided into $n_{\mathrm{ts}}$ subintervals of size $\Delta t_n := t_{n+1} - t_n$ delimited by a discrete time vector $\left\lbrace t_n \right\rbrace_{n=0}^{n_{\mathrm{ts}}}$.  We integrate the evolution equations \eqref{eq:alpha-th-evolution-eqn} by the mid-point rule, leading to the following discrete evolution equations,
\begin{align*}
\frac{\bm{\Gamma}_{n+1}^{\alpha} - \bm{\Gamma}_{n}^{\alpha}}{\Delta t_n} =  (\mathbb V^{\alpha})^{-1} : \bm{Q}^{\alpha}\left( \tilde{\bm{C}}_{n+\frac12}, \bm{\Gamma}_{n+\frac12}^{\alpha} \right),
\end{align*}
with
\begin{align*}
(\bullet)_{n+\frac12} := \frac12 \left( (\bullet)_{n} + (\bullet)_{n+1} \right).
\end{align*}
In the above, $(\bullet)_{n}$ represents the approximation of the quantity $(\bullet)$ at the time instance $t_n$. Since the right-hand side of the evolution equations is nonlinear in terms of $\bm{\Gamma}_{n+1}^{\alpha}$,  the local Newton-Raphson iteration procedure must be performed at each quadrature point to determine the internal state of each relaxation process. We summarized the local Newton-Raphson iteration procedure for solving $\bm{\Gamma}_{n+1}^{\alpha}$ as follows.

\begin{myenv}{Local Newton-Raphson iteration}
\noindent \textbf{Predictor stage}: Set
\begin{align*} 
\bm{\Gamma}_{n+1\:(0)}^{\alpha} = \bm{\Gamma}_{n}^{\alpha},
\end{align*}
for $\alpha =1, ..., M$.

\noindent \textbf{Multi-Corrector stage}: Repeat the following steps for $\alpha =1, ..., M$.
\begin{enumerate}
\item Construct the local residual
\begin{align*} 
\bm{\mathrm{r}}_{n+1\:(i)}^{\alpha} = \bm{\Gamma}_{n+1\:(i)}^{\alpha} - \bm{\Gamma}_{n}^{\alpha} - \Delta t_n (\mathbb V^{\alpha})^{-1} : \bm{Q}^{\alpha}\left( \tilde{\bm{C}}_{n+\frac12}, \bm{\Gamma}_{n+\frac12}^{\alpha} \right).
\end{align*}
If one of the stopping criteria
\begin{align*}
\frac{\left \vert \bm{\mathrm{r}}_{n+1\:(i)}^{\alpha} \right \vert}{\left \vert \bm{\mathrm{r}}_{n+1\:(0)}^{\alpha} \right \vert} \leq \mathrm{tol}_{\mathrm{r}}, \quad \left \vert \bm{\mathrm{r}}_{n+1\:(i)}^{\alpha} \right \vert \leq \mathrm{tol}_{\mathrm{a}}, \quad \mbox{and} \quad i = i_{\mathrm{max}}
\end{align*} 
is satisfied, set $\bm{\Gamma}_{n+1}^{\alpha} = \bm{\Gamma}_{n+1\:(i)}^{\alpha}$ and exit the multi-corrector stage. Otherwise, continue to Step 2.
\item Linearize the local residual
\begin{align*}
\left. \frac{\partial \bm{\mathrm{r}}_{n+1}^{\alpha}}{\partial \bm{\Gamma}_{n+1}^{\alpha}} \right \vert_{(i)} = \mathbb{I} - \frac{\Delta t_n}{2} (\mathbb V^{\alpha})^{-1} :  \mathbb{K}^{\alpha}\left( \tilde{\bm{C}}_{n+\frac12}, \bm{\Gamma}_{n+\frac12}^{\alpha} \right).
\end{align*}
\item Solve the incremental solution
\begin{align*}
\left( \mathbb{I} - \frac{\Delta t_n}{2} (\mathbb V^{\alpha})^{-1} : \mathbb{K}^{\alpha}\left( \tilde{\bm{C}}_{n+\frac12}, \bm{\Gamma}_{n+\frac12}^{\alpha}\right) \right) : \Delta \bm{\Gamma}_{n+1\:(i)}^{\alpha} = -\bm{\mathrm{r}}_{n+1\:(i)}^{\alpha}.
\end{align*}
\item Update the discrete internal state variable as
\begin{align*}
\bm{\Gamma}_{n+1\:(i+1)}^{\alpha} = \bm{\Gamma}_{ n+1\:(i)}^{\alpha} + \Delta \bm{\Gamma}_{n+1\:(i)}^{\alpha}.
\end{align*}
\end{enumerate}
\end{myenv}

\subsection{Modular implementation}
Given that the material model under consideration relies on the hyperelasticity of Hill's class, the implementation of the constitutive relations can be performed in a modular approach. It follows the algorithmic steps summarized by Miehe in additive plasticity \cite{Miehe1998a,Miehe2002,Friedlein2022} and entails three essential steps: the geometric preprocessing, constitutive model evaluation in the generalized strain space, and geometric postprocessing. Here we describe the modular approach for the viscoelasticity model in details. In the geometric preprocessing step, the strains and projection tensors are evaluated based on the deformation tensor and internal state variables at each quadrature point.

\noindent \textbf{Geometric preprocessing}: Given the deformation gradient $\bm F_{n+1}$ and the internal state variables $\{\bm \Gamma^{\alpha}_n\}$, perform the following steps.

\begin{enumerate}
\item Calculate the unimodular deformation tensor $\tilde{\bm C}_{n+1}$. 
\item Perform the spectral decomposition on the deformation tensor to obtain the principal stretches $\{\tilde{\lambda}_{a \: n+1}\}$ and Lagrangian principal directions $\{\bm N_{a \: n+1}\}$ for $a=1,2,3$.
\item Based on the scale functions $E^{\infty}_{\beta}$, calculate the strain 
\begin{align*}
\tilde{\bm E}^{\infty}_{\beta \: n+1}  =  \sum_{a=1}^{3} E^{\infty}_{\beta}(\tilde{\lambda}_{a \: n+1}) \bm{M}_{a \: n+1}
\end{align*}
and the related projection tensors $\mathbb Q^{\infty}_{\beta \: n+1}$ for $\beta = 1, \cdots, N$.
\item Based on the scale functions $E^{\alpha}$, calculate the strain 
\begin{align*}
\tilde{\bm E}^{\alpha}_{n+1}  =  \sum_{a=1}^{3} E^{\alpha}(\tilde{\lambda}_{a \: n+1}) \bm{M}_{a \: n+1}
\end{align*}
and the related projection tensors $\tilde{\mathbb Q}^{\alpha}_{n+1}$ and $\tilde{\mathbb L}^{\alpha}_{n+1}$ for $\alpha = 1, \cdots, M$.
\item Obtain $\bm \Gamma^{\alpha}_{n+1}$ for $\alpha=1,\cdots, M$ by invoking the local Newton-Raphson iteration outlined in Section \ref{sec:evolution_equation_integration}.
\item Perform the spectral decompositions on $\{\bm \Gamma^{\alpha}_{n+1}\}$ to obtain $\{\Gamma^{\alpha}_{a \: n+1}\}$ and $\{\bar{\bm N}^{\alpha}_{a \: n+1}\}$ for $\alpha = 1, \cdots, M$ and $a = 1, 2, 3$.
\item Based on the scale functions $E^{\mathrm v \: \alpha}$, calculate the strain 
\begin{align*}
\bm E^{\mathrm v \: \alpha}_{n+1}  =  \sum_{a=1}^{3} E^{\mathrm v \: \alpha}(\Gamma^{\alpha}_{a \: n+1}) \bar{\bm{M}}_{a \: n+1}
\end{align*}
for $\alpha = 1, \cdots, M$.
\end{enumerate}
\begin{remark}
In the above preprocessing procedures, we invoked the spectral decomposition algorithm developed by Scherzinger and Dohrmann \cite{Scherzinger2008}. Its accuracy is comparable to that of LAPACK \cite{Harari2022}, and its robustness is superior to the conventional Cardano's formula when there are nearly identical eigenvalues \cite{Guan2023}. Additionally, the projection tensors are computed following the formulas \eqref{eq:Q} and \eqref{eq:L} using the corresponding eigenvalues and eigenvectors.
\end{remark}

Moving to the second step, we conduct the constitutive evaluation in the generalized strain space. As a main feature of Hill's hyperelasticity, the conjugate stress-like tensors depend on the generalized strains in a linear fashion and can be straightforwardly evaluated within this step.

\noindent \textbf{Constitutive model evaluation}: Given the generalized strains, perform the following procedures to obtain their work-conjugate pairs as well as the fictitious elasticity tensors.

\begin{enumerate}
\item Evaluate the stresses $\tilde{\bm T}^{\infty}_{\beta \: n+1} = 2\mu^{\infty}_{\beta} \tilde{\bm E}^{\infty}_{\beta \: n+1}$ for $\beta =1, \cdots, N$.
\item Evaluate the stresses $\tilde{\bm T}^{\alpha}_{n+1} = 2\mu^{\alpha}\left( \tilde{\bm E}^{\alpha}_{n+1} - \bm E^{\mathrm v \: \alpha}_{n+1} \right)$ for $\alpha =1, \cdots, M$.
\item Evaluate the fictitious elasticity tensors for the equilibrium part as
\begin{align*}
\tilde{\mathbb{C}}_{\mathrm{iso} \: n+1}^{\infty} = \sum_{\beta = 1}^{N} J_{n+1}^{-\frac43}  \left( 2\mu_{\beta}^{\infty} \tilde{\mathbb{Q}}_{\beta \: n+1}^{\infty \: T} : \tilde{\mathbb{Q}}^{\infty}_{\beta \: n+1} + \tilde{\bm{T}}_{\beta \: n+1}^{\infty} : \tilde{\mathbb{L}}^{\infty}_{\beta \: n+1} \right).
\end{align*}
\item Evaluate the fictitious elasticity tensors for the non-equilibrium part as
\begin{align*}
\tilde{\mathbb{C}}_{\mathrm{iso} \: n+1}^{\mathrm{neq}} := \sum_{\alpha = 1}^{M} J_{n+1}^{-\frac43} \left( 2\mu^{\alpha} \tilde{\mathbb{Q}}^{\alpha \: T}_{n+1} : \tilde{\mathbb{Q}}^{\alpha}_{n+1} + \tilde{\bm{T}}^{\alpha}_{n+1} : \tilde{\mathbb{L}}_{n+1}^{\alpha}\right).
\end{align*}
\end{enumerate}

In the last step, known as geometric postprocessing, the Piola-Kirchhoff stresses and elasticity tensors are obtained by using the projectors.

\noindent \textbf{Geometric postprocessing}: Obtain the second Piola-Kirchhoff stresses and the elasticity tensor through the following procedures.

\begin{enumerate}
\item Obtain the equilibrium part of the fictitious stress as 
\begin{align*}
\tilde{\bm S}^{\infty}_{\mathrm{iso}\: n+1} = \sum_{\beta=1}^{N} \tilde{\bm T}^{\infty}_{\beta \: n+1} : \tilde{\mathbb Q}^{\infty}_{\beta \: n+1}.
\end{align*}
\item Obtain the non-equilibrium part of the fictitious stress as
\begin{align*}
\tilde{\bm S}^{\mathrm{neq}}_{\mathrm{iso} \: n+1} = \sum_{\alpha=1}^{M} \tilde{\bm T}^{\alpha}_{n+1} : \tilde{\mathbb Q}^{\alpha}_{n+1}.
\end{align*}
\item Project the fictitious stresses to the second Piola-Kirchhoff stress as
\begin{align*}
\bm S^{\infty}_{\mathrm{iso} \: n+1} = J^{-\frac23} \mathbb P : \tilde{\bm S}^{\infty}_{\mathrm{iso}\: n+1}, \quad
\bm S^{\mathrm{neq}}_{\mathrm{iso} \: n+1} = J^{-\frac23} \mathbb P : \tilde{\bm S}^{\mathrm{neq}}_{\mathrm{iso} \: n+1}, \quad
\bm S_{n+1} = \bm S^{\infty}_{\mathrm{iso} \: n+1} + \bm S^{\mathrm{neq}}_{\mathrm{iso} \: n+1}.
\end{align*}
\item Obtain the equilibrium part of the elasticity tensor as
\begin{align*}
\mathbb{C}_{\mathrm{iso} \: n+1}^{\infty} =  \mathbb{P}_{n+1} :  \tilde{\mathbb{C}}_{\mathrm{iso} \: n+1}^{\infty}  :\mathbb{P}^T_{n+1} +& \frac23  \mathrm{Tr} \left(J_{n+1}^{-\frac23}\tilde{\bm{S}}^{\infty}_{\mathrm{iso} \: n+1} \right) \tilde{\mathbb{P}}_{n+1} \displaybreak[2] \nonumber \\
&-  \frac23 \left( \bm{C}_{n+1}^{-1}\otimes \bm{S}_{\mathrm{iso} \: n+1}^{\infty} + \bm{S}_{\mathrm{iso} \: n+1}^{\infty}\otimes \bm{C}_{n+1}^{-1} \right).
\end{align*}
\item Obtain the non-equilibrium part of the elasticity tensor as
\begin{align*}
\mathbb{C}_{\mathrm{iso} \: n+1}^{\mathrm{neq}} = \mathbb{P}_{n+1} : \tilde{\mathbb{C}}_{\mathrm{iso} \: n+1}^{ \mathrm{neq} }:\mathbb{P}_{n+1}^T +& \frac23  \mathrm{Tr} \left(J_{n+1}^{-\frac23} \tilde{\bm{S}}^{\mathrm{neq} }_{\mathrm{iso} \: n+1} \right) \tilde{\mathbb{P}}_{n+1} \nonumber \displaybreak[2] \\
&-  \frac23 \left( \bm{C}_{n+1}^{-1}\otimes \bm{S}_{\mathrm{iso} \: n+1}^{\mathrm{neq}} + \bm{S}_{\mathrm{iso} \: n+1}^{\mathrm{neq}}\otimes \bm{C}_{n+1}^{-1} \right).
\end{align*}
\item Obtain the isochoric part of the elasticity tensor as
\begin{align*}
\mathbb C_{\mathrm{iso} \: n+1 } = \mathbb{C}_{\mathrm{iso} \: n+1}^{\infty} + \mathbb{C}_{\mathrm{iso} \: n+1}^{\mathrm{neq}}.
\end{align*}
\end{enumerate}
Through the above three modular procedures, we may obtain the stress and elasticity tensor that are necessary for the assembly of the residual vector and consistent tangent matrix in implicit analysis. We mention that the implementation of the algorithm relies heavily on efficient calculations of tensorial operations, in which the tensor symmetry can be exploited to save memory and computational costs.

\section{Examples}
\label{sec:examples}
In this section, we investigate the performance of the proposed model. The boundary of the domain $\Gamma_{\bm X} := \partial \Omega_{\bm X}$ admits non-overlapping subdivisions as $\Gamma_{\bm X} = \overline{\Gamma^{G}_{\bm X} \cup  \Gamma^{H}_{\bm X}}$. The functions $\bm G$ and $\bm H$ are the boundary displacement and traction prescribed on $\Gamma^{G}_{\bm X}$ and $\Gamma^{H}_{\bm X}$, respectively. In more involved scenarios, we may prescribe, for example, the normal component of the displacement and the tangential components of the traction on a boundary region. That type of boundary condition will be specified with the aid of Cartesian components of the functions $\bm G$ and $\bm H$. The problem is discretized by an inf-sup stable element pair \cite{Liu2019a}. The discrete pressure space is generated by the NURBS basis functions with degree $\mathsf p$ and highest possible continuity. The discrete velocity space is constructed by $p$-refinement to raise the degree to $\mathsf p+1$ while maintaining the continuity. The time integration is based on the generalized-$\alpha$ scheme \cite{Jansen2000,Kadapa2017,Liu2018}, which is parameterized by a single parameter $\varrho_{\infty}$, the spectral radius of the amplification matrix at the highest mode. The following choices are made in the numerical investigations, unless otherwise specified.

\begin{enumerate}
\item The meter-kilogram-second system of units is used;
\item The two-parameter Curnier-Rakotomanana strain family is utilized to characterize the hyperelastic material behavior within Hill's framework;
\item The viscosity tensor is isotropic with the deviatoric and volumetric viscosities being identical (i.e., $\mathbb V^{\alpha} = 2\eta^{\alpha} \mathbb I$);
\item The relaxation times are defined as $\tau^{\alpha} := \eta^{\alpha} / \mu^{\alpha}$;
\item We choose $\varrho_{\infty}=0.0$ in the generalized-$\alpha$ scheme;
\item We use $\mathsf p+2$ Gaussian quadrature points in each direction.
\item We use parameters $\mathrm{tol}_\mathrm{r} = 10^{-10}$, $\mathrm{tol}_\mathrm{a} = 10^{-10}$, $i_{\mathrm{max}} = 5$ in the local Newton-Raphson iteration;
\item We use the relative tolerance $\mathrm{tol}_\mathrm{R} = 10^{-10}$, absolute tolerance $\mathrm{tol}_\mathrm{A} = 10^{-10}$, and maximum number of iterations $l_{\mathrm{max}} = 10$ as the stopping criteria in the global Newton-Raphson iteration;
\end{enumerate}

\subsection{Creep test}
\begin{table}[htbp]
  \centering 
  \begin{tabular}{ m{.45\textwidth} m{.45\textwidth} }
    \hline
    \begin{minipage}{.45\textwidth}
    \centering
      \includegraphics[width=1.15\linewidth, trim=120 270 90 250, clip]{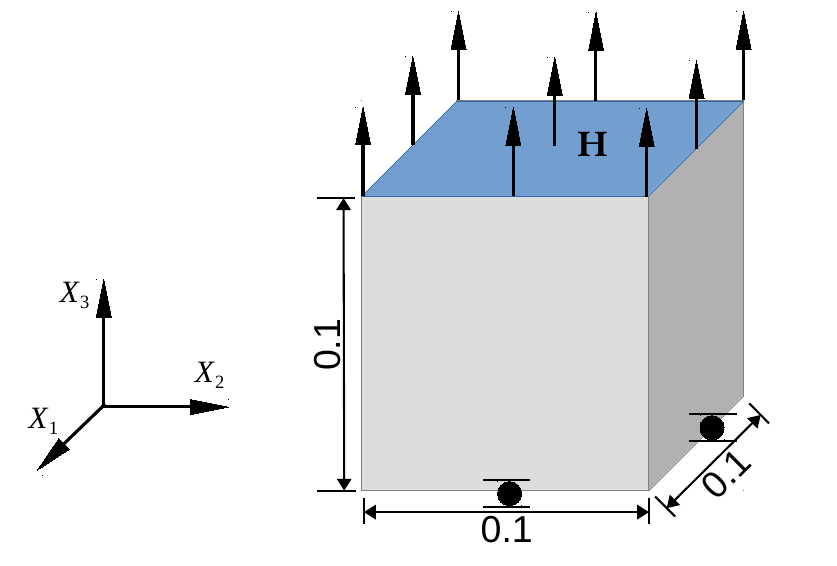}
    \end{minipage}
    & 
    \begin{minipage}{.45\textwidth}
      \begin{itemize}
        \item[] Material properties:
        \item[] $\rho_0 = 1.0\times 10^3$, \quad $N = 1$, \quad $M = 1$,
        \item[] $G_{\mathrm{iso}}^{\infty} = \mu_{1}^{\infty} \left| \tilde{\bm{E}}^{\infty}_{1} \right|^2$,
        \item[] $\mu_1^{\infty} = 4.225 \times 10^5$,
        \item[] $\Upsilon^{1} = \mu^{1} \left| \tilde{\bm{E}}^{1} - \bm{E}^{\mathrm{v}\: 1} \right|^2$,
        \item[] $\mu^1 = 4.225 \times 10^5$.
      \end{itemize}
    \end{minipage}   
    \\
    \hline
  \end{tabular}
  \caption{Creep test: problem setting.}
\label{table:creep_test}
\end{table}

\begin{figure}
\begin{center}
\begin{tabular}{cc}
\includegraphics[angle=0, trim=100 215 100 100, clip=true, scale = 0.18]{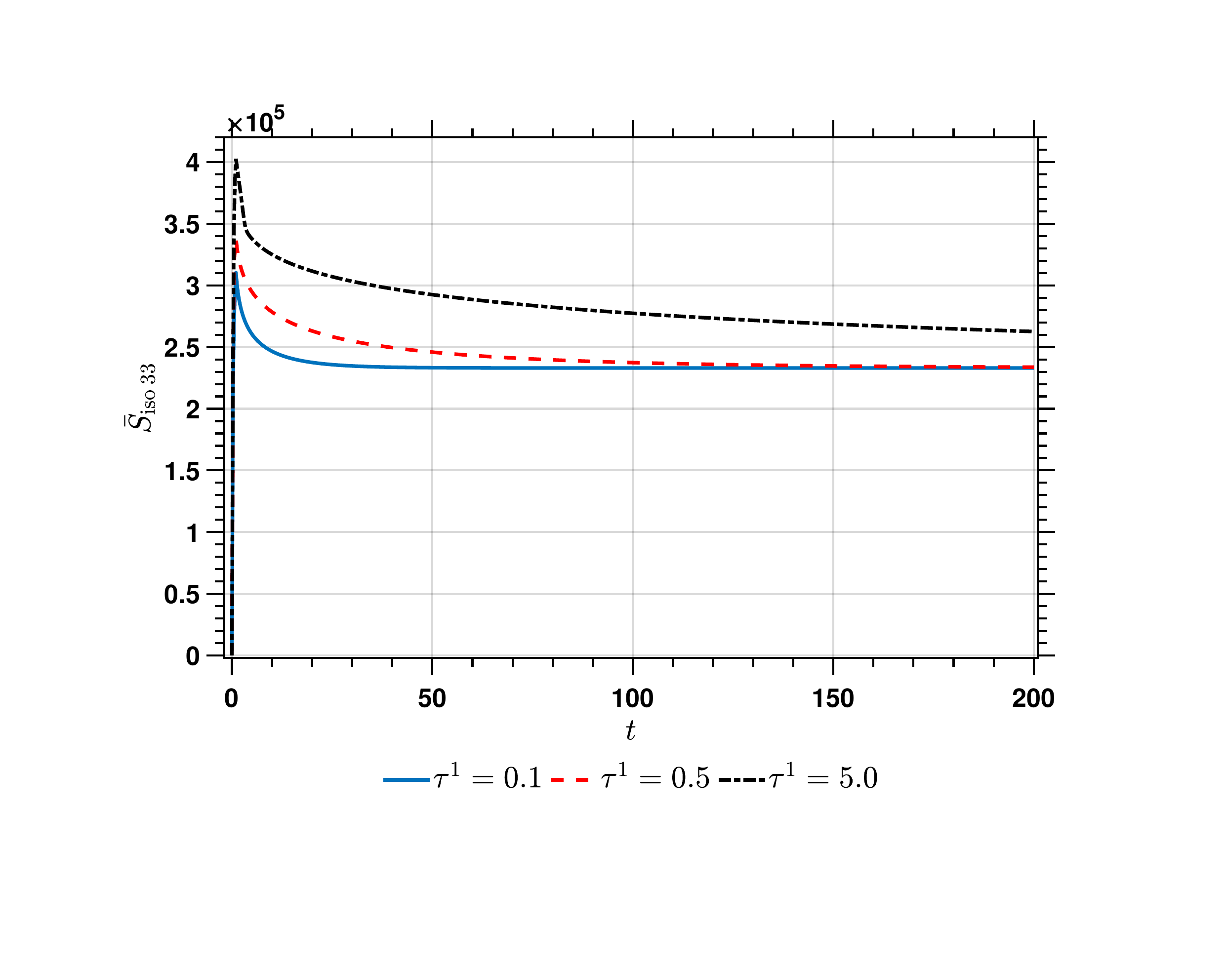} &
\includegraphics[angle=0, trim=100 215 100 100, clip=true, scale = 0.18]{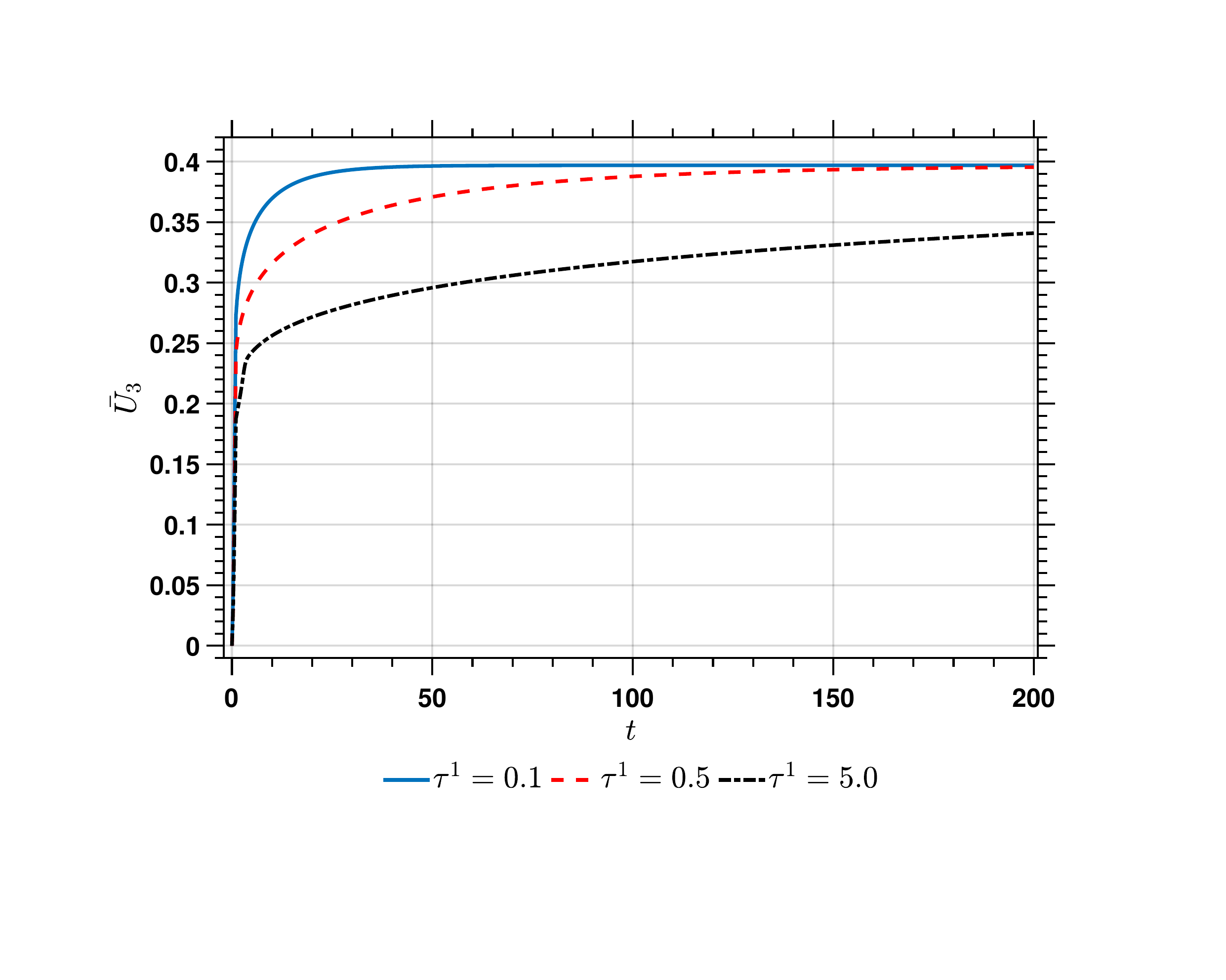} \\
\includegraphics[angle=0, trim=100 215 100 100, clip=true, scale = 0.18]{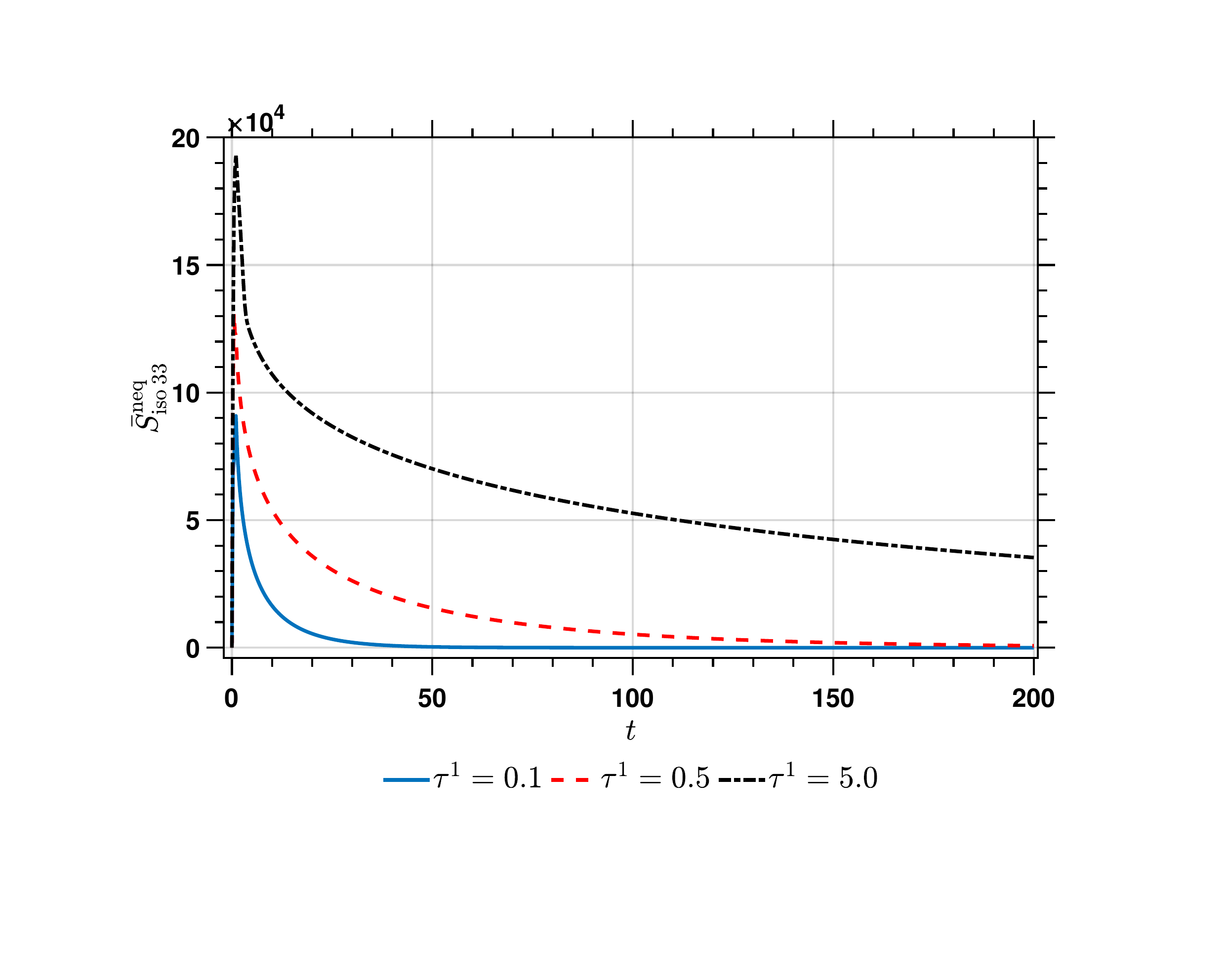} &
\includegraphics[angle=0, trim=100 215 100 100, clip=true, scale = 0.18]{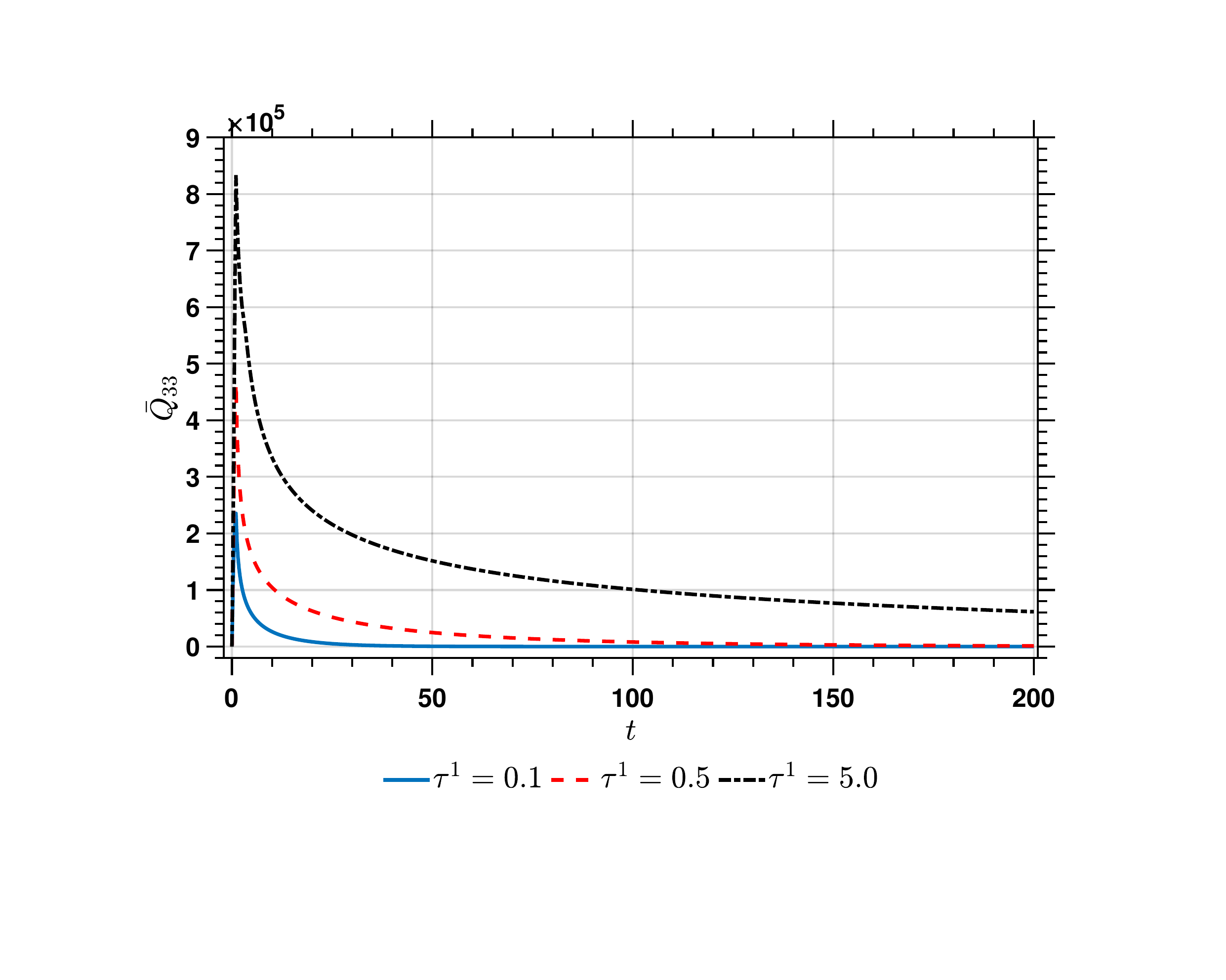} \\
\multicolumn{2}{c}{ \includegraphics[angle=0, trim=250 180 250 720, clip=true, scale = 0.25]{./creep_Q33_tau.pdf} }
\end{tabular}
\end{center}
\caption{Creep test: the effects of the relaxation time. The strain is given by the parameters $m=1.2$ and $n=1.4$.}
\label{fig:creep_tau}
\end{figure}

\begin{figure}
\begin{center}
\begin{tabular}{cc}
\includegraphics[angle=0, trim=100 215 100 100, clip=true, scale = 0.18]{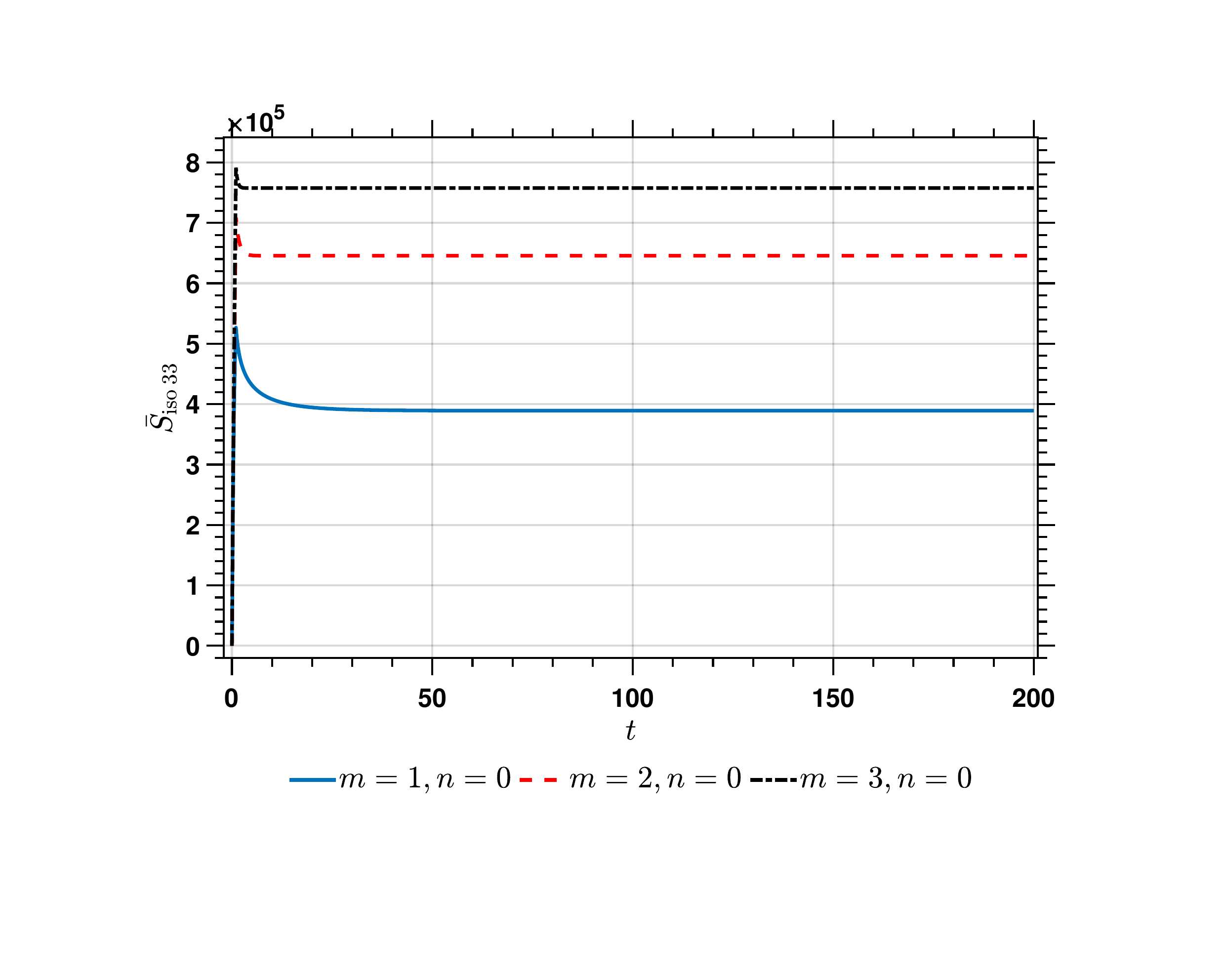} &
\includegraphics[angle=0, trim=100 215 100 100, clip=true, scale = 0.18]{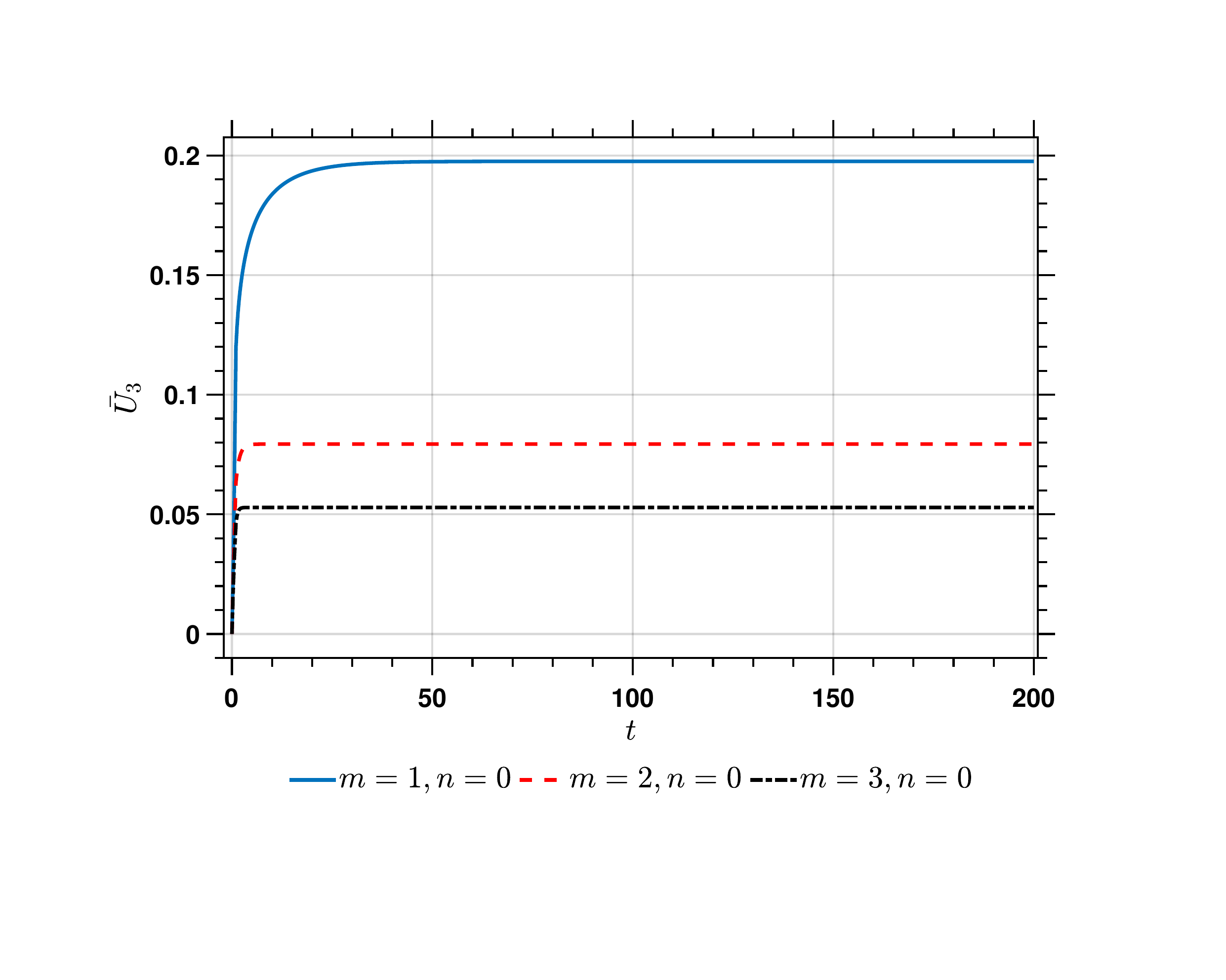} \\
\includegraphics[angle=0, trim=100 215 100 100, clip=true, scale = 0.18]{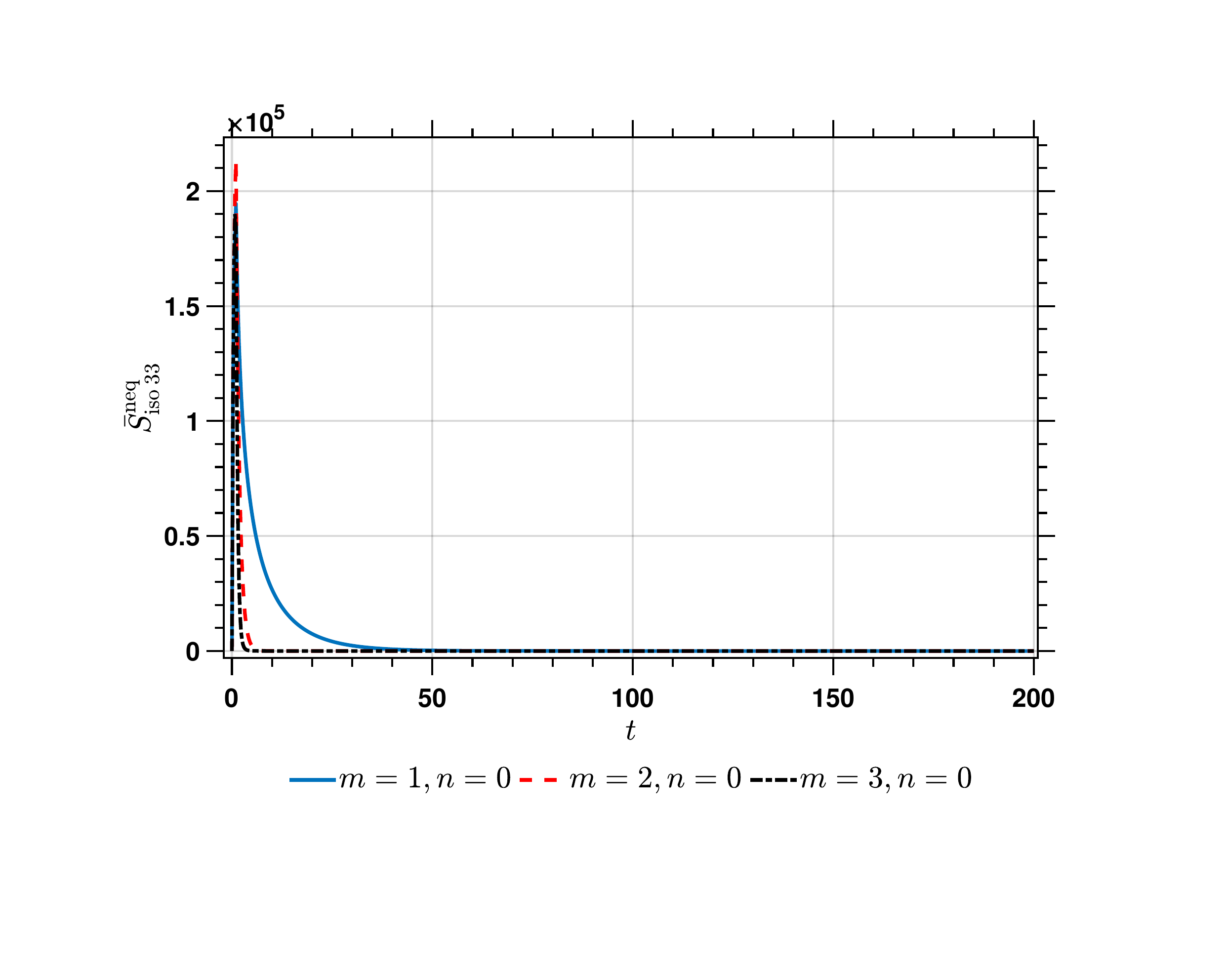} &
\includegraphics[angle=0, trim=100 215 100 100, clip=true, scale = 0.18]{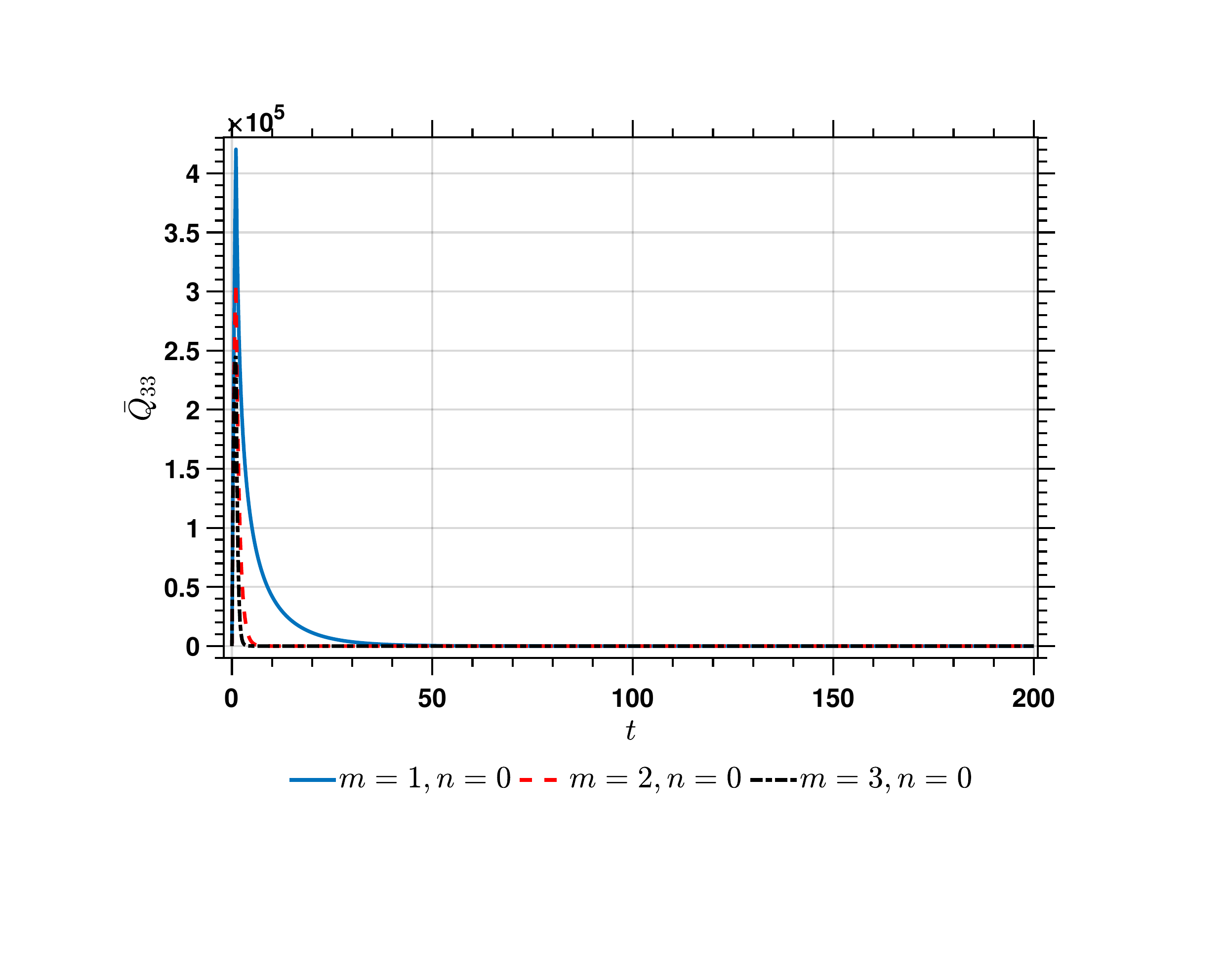} \\
\multicolumn{2}{c}{ \includegraphics[angle=0, trim=150 180 150 720, clip=true, scale = 0.25]{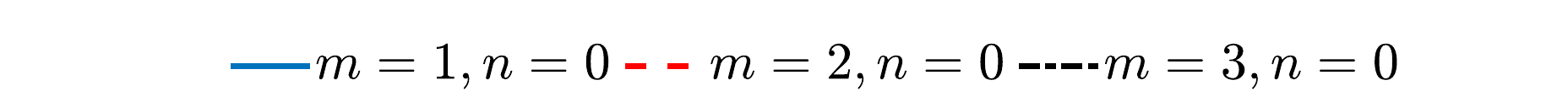} }
\end{tabular}
\end{center}
\caption{Creep test: the effects of the different strains. The relaxation time is $0.5$.}
\label{fig:creep_paras}
\end{figure}

We investigate the creep behavior with the problem setting summarized in Table \ref{table:creep_test}. The strains $\tilde{\bm E}^{\infty}_1$, $\tilde{\bm E}^1$, and $\bm E^{\mathrm{v} \: 1}$ are defined by the same two parameters $m$ and $n$, which will be specified momentarily. The problem is integrated up to $T=200$. On the bottom boundary, the tangential components of the traction are zero (i.e., $H_1 = H_2 = 0$), and the normal displacement is fixed (i.e., $G_3 = 0$). On the top boundary $\Gamma_{\mathrm{top}}$, the traction of the form $\bm{H} = \left[ 0, 0, H_3(t) \right]^T$ is imposed with 
\begin{align*}
H_3(t) = 
\begin{cases}
1.737 \times 10^6 t, & t \leq 1.0,\\
1.737 \times 10^6, & 1.0 < t \leq T.
\end{cases}
\end{align*}
On the rest boundary surfaces, stress-free boundary conditions are applied. In this example, we monitor the following four averaged quantities on the top surface,
\begin{align*} 
\bar{S}_{\mathrm{iso}\:33} := \frac{\int_{\Gamma_{\mathrm{top}}} S_{\mathrm{iso}\:33} d{\Gamma}_X}{\int_{\Gamma_{\mathrm{top}}} d{\Gamma}_X}, \quad \bar{U}_{3} := \frac{\int_{\Gamma_{\mathrm{top}}} U_3 d{\Gamma}_X}{\int_{\Gamma_{\mathrm{top}}} d{\Gamma}_X}, \quad \bar{S}^{\mathrm{neq}}_{\mathrm{iso}\:33} := \frac{\int_{\Gamma_{\mathrm{top}}} S^{\mathrm{neq}}_{\mathrm{iso}\:33} d{\Gamma}_X}{\int_{\Gamma_{\mathrm{top}}} d{\Gamma}_X}, \quad \bar{Q}_{33} := \frac{\int_{\Gamma_{\mathrm{top}}} Q_{33} d{\Gamma}_X}{\int_{\Gamma_{\mathrm{top}}} d{\Gamma}_X}.
\end{align*}
Before our investigation, we have adopted two meshes to guarantee the mesh independence of the reported results. In the following, the results are obtained based on the time step size $\Delta t_n = 0.05$ and a mesh of $5\times5\times5$ elements with the polynomial degree $\mathsf{p}=2$. In Figure \ref{fig:creep_tau}, we analyze the model behavior with three different relaxation times. It can be observed that the averaged stress $\bar{\bm S}_{\mathrm{iso} \: 33}$ and the averaged displacement $\bar{U}_3$ approach their values at the equilibrium state, for the cases with smaller values of $\tau^1$. The corresponding values of $\bar{Q}_{33}$ and $\bar{S}^{\mathrm{neq}}_{\mathrm{iso}\:33}$ both tend towards zero, confirming the relaxation property. It is also apparent that, with a larger value of the relaxation time, the model requires a longer period of time to achieve the equilibrium state. In Figure \ref{fig:creep_paras}, we examine the influence of the generalized strains by considering difference choices of $m$ and $n$ in the scale function of the Curnier-Rakotomanana family. In particular, the case of $m=2$ and $n=0$ is considered, which corresponds to the finite deformation linear model \cite{Liu2021b}. It is observed that different strains have a significant impact on the values of $\bar{\bm S}_{\mathrm{iso} \: 33}$ and $\bar{U}_{3}$ at the equilibrium state. In the meantime, the averaged stresses $\bar{S}^{\mathrm{neq}}_{\mathrm{iso}\:33}$ and $\bar{Q}_{33}$ vanish with different rates when the strains are different. 

\subsection{Shear test}
\begin{table}[htbp]
  \centering 
  \begin{tabular}{ m{.45\textwidth} m{.45\textwidth} }
    \hline \\
    \begin{minipage}{.45\textwidth}
    \centering
      \includegraphics[width=1.15\linewidth, trim=120 290 90 270, clip]{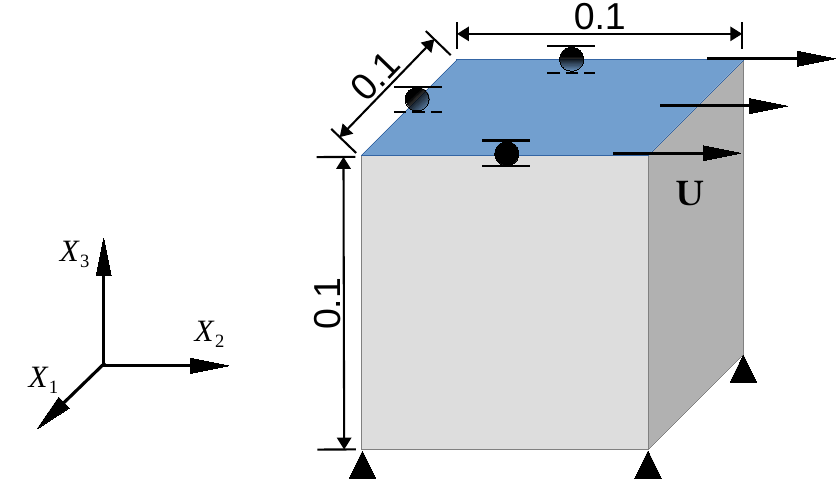}
    \end{minipage}
    &
    \begin{minipage}{.45\textwidth}
      \begin{itemize}
        \item[] Material properties:
        \item[] $\rho_0 = 1.0\times 10^3,\quad N = 2, \quad M = 1$,
        \item[] $G_{\mathrm{iso}}^{\infty} = \mu_1^{\infty} \left| \tilde{\bm{E}}^{\infty}_1 \right|^2+\mu_2^{\infty} \left| \tilde{\bm{E}}^{\infty}_2 \right|^2$, 
        \item[] $\mu_1^{\infty} = 1.75 \times 10^5, \quad \mu_2^{\infty} = 0.35 \times 10^5$,
        \item[] $\Upsilon^1 =  \mu^1 \left| \tilde{\bm{E}}^1 - \bm{E}^{\mathrm{v}\:1}\right|^2$,
        \item[] $\tau^1 = 17.5$, $\mu^1 = 5.36 \times 10^5$.
      \end{itemize}
    \end{minipage} \\ \\
    \hline
  \end{tabular}
    \caption{Shear test: problem setting}
\label{table:shear_test}
\end{table}

\begin{figure}
	\begin{center}
		\begin{tabular}{cc}
			\includegraphics[angle=0, trim=80 225 100 100, clip=true, scale = 0.2]{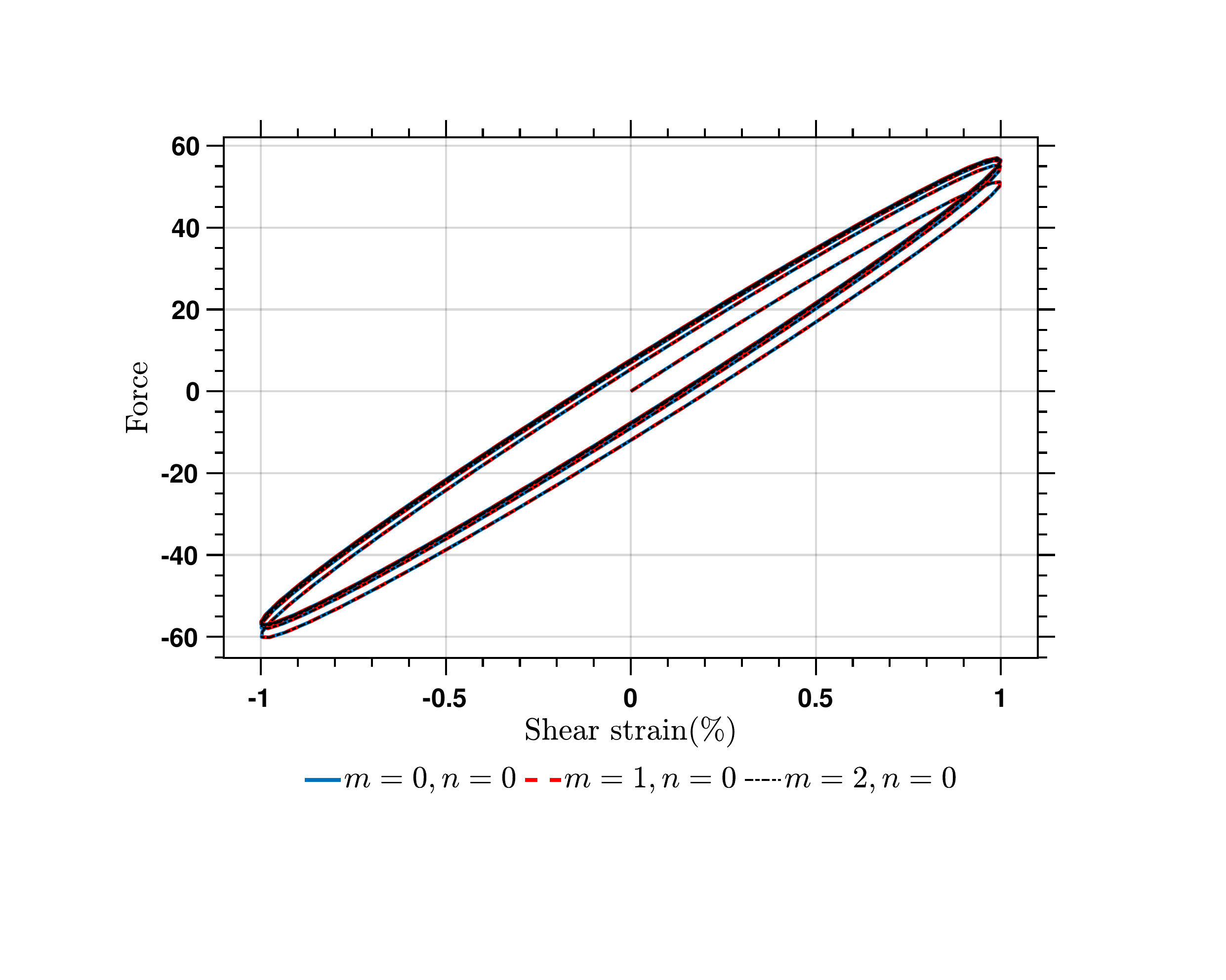} & 
			\includegraphics[angle=0, trim=80 225 100 100, clip=true, scale = 0.2]{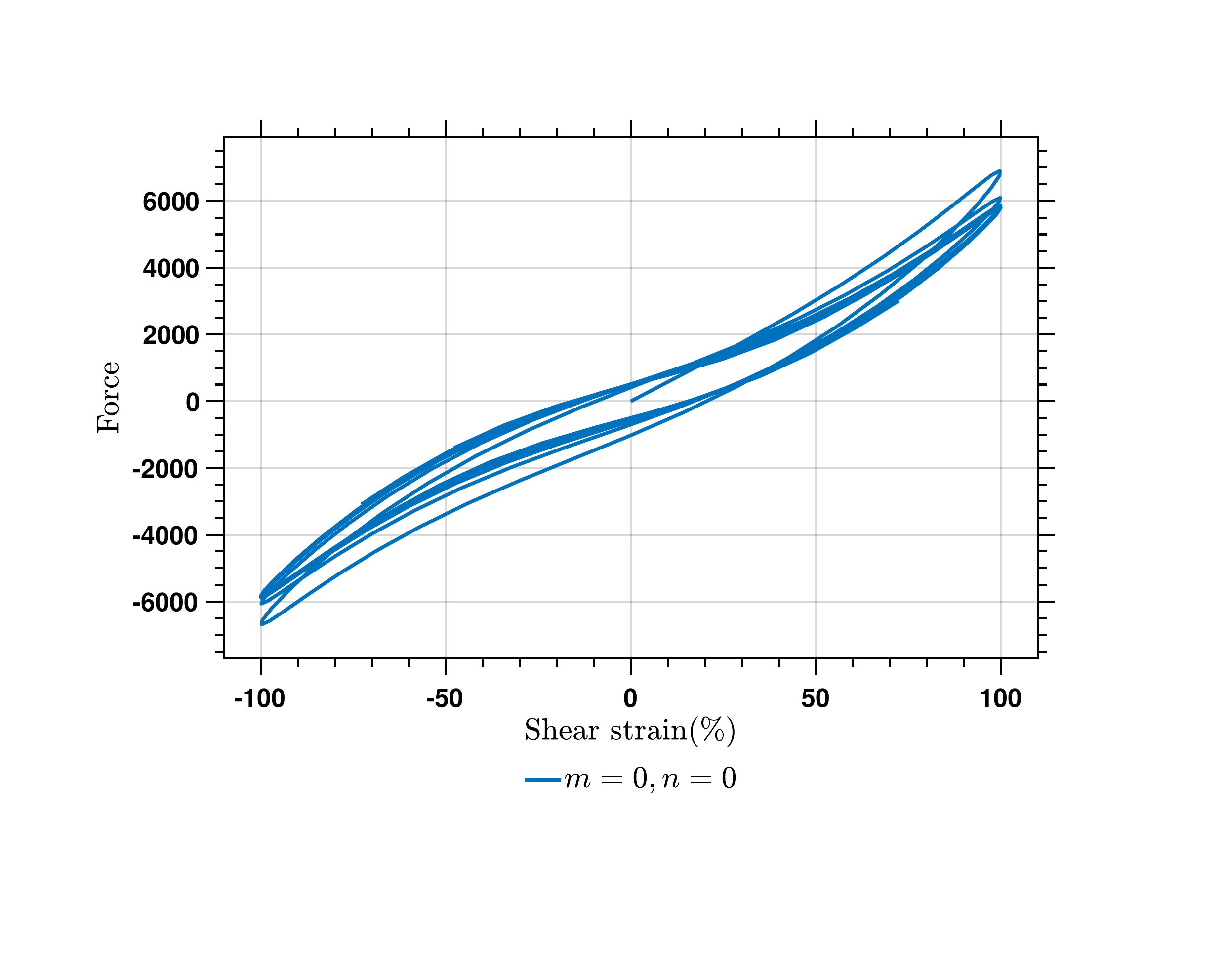} \\
			(a) & (b) \\
			\includegraphics[angle=0, trim=80 225 100 100, clip=true, scale = 0.2]{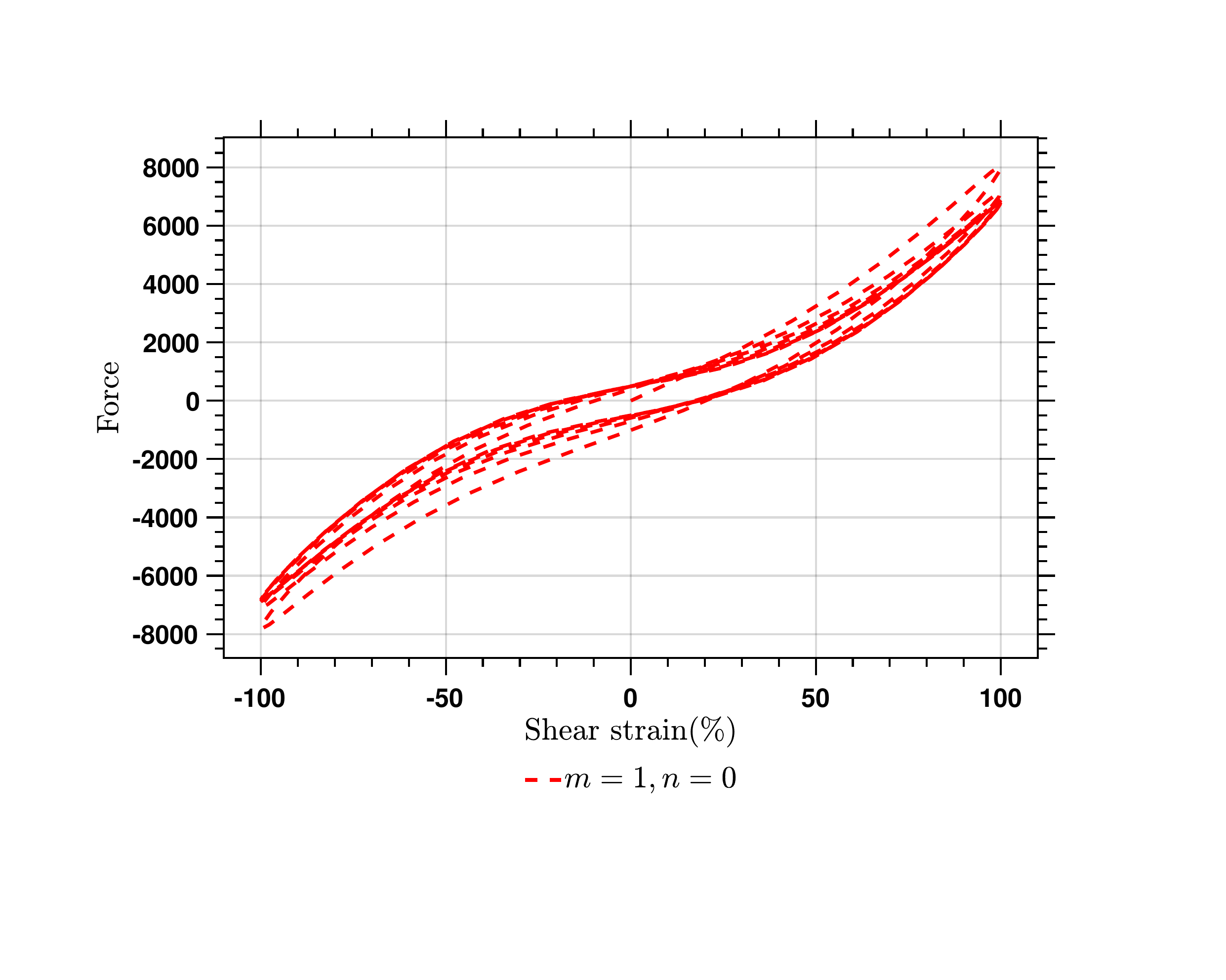} & 
			\includegraphics[angle=0, trim=80 225 100 100, clip=true, scale = 0.2]{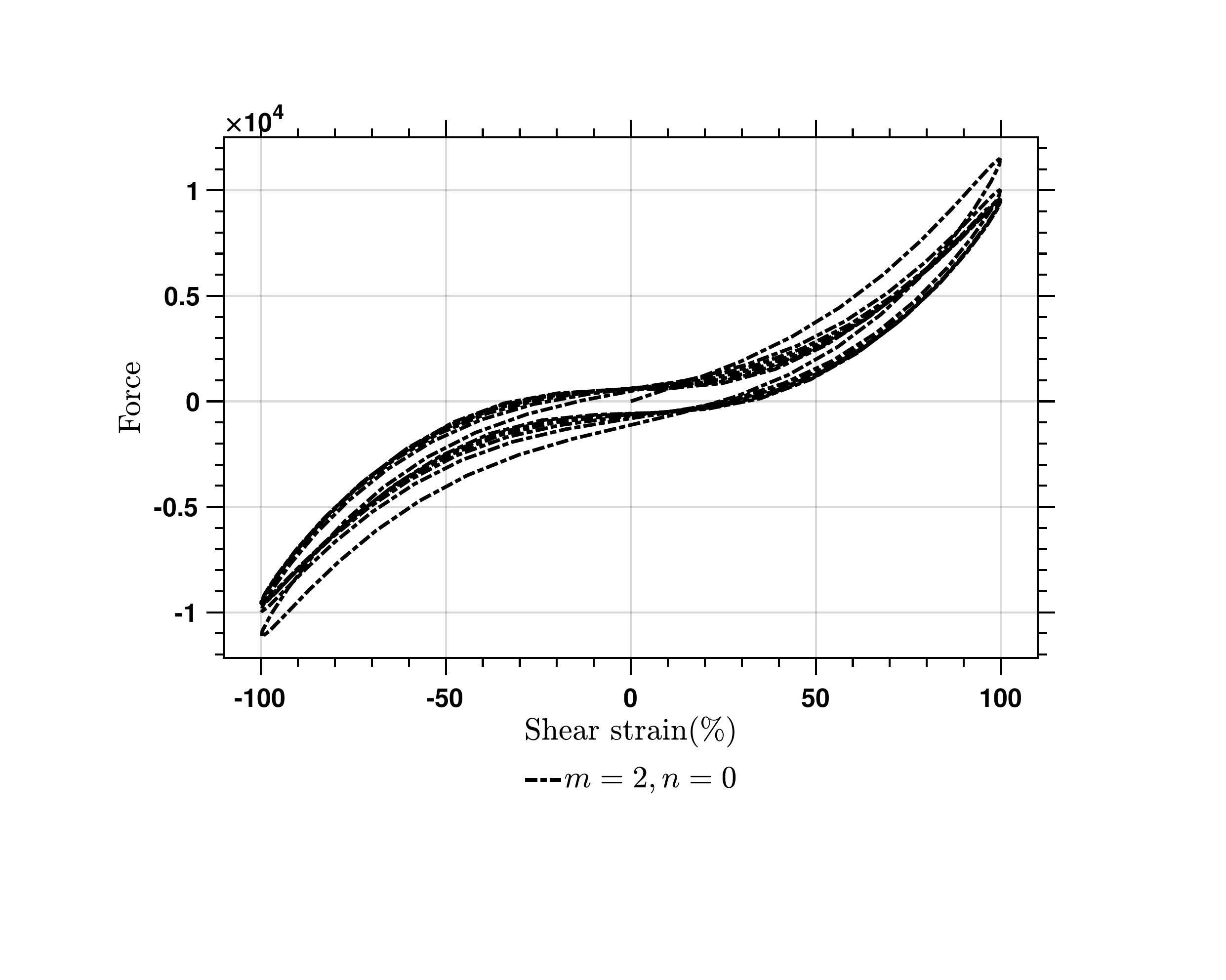} \\
			(c) & (d) \\
			\multicolumn{2}{c}{ \includegraphics[angle=0, trim=250 180 250 730, clip=true, scale = 0.22]{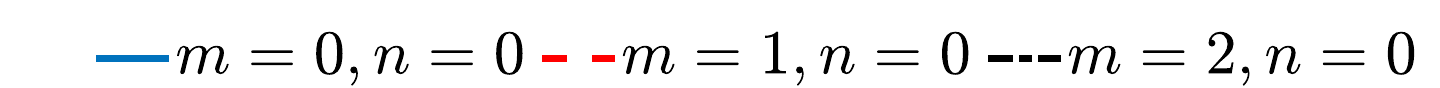} }
		\end{tabular}
	\end{center}
	\caption{Shear test: the hysteresis loops of different strains. The case with $U_0/L_0=1\%$ is illustrated in (a) for all three considered strains.}
	\label{fig:shear_paras}
\end{figure}

In this example, we investigate the model behavior under cyclic loads, and the problem setting is outlined in Table \ref{table:shear_test}. A set of two generalized strains are invoked to define the equilibrium part of the energy $G^{\infty}_{\mathrm{iso}}$, and their parameters are $m_1=2.0$, $n_1=0.0$ and $m_2=n_2=0.0$, respectively. A single set of parameter is utilized to define the two strains in the configurational free energy $\Upsilon^1$ and are denoted by $m$ and $n$ in this example. The block is clamped on the bottom and is subjected to a sinusoidal displacement loading $G_2(t) = U_0\sin(\omega t)$ along the $X_2$-direction on the top surface. The frequency is fixed to be $\omega = 0.3$. Stress-free boundary conditions are applied on the rest boundary surfaces. The shear strain on the top surface is given by $G_2(t)/L_0$ here, with $L_0 = 0.1$ being the edge length of the cube. We consider two cases with the maximum shear strains $U_0/L_0$ equaling to $1\%$ and $100\%$, repectively. Six cycles of loading are performed in each simulation to obtain the dynamic equilibrium state \cite[p.~3472]{Reese1998}. The time step size $\Delta t_n$ is fixed to be $0.1$. The spatial discretization is identical to that of the previous example.

We monitor the total traction on the top surface and plot the value of its component in the $X_2$-direction against the shear strain $G_2(t)/L_0$ in Figure \ref{fig:shear_paras}. The results illustrate the hysteresis loops corresponding to the two different loading parameters of $U_0/L_0$. In particular, we considered the strains $\tilde{\bm E}^1$ and $\bm E^{\mathrm v \: 1}$ parameterized by $(m,n)=(0,0)$, $(1,0)$, and $(2,0)$. When $U_0/L_0 = 100\%$, it is evident that the peak value of the force is proportional to the value of $m$. In contrast, when the shear is small (i.e., $U_0/L_0 = 1\%$), the influence of the parameter $m$ on the force is negligible. We recall that the parameter $m$ enters into the stress-strain relation as the power of the stretches, which may explain the aforesaid phenomena. Also, with the increase of the maximum shear strain, the strain rate gets stronger, and this drives the material to behave closer to a hyperelastic material. This observation matches with previous studies on the cyclic loading test of a viscoelastic body \cite{Reese1998}.

\subsection{Bearing}
\begin{table}[htbp]
  \centering 
  \begin{tabular}{cc}
    \hline \\
    \begin{minipage}{.45\textwidth}
      \begin{itemize}
        \item[] Rubber layers:
        \item[] $\rho_0 = 1.0\times 10^3$,
        \item[] $N = 2$, $M = 2$,
        \item[] $G_{\mathrm{iso}}^{\infty} = \mu_1^{\infty} \left| \tilde{\bm{E}}^{\infty}_1 \right|^2+\mu_2^{\infty} \left| \tilde{\bm{E}}^{\infty}_2 \right|^2$, 
        \item[] $\mu_1^{\infty} = 1.75 \times 10^5, \quad \mu_2^{\infty} = 0.35 \times 10^5$,
        \item[] $\Upsilon^1 =  \mu^1 \left| \tilde{\bm{E}}^1 - \bm{E}^{\mathrm{v}\:1}\right|^2$,
        \item[] $\Upsilon^2 =  \mu^2 \left| \tilde{\bm{E}}^2 - \bm{E}^{\mathrm{v}\:2}\right|^2$,
        \item[] $\tau^1 = 17.5, \quad \tau^2 = 17.5$,
        \item[] $\mu^1 = 0.536 \times 10^4, \quad \mu^2 = 5.3064 \times 10^5$.
      \end{itemize}
    \end{minipage}
    &
    \begin{minipage}{.45\textwidth}
      \begin{itemize}
        \item[] Steel layers:
        \item[] $G^{\mathrm{st}} = \mu^{\mathrm{st}} \left| \tilde{\bm{E}}^{\mathrm{st}} \right|^2$,
        \item[] $\mu^{\mathrm{st}}=  8.077 \times 10^{10}$,
        \item[] $m^{\mathrm{st}} = 2.0, n^{\mathrm{st}} = 2.0$.
      \end{itemize}
      \includegraphics[width=1.0\linewidth, trim=60 230 40 270, clip]{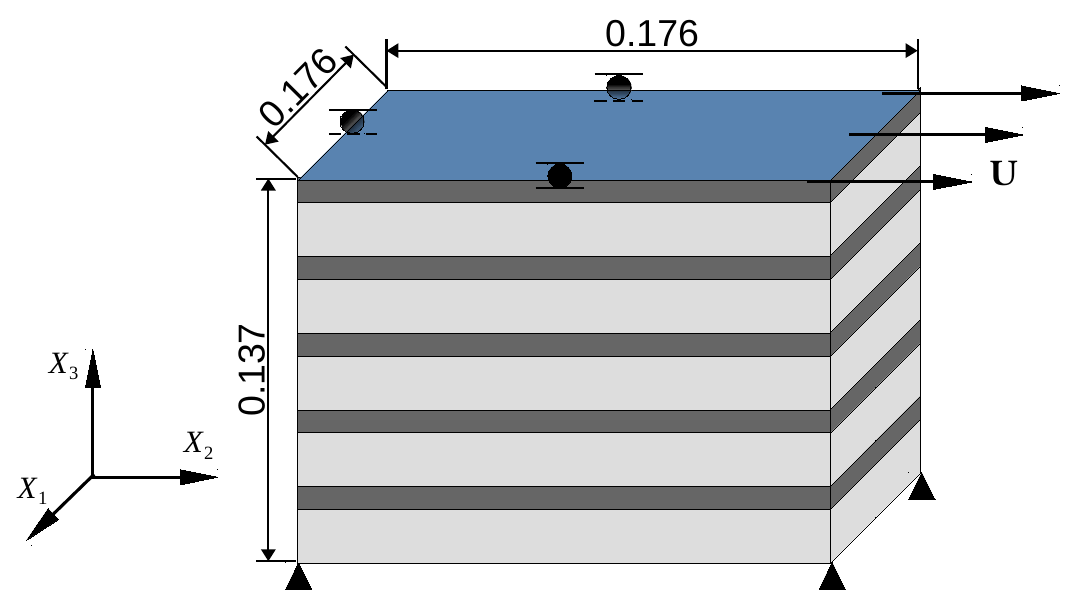}
    \end{minipage} \\ \\
    \hline
  \end{tabular}
  \caption{Bearing test: problem setting. The dark color represents the steel layer with height being 0.003; the light color represents the rubber layer with height being 0.025.}
\label{table:bearing test}
\end{table}

\begin{figure}
\begin{center}
\begin{tabular}{cc}
\multicolumn{2}{c}{ \includegraphics[angle=0, trim=170 180 160 950, clip=true, scale = 0.24]{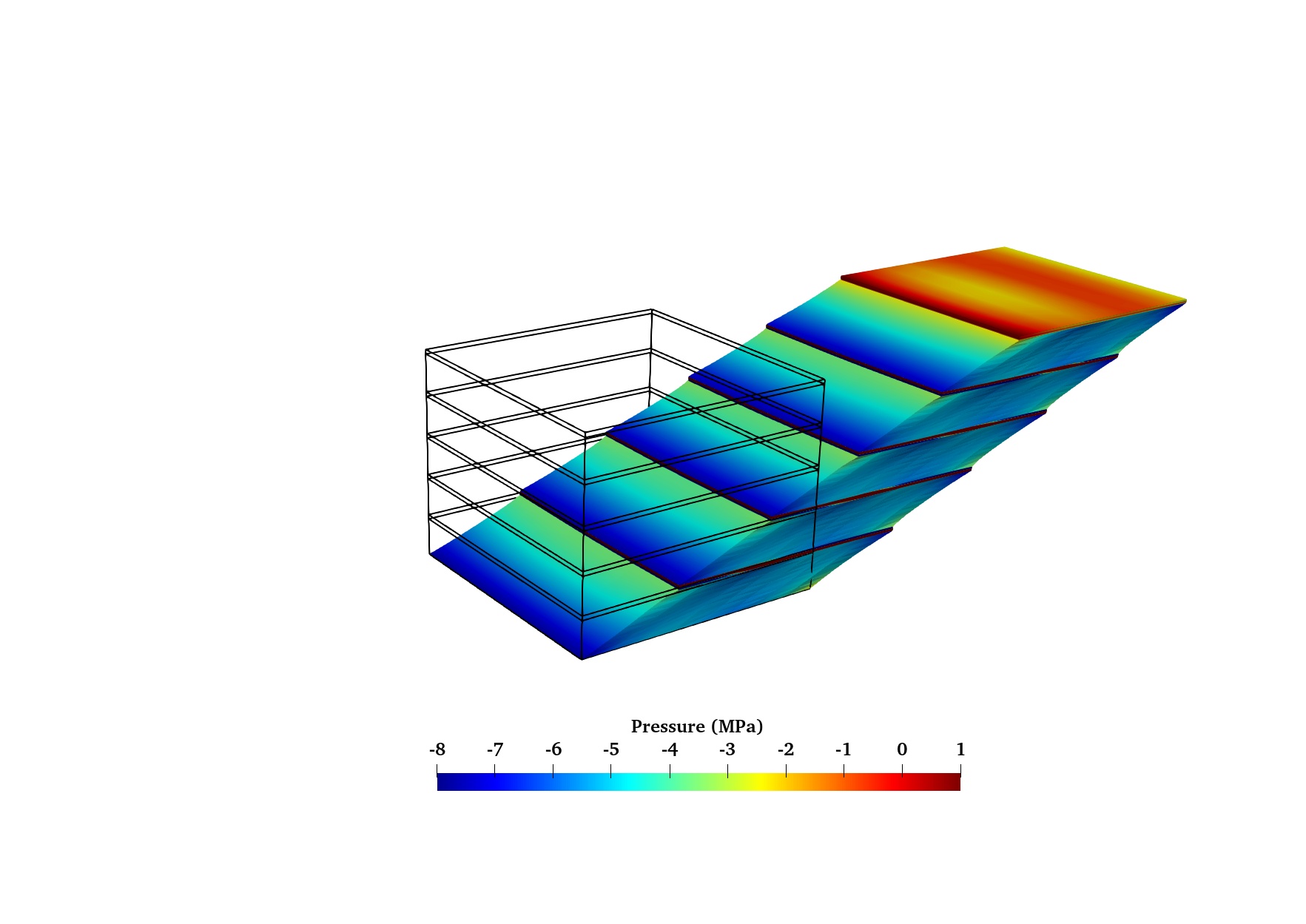} } \\
\includegraphics[angle=0, trim=550 350 150 320, clip=true, scale = 0.20]{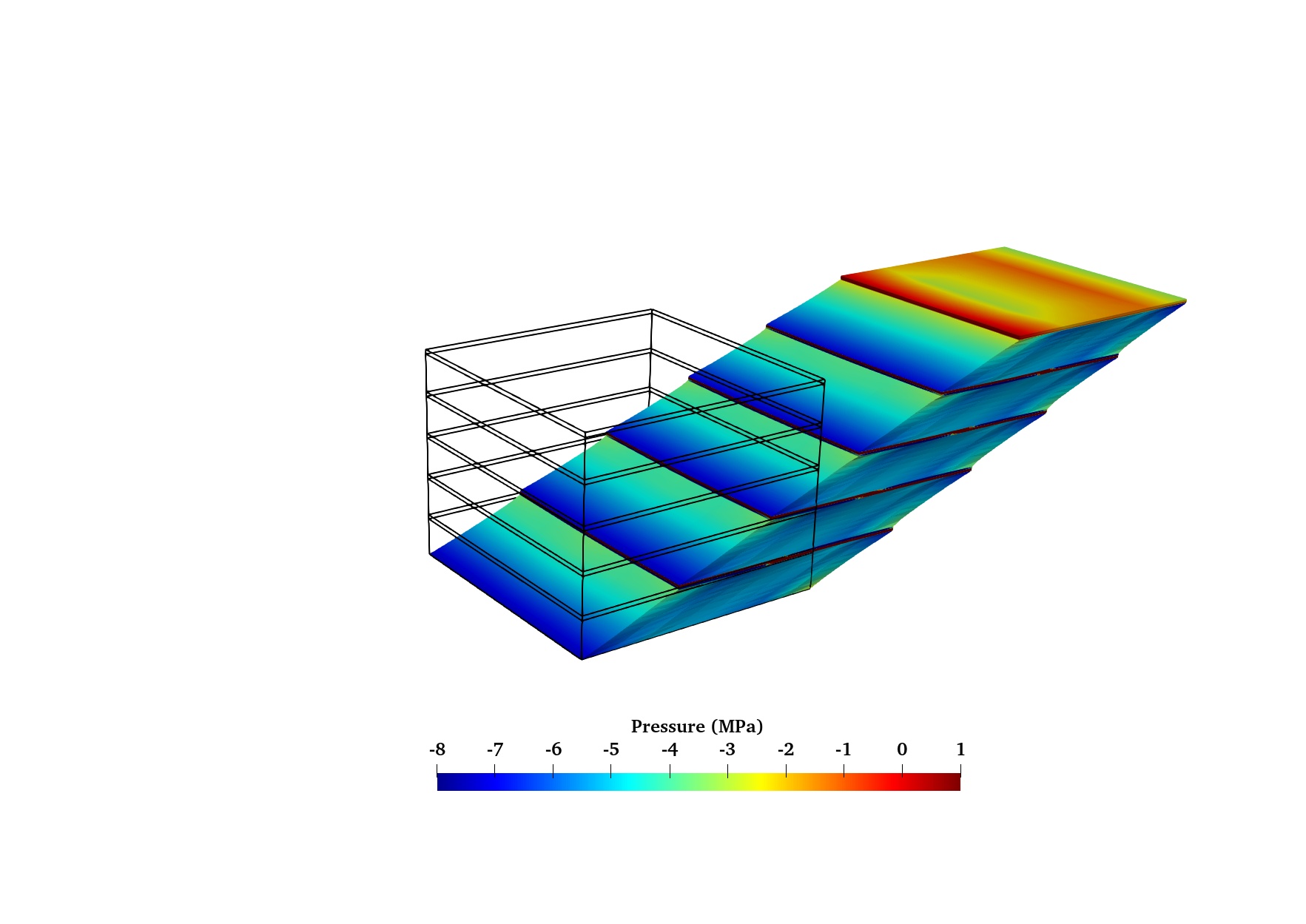} &
\includegraphics[angle=0, trim=550 350 150 320, clip=true, scale = 0.20]{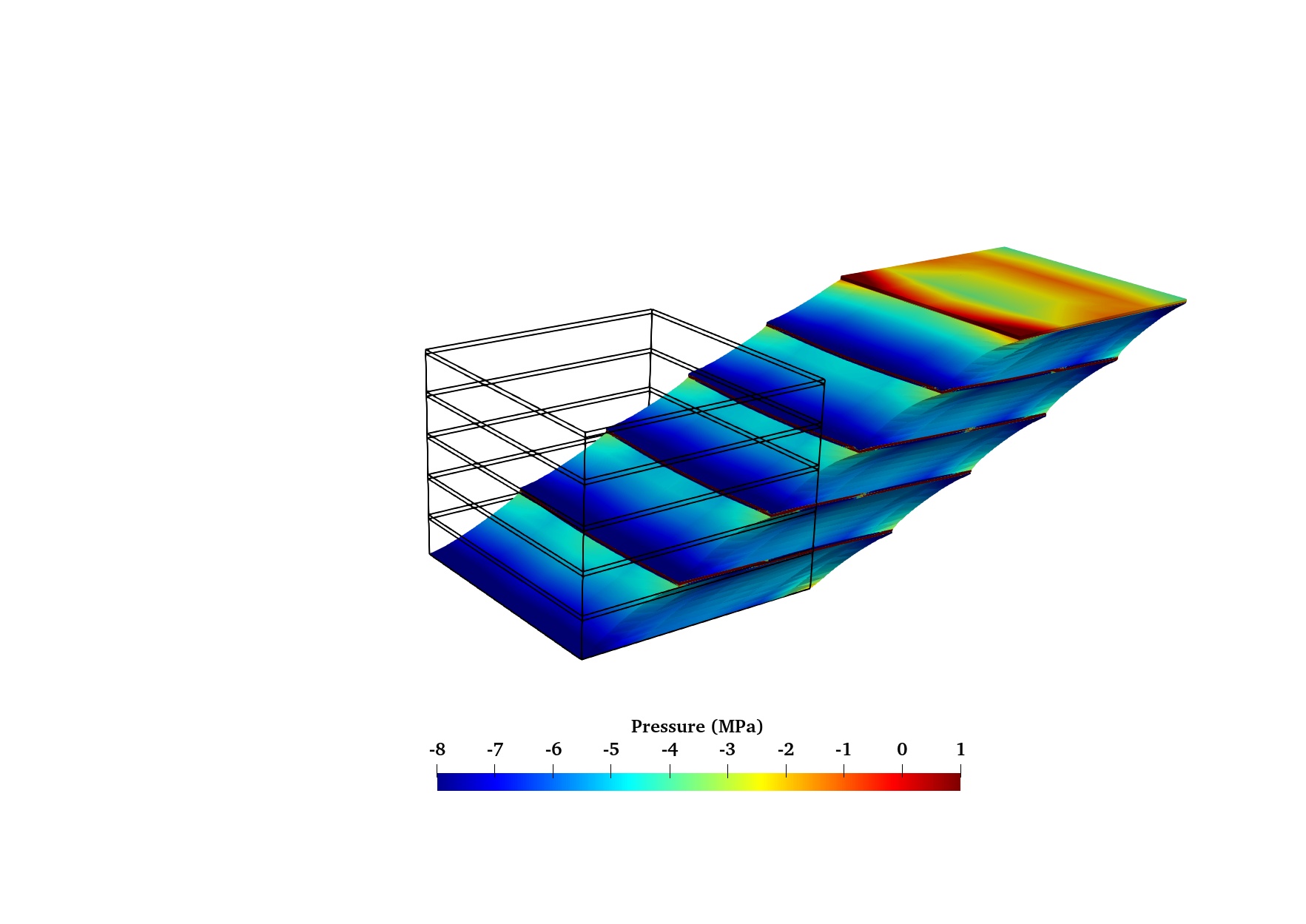} \\
$\omega = 0.03$ & $\omega = 3.0$
\end{tabular}
\end{center}
\caption{Bearing test: the pressure distributions of the bearing at the maximum shear state with two strain rates.}
\label{fig:bearing_deformation}
\end{figure}

In the final example, we explore the shearing behavior of a multilayered rubber-steel structure, which is used for seismic isolation. The problem setting is summarized in Table \ref{table:bearing test}. The steel layers are modeled by the generalized Hill's hyperelasticity model based on $\tilde{\bm E}^\mathrm{st}$, whose parameters are $m^\mathrm{st}=2.0$ and $n^\mathrm{st}=2.0$. As for the rubber layers, both the equilibrium and non-equilibrium parts are characterized by two quadratic energy terms. The generalized strains $\tilde{\bm E}^{\infty}_{1}$, $\tilde{\bm E}^{1}$, and $\bm E^{\mathrm v \: 1}$ are defined by the same scale function with parameters $m_1 = 2.0$ and $n_1 = 0.0$, while the generalized strains $\tilde{\bm E}^{\infty}_{2}$, $\tilde{\bm E}^{2}$, and $\bm E^{\mathrm v \: 2}$ are defined by the same scale function with parameters $m_2 = 0.0$ and $n_2 = 0.0$.

The bearing is fixed on the bottom surface and is subjected to a sinusoidal displacement loading along the $X_2$-direction on the top surface with $G_2(t) = 0.35 \mathrm{sin}(\omega t)$. Stress-free boundary conditions are applied on the rest boundary surfaces. The geometry of the bearing is represented by multi-patch NURBS with each layer constructed by a single patch. The presented results are based on a spatial mesh of $3 \times 3 \times 3$ elements for each patch with $\mathsf{p} =2$. The results have been verified by a mesh independence study. The frequencies $\omega$ under investigation take the values of $0.03$, $0.3$, and $3.0$, with the corresponding time step sizes of $0.1$, $0.01$, and $0.0001$, respectively. The simulations are conducted for two cycles for each simulation. 

Figure \ref{fig:bearing_deformation} depicts the pressure distributions when the multilayered structure reaches the maximum strain. As is shown in the figure, the higher values of $\omega$ results in a more intense pressure distribution. In Figure \ref{fig:bearing_force_displacement}, we illustrate the total force against the displacement in the $X_2$-direction on the top boundary surface. The curves obtained by three distinct meshes are plotted, demonstrating mesh independence. We observe that the increase of the frequency $\omega$ leads to a slight increase of the peak value of the hysteresis curve, consistent with the result of Figure \ref{fig:bearing_deformation}. Higher frequencies also cause the material to behave closer to hyperelasticity. These observations, again, align well with the prior studies on viscoelastic hysteresis \cite{Liu2021b,Reese1998}.

\begin{figure}
\begin{center}
\begin{tabular}{ccc}
\multicolumn{3}{c}{ \includegraphics[angle=0, trim=100 170 100 735, clip=true, scale = 0.22]{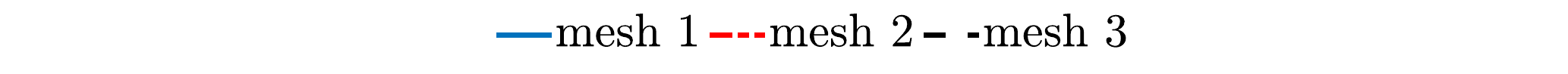} } \\
\includegraphics[angle=0, trim=100 210 120 100, clip=true, scale = 0.15]{./bearing_omega_0d03.pdf} &
\includegraphics[angle=0, trim=100 210 120 100, clip=true, scale = 0.15]{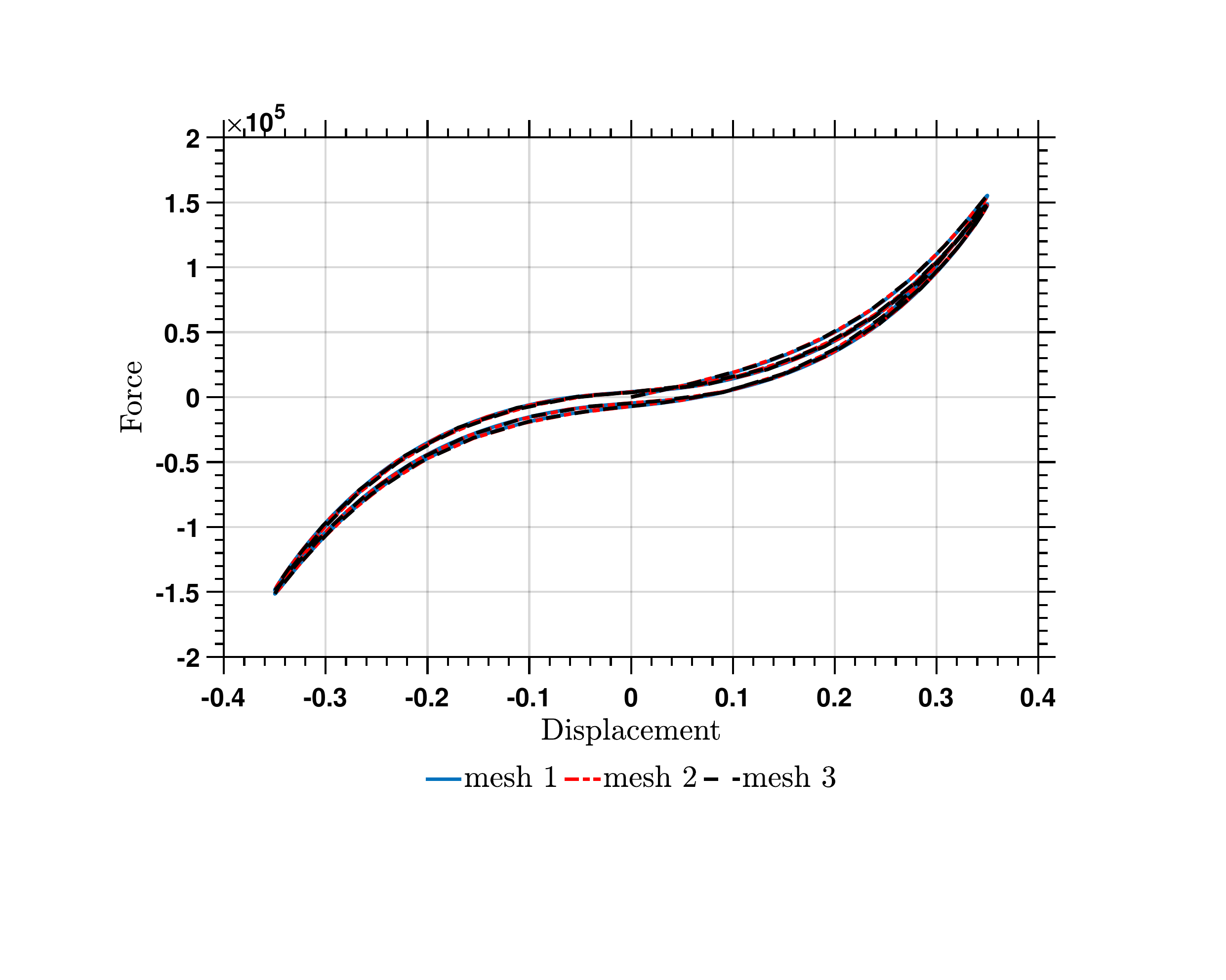} &
\includegraphics[angle=0, trim=100 210 120 100, clip=true, scale = 0.15]{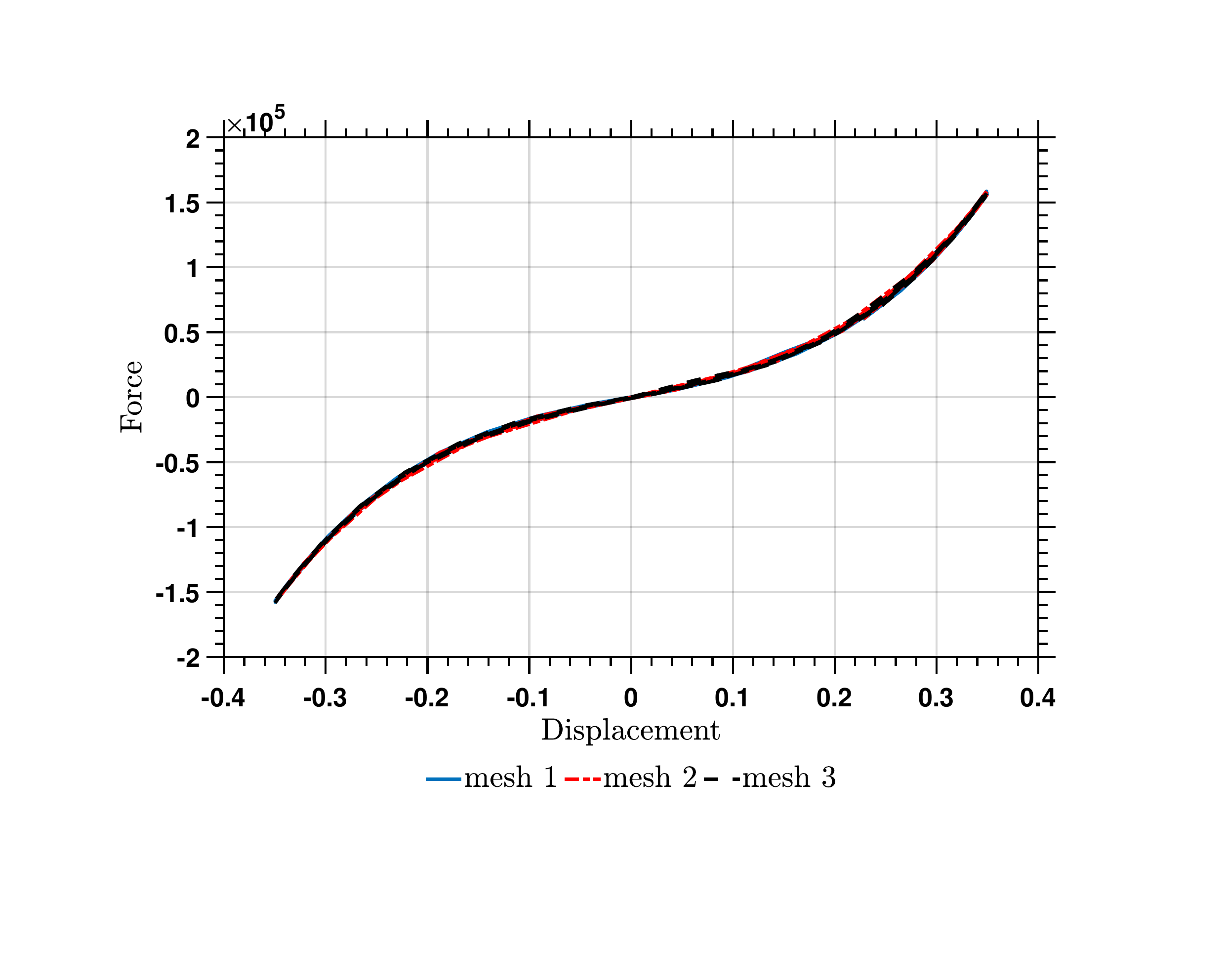} \\
$\omega = 0.03$ & $\omega=0.3$ & $\omega=3.0$
\end{tabular}
\end{center}
\caption{Bearing test: the force-displacement curve for different strain rates.}
\label{fig:bearing_force_displacement}
\end{figure}

\section{Conclusion}
\label{sec:conclusion}
In this work, we present a viscoelasticity model that extends the finite deformation linear viscoelastic model \cite{Simo1987,Holzapfel1996,Liu2021b} to the nonlinear regime. Our approach stems from an examination of the relaxation property of the non-equilibrium stress. It is found to be intricately linked to the kinematic assumption akin to the Green-Naghdi type developed in elastoplasticity. This inspires a set of kinematic assumptions tailored for viscoelasticity, from which a nonlinear theory is systematically constructed. To maintain consistency with the existing models, we adopt the hyperelasticity of Hill's class into our theory. Additionally, we constructed the theory based on the Helmholtz and Gibbs free energies, which are applicable to the pure displacement and mixed formulations, respectively. Within this framework, the nonlinear extension is achieved through generalized strains, and the previous linear model \cite{Liu2021b} gets recovered when the strain is instantiated to be of the Green-Lagrange type. This modeling framework differs from the theory based on the intermediate configuration in a fundamental manner. Since the internal state variable is a Lagrangian tensor belonging to $\mathrm{Sym}(3)_+$, there is no necessity to make assumptions regarding the rotation of the intermediate configuration. This feature can be particularly beneficial in treating anisotropy as one does not need to model the evolution of the symmetry groups. The benefit of the Green-Naghdi assumption is also manifested in guaranteeing the satisfaction of the relaxation property of the non-equilibrium stress. This is a rigorously established property for models built upon this assumption.

The numerical implementation of the model can be achieved in a modular approach, encompassing the geometric preprocessing, constitutive evaluation, and geometric postprocessing. Due to the nonlinearity, the evolution equation is integrated by the midpoint rule, and the internal state variables are determined locally at each quadrature point through a local Newton-Raphson iteration. Numerical examples, including creep, relaxation, and cyclic shear, are provided to illustrate the effectiveness of the model. We also considered a bearing example to showcase the practical applicability of the proposed model.

Our future work will focus on developing a nonlinear anisotropic viscoelasticity model. It is also imperative to consider integration methods that preserve critical physical features. This will be useful for the evaluating the long-term performance of viscoelastic devices and providing insights into their nonlinear behaviors.

\section*{Acknowledgements}
This work is supported by the National Natural Science Foundation of China [Grant Numbers 12172160], Shenzhen Science and Technology Program [Grant Number JCYJ20220818100600002], Southern University of Science and Technology [Grant Number Y01326127], and the Department of Science and Technology of Guangdong Province [2021QN020642]. Computational resources are provided by the Center for Computational Science and Engineering at the Southern University of Science and Technology.

\appendix
\section{Derivation of $\mathbb{Q}$}
\label{ap:Q}
We start by deriving an explicit formula for the derivative of $\bm{M}_a$ with respect to $\bm{C}$. According to \eqref{eq:C_spectral}, the eigenvalue problem for $\bm C$ reads as
\begin{align}
\label{eq:CN}
\bm{C} \bm{N}_a = \lambda_{a}^2\bm{N}_a   .
\end{align}
Taking derivative with respect to $\bm C$ at both sides of \eqref{eq:CN} leads to
\begin{align*}
\frac{\partial \left( \bm{C}\bm{N}_a \right) }{\partial \bm{C}} = \frac{\partial (\lambda_a^2 \bm{N}_a)}{\partial \bm{C}}, \displaybreak[2]
\end{align*}
which can be expanded as
\begin{align*}
 \bm{N}_a \cdot \frac{\partial \bm{C}^T}{\partial \bm{C}} + \bm{C} \frac{\partial \bm{N}_a}{\partial \bm{C}} = \bm{N}_a \otimes \frac{\partial \lambda_a^2}{\partial \bm{C}}  + \lambda_a^2 \frac{\partial \bm{N}_a}{\partial \bm{C}}.
\end{align*}
Performing the dot product of the above with the eigenvector $\bm N_b$, with $b \neq a$, results in
\begin{align}
\label{eq:N_CN_C}
\left( \bm{N}_b \otimes \bm{N}_a \right): \frac{\partial \bm{C}}{\partial \bm{C}} + \bm{N}_b \cdot \bm{C} \frac{\partial \bm{N}_a}{\partial \bm{C}} = \left( \bm{N}_b \cdot \bm{N}_a \right) \frac{\partial \lambda_a^2}{\partial \bm{C}}  + \lambda_a^2  \bm{N}_b \cdot \frac{\partial \bm{N}_a}{\partial \bm{C}} =  \lambda_a^2  \bm{N}_b \cdot \frac{\partial \bm{N}_a}{\partial \bm{C}}.
\end{align}
The second term on the left-hand side can be rewritten as
\begin{align}
  \label{eq:symm_C}
  \bm{N}_b \cdot \bm{C} \frac{\partial \bm{N}_a}{\partial \bm{C}}= \left( \bm{C}^T\bm{N}_b \right) \cdot \frac{\partial \bm{N}_a}{\partial \bm{C}} = \lambda_b^2 \bm{N}_b \cdot \frac{\partial \bm{N}_a}{\partial \bm{C}},
\end{align}
due to the symmetry of $\bm C$. With that, the relation \eqref{eq:N_CN_C} can be reorganized as
\begin{align}
  \label{eq:lambda_N_C_N}
\left(\lambda_a^2 -\lambda_b^2 \right) \bm{N}_b \cdot \frac{\partial \bm{N}_a}{\partial \bm{C}} = \left( \bm{N}_b \otimes \bm{N}_a \right): \frac{\partial \bm{C}}{\partial \bm{C}} = \frac{1}{2} \left( \bm{N}_a \otimes \bm{N}_b + \bm{N}_b \otimes\bm{N}_a \right),
\end{align}
which further leads to
\begin{align*}
  \bm{N}_b \cdot \frac{\partial \bm{N}_a}{\partial \bm{C}} = \frac{1}{2\left(\lambda_a^2 -\lambda_b^2 \right)} \left( \bm{N}_a \otimes \bm{N}_b + \bm{N}_b \otimes\bm{N}_a \right).
\end{align*}
From the above, the derivative of the eigenvector $\bm{N}_a$ with respect to $\bm{C}$ can be represented as 
\begin{align}
\label{eq:N_C}
\frac{\partial \bm{N}_a}{\partial \bm{C}} =\frac{1}{2}\sum_{b \neq a}^3 \frac{1}{\lambda_a^2 - \lambda_b^2} \bm{N}_b \otimes \left( \bm{N}_b \otimes \bm{N}_a + \bm{N}_a \otimes \bm{N}_b \right) , 
\end{align}
with  
\begin{align*}
  \bm{N}_a \cdot \frac{\partial \bm{N}_a}{\partial \bm{C}} = \frac{1}{2} \frac{\partial \left( \bm{N}_a \cdot \bm{N}_a \right) }{\partial \bm{C}} = \mathbb{O}.
\end{align*}
Employing the expression derived from \eqref{eq:N_C}, the derivative of $\bm{M}_a$ with respect to $\bm{C}$  can be written as
\begin{align}
  \label{eq:M_C}
\frac{\partial \bm{M}_a}{\partial \bm{C}} =& \frac{\partial \left( \bm{N}_a \otimes \bm{N}_a \right)}{\partial \bm{C}}  \nonumber \\ 
  = & \sum_{b \neq a}^3 \frac{1}{\lambda_a^2 - \lambda_b^2} \frac{1}{2} \left[ \bm{N}_b \otimes \bm{N}_a \otimes  \left( \bm{N}_b \otimes \bm{N}_a + \bm{N}_a \otimes \bm{N}_b \right) +  \bm{N}_a \otimes \bm{N}_b \otimes \left( \bm{N}_b \otimes \bm{N}_a + \bm{N}_a \otimes \bm{N}_b \right) \right] \nonumber  \displaybreak[2] \\ 
  = & \sum_{b \neq a}^3  \frac{1}{\lambda_a^2 - \lambda_b^2} \left( \bm{M}_b \odot \bm{M}_a + \bm{M}_a \odot \bm{M}_b\right) . 
\end{align}
With \eqref{eq:M_C}, we may obtain an explicit formula for $\mathbb Q$ as
\begin{align*}
  \mathbb{Q} = 2\frac{\partial \bm{E}}{\partial \bm{C}} =& \sum_{a = 1}^3 \left( 2 \frac{\partial E}{\partial (\lambda^2)}(\lambda_a) \bm{M}_a \otimes \frac{\partial \lambda_a^2}{\partial \bm{C}}\right) + \sum_{a = 1}^3 2 E(\lambda_a) \frac{\partial \bm{M}_a}{\partial \bm{C}} \nonumber \displaybreak[2] \\
  =& \sum_{a=1}^{3} d_{a} \bm{M}_a \otimes \bm{M}_a + \sum_{a=1}^3 \sum_{b\neq a}^3 \vartheta_{ab}  \bm M_a \odot \bm M_b,
\end{align*}
in which we invoked the relation \cite[p.~91]{Holzapfel2000}
\begin{align}
\label{eq:lambda_C}
\frac{\partial \lambda_a^2}{\partial \bm{C}} =\bm{M}_a,
\end{align}
and the terms $d_a$ and $\vartheta_{ab}$ are given by \eqref{eq:d_theta}.

\section{Derivation of $\mathbb{L}$}
\label{ap:L}
We start the derivation of the explicit formula for $\mathbb L$ by considering the second derivative of $\bm{M}_a$ with respect to $\bm{C}$. Following \eqref{eq:M_C}, we have
\begin{align}
\label{eq:M_CC_de}
\frac{\partial^2 \bm{M}_a}{\partial \bm{C} \partial \bm{C}} &=  \sum_{b \neq a}^3  \left( \bm{M}_b \odot \bm{M}_a + \bm{M}_a \odot \bm{M}_b \right) \otimes \frac{\partial }{\partial \bm{C}} \left( \frac{1}{\lambda_a^2 - \lambda_b^2} \right)   + \sum_{b \neq a}^3  \frac{1}{\lambda_a^2 - \lambda_b^2}   \frac{\partial }{\partial \bm{C}} \left( \bm{M}_b \odot \bm{M}_a + \bm{M}_a \odot \bm{M}_b\right) . 
\end{align}
The first term on the right-hand side can be written as
\begin{align}
  &  \sum_{b \neq a}^3  \left( \bm{M}_b \odot \bm{M}_a + \bm{M}_a \odot \bm{M}_b \right) \otimes \frac{\partial }{\partial \bm{C}} \left( \frac{1}{\lambda_a^2 - \lambda_b^2} \right)  \nonumber \displaybreak[2] \\ 
  & = \sum_{b \neq a}^3 \frac{1}{\left(\lambda_a^2 - \lambda_b^2\right)^2} \left( \bm{M}_b \odot \bm{M}_a + \bm{M}_a \odot \bm{M}_b \right) \otimes  \left(\frac{\partial \lambda_b^2}{\partial \bm{C}} - \frac{\partial \lambda_a^2}{\partial \bm{C}} \right)  \nonumber \displaybreak[2] \\ 
  & = \sum_{b \neq a}^3  \frac{1}{ 4 \left(\lambda_a^2 - \lambda_b^2 \right)^2} \left[  4 \left( \bm{M}_b \odot \bm{M}_a \right) \otimes \bm{M}_b + 4 \left( \bm{M}_a \odot \bm{M}_b \right)  \otimes \bm{M}_b  -  4  \left( \bm{M}_b \odot \bm{M}_a \right) \otimes \bm{M}_a - 4 \left( \bm{M}_a \odot \bm{M}_b  \right) \otimes \bm{M}_a  \right] \nonumber \displaybreak[2] \\ 
  \label{eq:M_CC_1}
  & = \sum_{b \neq a}^3  \frac{1}{ 4 \left(\lambda_a^2 - \lambda_b^2 \right)^2} \left( \mathbb{H}_{abb} - \mathbb{H}_{baa} \right).
\end{align}
In the last equality of the above, we have invoked the following identities
\begin{align}
  \mathbb{H}_{abb} =& 4 \left( \bm{M}_b \odot \bm{M}_a \right) \otimes \bm{M}_b + 4 \left( \bm{M}_a \odot \bm{M}_b \right) \otimes \bm{M}_b,   \nonumber  \displaybreak[2] \\ 
  \label{eq:Haab_Hbaa}
  \mathbb{H}_{baa} =& 4 \left( \bm{M}_b \odot \bm{M}_a \right)  \otimes \bm{M}_a + 4 \left( \bm{M}_a \odot \bm{M}_b \right)  \otimes \bm{M}_a,
\end{align}
due to the definition of $\mathbb H_{abc}$ given in \eqref{eq:H}. Next, we consider the derivative of $\bm{M}_a \odot \bm{M}_b$ with respect to $\bm{C}$. Through exploiting \eqref{eq:N_C}, it can be explicitly represented as
\begin{align*}
  &\sum_{b \neq a}^3  \frac{1}{\lambda_a^2 - \lambda_b^2}   \frac{\partial }{\partial \bm{C}} \left(  \bm{M}_a \odot \bm{M}_b\right)     \nonumber  \displaybreak[2] \\  
  =& \sum_{b \neq a}^3 \frac{1}{\left(\lambda_a^2 - \lambda_b^2 \right) } \left\{   \sum_{c \neq a}^3  \frac{1}{ 4 \left(\lambda_a^2 - \lambda_c^2 \right)}  \left[ \bm{N}_c \otimes \bm{N}_b \otimes \bm{N}_a \otimes \bm{N}_b \otimes \left(  \bm{N}_c \otimes \bm{N}_a + \bm{N}_a \otimes \bm{N}_c \right) \right] \right.  \nonumber \displaybreak[2] \\ 
  &  + \sum_{c \neq b}^3  \frac{1}{ 4 \left(\lambda_b^2 - \lambda_c^2 \right)}  \left[ \bm{N}_a \otimes \bm{N}_c \otimes \bm{N}_a \otimes \bm{N}_b \otimes \left(  \bm{N}_c \otimes \bm{N}_b + \bm{N}_b \otimes \bm{N}_c \right) \right]  \nonumber \displaybreak[2] \\ 
  &  + \sum_{c \neq a}^3  \frac{1}{ 4 \left(\lambda_a^2 - \lambda_c^2 \right)}  \left[ \bm{N}_a \otimes \bm{N}_b \otimes \bm{N}_c \otimes \bm{N}_b \otimes \left(  \bm{N}_c \otimes \bm{N}_a + \bm{N}_a \otimes \bm{N}_c \right) \right]  \nonumber \displaybreak[2] \\ 
  &  + \sum_{c \neq b}^3  \frac{1}{ 4 \left(\lambda_b^2 - \lambda_c^2 \right)}  \left[ \bm{N}_a \otimes \bm{N}_b \otimes \bm{N}_a \otimes \bm{N}_c \otimes \left(  \bm{N}_c \otimes \bm{N}_b + \bm{N}_b \otimes \bm{N}_c \right) \right]  \nonumber \displaybreak[2] \\ 
  &  + \sum_{c \neq a}^3  \frac{1}{ 4 \left(\lambda_a^2 - \lambda_c^2 \right)}  \left[ \bm{N}_c \otimes \bm{N}_b \otimes \bm{N}_b \otimes \bm{N}_a \otimes \left(  \bm{N}_c \otimes \bm{N}_a + \bm{N}_a \otimes \bm{N}_c \right) \right]  \nonumber \displaybreak[2] \\ 
  &  + \sum_{c \neq b}^3  \frac{1}{ 4 \left(\lambda_b^2 - \lambda_c^2 \right)}  \left[ \bm{N}_a \otimes \bm{N}_c \otimes \bm{N}_b \otimes \bm{N}_a \otimes \left(  \bm{N}_c \otimes \bm{N}_b + \bm{N}_b \otimes \bm{N}_c \right) \right]  \nonumber \displaybreak[2] \\ 
  &  + \sum_{c \neq b}^3  \frac{1}{ 4 \left(\lambda_b^2 - \lambda_c^2 \right)}  \left[ \bm{N}_a \otimes \bm{N}_b \otimes \bm{N}_c \otimes \bm{N}_a \otimes \left(  \bm{N}_c \otimes \bm{N}_b + \bm{N}_b \otimes \bm{N}_c \right) \right]  \nonumber \displaybreak[2] \\ 
  & \left. + \sum_{c \neq a}^3  \frac{1}{ 4 \left(\lambda_a^2 - \lambda_c^2 \right)}  \left[ \bm{N}_a \otimes \bm{N}_b \otimes \bm{N}_b \otimes \bm{N}_c \otimes \left(  \bm{N}_c \otimes \bm{N}_a + \bm{N}_a \otimes \bm{N}_c \right) \right] \right\}.
\end{align*} 
The above can be reorganized as
\begin{align}
  &\sum_{b \neq a}^3  \frac{1}{\lambda_a^2 - \lambda_b^2}   \frac{\partial }{\partial \bm{C}} \left( \bm{M}_a \odot \bm{M}_b \right) \nonumber \displaybreak[2] \\ 
  =&  \sum_{b \neq a }^3 \sum_{\substack{ c \neq a \\  c \neq b} }^3 \frac{1}{ 4 \left(\lambda_a^2 - \lambda_b^2 \right) \left(\lambda_a^2 - \lambda_c^2 \right)} \left[  \bm{N}_c \otimes \bm{N}_b \otimes \bm{N}_a \otimes \bm{N}_b \otimes \left(  \bm{N}_c \otimes \bm{N}_a + \bm{N}_a \otimes \bm{N}_c \right) \right.  \nonumber \displaybreak[2] \\ 
  & + \bm{N}_a \otimes \bm{N}_b \otimes \bm{N}_c \otimes \bm{N}_b \otimes \left(  \bm{N}_c \otimes \bm{N}_a + \bm{N}_a \otimes \bm{N}_c \right)   \nonumber \displaybreak[2] \\ 
  & + \bm{N}_c \otimes \bm{N}_b \otimes \bm{N}_b \otimes \bm{N}_a \otimes \left(  \bm{N}_c \otimes \bm{N}_a + \bm{N}_a \otimes \bm{N}_c \right)  \nonumber \displaybreak[2] \\ 
  & + \left. \bm{N}_a \otimes \bm{N}_b \otimes \bm{N}_b \otimes \bm{N}_c \otimes \left(  \bm{N}_c \otimes \bm{N}_a + \bm{N}_a \otimes \bm{N}_c \right) \right]  \nonumber \displaybreak[2] \\ 
  & + \sum_{b \neq a }^3 \sum_{\substack{ c \neq a \\  c \neq b} }^3 \frac{1}{ 4 \left(\lambda_a^2 - \lambda_b^2 \right) \left(\lambda_b^2 - \lambda_c^2 \right)} \left[  \bm{N}_a \otimes \bm{N}_c \otimes \bm{N}_a \otimes \bm{N}_b \otimes \left(  \bm{N}_c \otimes \bm{N}_b + \bm{N}_b \otimes \bm{N}_c \right) \right.  \nonumber \displaybreak[2] \\ 
  & + \bm{N}_a \otimes \bm{N}_b \otimes \bm{N}_a \otimes \bm{N}_c \otimes \left(  \bm{N}_c \otimes \bm{N}_b + \bm{N}_b \otimes \bm{N}_c \right)   \nonumber \displaybreak[2] \\ 
  & + \bm{N}_a \otimes \bm{N}_c \otimes \bm{N}_b \otimes \bm{N}_a \otimes \left(  \bm{N}_c \otimes \bm{N}_b + \bm{N}_b \otimes \bm{N}_c \right) \nonumber \displaybreak[2] \\ 
  & + \left. \bm{N}_a \otimes \bm{N}_b \otimes \bm{N}_c \otimes \bm{N}_a \otimes \left(  \bm{N}_c \otimes \bm{N}_b + \bm{N}_b \otimes \bm{N}_c \right) \right]  \nonumber \displaybreak[2] \\ 
  & + \sum_{b \neq a}^3  \frac{1}{ 4  \left(\lambda_a^2 - \lambda_b^2 \right)^2} \left[  \bm{N}_b \otimes \bm{N}_b \otimes \bm{N}_a \otimes \bm{N}_b \otimes \left(  \bm{N}_b \otimes \bm{N}_a + \bm{N}_a \otimes \bm{N}_b \right) \right.  \nonumber \displaybreak[2] \\ 
  & + \bm{N}_a \otimes \bm{N}_b \otimes \bm{N}_b \otimes \bm{N}_b \otimes \left(  \bm{N}_b \otimes \bm{N}_a + \bm{N}_a \otimes \bm{N}_b \right)   \nonumber \displaybreak[2] \\ 
  & + \bm{N}_b \otimes \bm{N}_b \otimes \bm{N}_b \otimes \bm{N}_a \otimes \left(  \bm{N}_b \otimes \bm{N}_a + \bm{N}_a \otimes \bm{N}_b \right)  \nonumber \displaybreak[2] \\ 
  & + \left.  \bm{N}_a \otimes \bm{N}_b \otimes \bm{N}_b \otimes \bm{N}_b \otimes \left(  \bm{N}_b \otimes \bm{N}_a + \bm{N}_a \otimes \bm{N}_b \right) \right]  \nonumber \displaybreak[2] \\  
  & - \left[  \bm{N}_a \otimes \bm{N}_a \otimes \bm{N}_a \otimes \bm{N}_b \otimes \left(  \bm{N}_a \otimes \bm{N}_b + \bm{N}_b \otimes \bm{N}_a \right) \right.  \nonumber \displaybreak[2] \\ 
  & + \bm{N}_a \otimes \bm{N}_b \otimes \bm{N}_a \otimes \bm{N}_a \otimes \left(  \bm{N}_a \otimes \bm{N}_b + \bm{N}_b \otimes \bm{N}_a \right)   \nonumber \displaybreak[2]\\ 
  & + \bm{N}_a \otimes \bm{N}_a \otimes \bm{N}_b \otimes \bm{N}_a \otimes \left(  \bm{N}_a \otimes \bm{N}_b + \bm{N}_b \otimes \bm{N}_a \right) \nonumber \displaybreak[2]\\ 
  \label{eq:odot_C_sep}
  & + \left. \bm{N}_a \otimes \bm{N}_b \otimes \bm{N}_a \otimes \bm{N}_a \otimes \left(  \bm{N}_a \otimes \bm{N}_b + \bm{N}_b \otimes \bm{N}_a \right) \right]. 
\end{align}
With the definition of $\mathbb H$ given in \eqref{eq:H} and \eqref{eq:odot_C_sep}, the second term on the right-hand side of \eqref{eq:M_CC_de} can be represented as
\begin{align}
  &\sum_{b \neq a}^3  \frac{1}{\lambda_a^2 - \lambda_b^2}   \frac{\partial }{\partial \bm{C}} \left( \bm{M}_b \odot \bm{M}_a + \bm{M}_a \odot \bm{M}_b\right)   \nonumber \displaybreak[2] \\ 
   = & \sum_{b \neq a }^3 \sum_{\substack{ c \neq a \\  c \neq b} }^3 \frac{1}{4 \left(\lambda_a^2 - \lambda_b^2 \right) \left(\lambda_a^2 - \lambda_c^2 \right)} \left[\mathbb{H}_{bca}  + \mathbb{H}_{bac} \right]  
    + \sum_{b \neq a }^3 \sum_{ \substack{ c \neq a \\  c \neq b} }^3\frac{1}{ 4 \left(\lambda_a^2 - \lambda_b^2 \right) \left(\lambda_b^2 - \lambda_c^2 \right)} \left[\mathbb{H}_{abc}  + \mathbb{H}_{acb} \right]   \nonumber \displaybreak[2] \\ 
   \label{eq:M_CC_2}
  & + \sum_{b \neq a}^3  \frac{1}{ 4 \left(\lambda_a^2 - \lambda_b^2 \right)^2} \left[ \left(\mathbb{H}_{bba}  + \mathbb{H}_{bab} \right) - \left(\mathbb{H}_{aab}  + \mathbb{H}_{aba} \right) \right]   . 
\end{align}
Combining \eqref{eq:M_CC_1} and \eqref{eq:M_CC_2} leads to
\begin{align}
  \frac{\partial^2 \bm{M}_a}{\partial \bm{C}\partial \bm{C}} =&\sum_{b \neq a}^3  \frac{1}{ 4 \left(\lambda_a^2 - \lambda_b^2 \right)^2} \left[  \left( \mathbb{H}_{bba}  + \mathbb{H}_{bab}  + \mathbb{H}_{abb} \right) - \left(\mathbb{H}_{aab}  + \mathbb{H}_{aba}  +  \mathbb{H}_{baa} \right) \right]  \nonumber  \displaybreak[2]\\ 
  \label{eq:M_CC}
  &+ \sum_{b \neq a }^3 \sum_{ \substack{ c \neq a \\  c \neq b} }^3\frac{1}{ 4 \left(\lambda_a^2 - \lambda_b^2 \right) \left(\lambda_a^2 - \lambda_c^2 \right)} \left[\mathbb{H}_{bca} + \mathbb{H}_{bac}  +\mathbb{H}_{abc}   \right],\displaybreak[2]
\end{align}
in which we have invoked the following identity
\begin{align*}
  \sum_{b \neq a }^3 \sum_{ \substack{ c \neq a \\  c \neq b} }^3 \frac{1}{ 4 \left(\lambda_a^2 - \lambda_b^2 \right) \left(\lambda_b^2 - \lambda_c^2 \right)} \left[\mathbb{H}_{abc}  + \mathbb{H}_{acb} \right] = \sum_{b \neq a}^3 \sum_{\substack{ c \neq a \\  c \neq b} }^3 \frac{1}{ 4 \left(\lambda_a^2 - \lambda_b^2 \right) \left(\lambda_a^2 - \lambda_c^2 \right)} \mathbb{H}_{abc} .
\end{align*}
According to \eqref{eq:M_CC} and \eqref{eq:lambda_C}, we may have $\mathbb{L}$ as
\begin{align*}
  \mathbb{L} = 4\frac{\partial^2 \bm{E}}{\partial \bm{C} \partial \bm{C}} =&  4 \frac{\partial }{\partial \bm{C}} \left[ \sum_{a = 1}^3 \left(  \frac{\partial E}{\partial (\lambda^2)}(\lambda_a) \bm{M}_a \otimes \bm{M}_a \right) + \sum_{a = 1}^3  E(\lambda_a) \frac{\partial \bm{M}_a}{\partial \bm{C}} \right] \nonumber \displaybreak[2]\\  
  =&  \sum_{a=1}^3 4 \frac{\partial^2 E}{\partial (\lambda^2) \partial (\lambda^2)} \left( \lambda_a \right)\bm{M}_a  \otimes \bm{M}_a  \otimes \bm{M}_a + \sum_{a=1}^3 4 \frac{\partial E}{\partial (\lambda^2)}(\lambda_a)  \frac{\partial }{\partial \bm{C}} \left( \bm{M}_a \otimes \bm{M}_a \right) \nonumber \displaybreak[2] \\ 
  & + \sum_{a=1}^3 4  \frac{\partial E}{\partial (\lambda^2)} \left( \lambda_a \right) \frac{\partial \bm{M}_a}{\partial \bm{C}} \otimes \bm{M}_a + \sum_{a=1}^3 4 E(\lambda_a) \frac{\partial^2 \bm{M}_a}{\partial \bm{C} \partial \bm{C}}. \displaybreak[2]
\end{align*}
Combining the results of \eqref{eq:d_theta}$_1$, \eqref{eq:fa}, \eqref{eq:H}, \eqref{eq:N_C}, \eqref{eq:M_C}, \eqref{eq:Haab_Hbaa} and \eqref{eq:M_CC}, $\mathbb{L}$ may be represented as
\begin{align}
\mathbb{L} =& \sum_{a=1}^3 f_a \bm{M}_a  \otimes \bm{M}_a  \otimes \bm{M}_a + \sum_{a=1}^3 \sum_{b\neq a}^3 \frac{d_a}{2 (\lambda_a^2 - \lambda_b^2)} \left( \mathbb{H}_{aba} + \mathbb{H}_{aab} + \mathbb{H}_{baa} \right)  \nonumber  \displaybreak[2]\\ 
& + \sum_{b \neq a}^3  \frac{ E(\lambda_a)}{  \left(\lambda_a^2 - \lambda_b^2 \right)^2} \left[  \left( \mathbb{H}_{bba}  + \mathbb{H}_{bab}  + \mathbb{H}_{abb} \right) - \left(\mathbb{H}_{aab}  + \mathbb{H}_{aba}  +  \mathbb{H}_{baa} \right) \right]  \nonumber \displaybreak[2] \\ 
\label{eq:L_H}
&+ \sum_{b \neq a }^3 \sum_{ \substack{ c \neq a \\  c \neq b} }^3\frac{ E(\lambda_a)}{  \left(\lambda_a^2 - \lambda_b^2 \right) \left(\lambda_a^2 - \lambda_c^2 \right)} \left[\mathbb{H}_{bca} + \mathbb{H}_{bac}  +\mathbb{H}_{abc}  \right]. \displaybreak[2]
\end{align}
In the above, we have utilized the following,
\begin{align*}
\frac{\partial }{\partial \bm{C}} \left( \bm{M}_a \otimes \bm{M}_a\right) = \frac{1}{4} \left( \mathbb{H}_{aba} + \mathbb{H}_{aab} \right),  \quad
\frac{\partial \bm{M}_a}{ \partial \bm{C}} = \frac{1}{4} \mathbb{H}_{baa}. 
\end{align*}
With  \eqref{eq:L_H}, we may obtain an explicit formula for $\mathbb{L}$  as
\begin{align*}
  \mathbb{L} &= \sum^{3}_{a=1} f_{a} \bm{M}_a\otimes \bm{M}_a\otimes \bm{M}_a + \sum^{3}_{a=1} \sum^{3}_{b \neq a} \xi_{ab} \left( \mathbb{H}_{abb} + \mathbb{H}_{bab} + \mathbb{H}_{bba} \right) + \sum^{3}_{a=1} \sum^{3}_{b \neq a} \sum^{3}_{\substack{c \neq a \\ c \neq b}} \eta \mathbb{H}_{abc} , 
\end{align*}
in which we invoked the following identities,
\begin{align}
\label{eq:d_lambda}
&\sum_{a = 1}^3 \sum_{b \neq a}^3 \frac{d_a }{ 2 \left(\lambda_a^2 - \lambda_b^2 \right)} \left(\mathbb{H}_{aba} + \mathbb{H}_{aab} + \mathbb{H}_{baa} \right)  =  - \sum_{a = 1}^3 \sum_{b \neq a}^3 \frac{d_b }{ 2 \left( \lambda_a^2 - \lambda_b^2 \right) } \left(\mathbb{H}_{abb} + \mathbb{H}_{bab} + \mathbb{H}_{bba} \right)  , \displaybreak[2]  \\ 
\label{eq:E_lambda}
& \sum_{b \neq a}^3  \frac{ E(\lambda_a)}{  \left(\lambda_a^2 - \lambda_b^2 \right)^2} \left[  \left( \mathbb{H}_{bba}  + \mathbb{H}_{bab}  + \mathbb{H}_{abb} \right) - \left(\mathbb{H}_{aab}  + \mathbb{H}_{aba}  +  \mathbb{H}_{baa} \right) \right] =  \sum_{a = 1}^3 \sum_{b \neq a}^3 \frac{\vartheta_{ab}}{ 2(\lambda_a^2 - \lambda_b^2)} \left(\mathbb{H}_{abb} + \mathbb{H}_{bab} + \mathbb{H}_{bba} \right) ,  \displaybreak[2] \\ 
\label{eq:E_lambda2}
& \sum_{b \neq a }^3 \sum_{ \substack{ c \neq a \\  c \neq b} }^3\frac{ E(\lambda_a)}{  \left(\lambda_a^2 - \lambda_b^2 \right) \left(\lambda_a^2 - \lambda_c^2 \right)} \left[\mathbb{H}_{bca} + \mathbb{H}_{bac}  +\mathbb{H}_{abc}  \right] \nonumber \\ 
&= \sum_{a = 1}^3 \sum_{b \neq a}^3 \sum_{\substack{ c \neq a \\  c \neq b} }^3  \left[\sum_{d = 1}^3 \sum_{e \neq d}^3 \sum_{\substack{ f \neq d \\  f \neq e} }^3 \frac{E(\lambda_d)}{2 \left(\lambda_d^2 - \lambda_e^2 \right) \left(\lambda_d^2 - \lambda_f^2 \right)} \right]\mathbb{H}_{abc}, \displaybreak[2]
\end{align}
and the terms $\xi_{ab}$ and $\eta$ are given by \eqref{eq:xiab} and \eqref{eq:eta}, respectively. 

\section{A discourse on the connection between the two kinematic assumptions}
\label{Connections between the two kinematic assumptions}
In this section, we briefly explore the relationship between the multiplicative and additive decompositions. In the assumption of the multiplicative decomposition, the deformation gradient is decomposed as $\bm{F} = \bm{F}^{\mathrm{e}} \bm{F}^{\mathrm{v}}$. This leads to the elastic and viscous deformation tensors as
\begin{align*}
\bm{C}^{\mathrm{e}} := \bm{F}^{\mathrm{e}\:T} \bm{F}^{\mathrm{e}} = \bm{F}^{\mathrm{v}\:-T} \bm{C} \bm{F}^{\mathrm{v}\:-1} \quad \mbox{and} \quad \bm{C}^{\mathrm{v}} := \bm{F}^{\mathrm{v}\:T} \bm{F}^{\mathrm{v}}.
\end{align*}
Similar to the definition \eqref{eq:C_spectral}, the two enjoy the following spectral decompositions
\begin{align*}
\bm{C}^{\mathrm{e}} := \sum_{a=1}^{3} \lambda^{\mathrm{e}\:2}_{a} \hat{\bm{N}}_{a} \otimes \hat{\bm{N}}_{a} \quad \mbox{and} \quad \bm{C}^{\mathrm{v}} := \sum_{a=1}^{3} \lambda^{\mathrm{v}\:2}_{a} \bm{N}_{a} \otimes \bm{N}_{a},
\end{align*}
in which, $\{\lambda_a^{\mathrm{e}} \}$, $\{\lambda_a^{\mathrm{v}} \}$, and $\{\hat{\bm{N}}_{a}\}$, for $a = 1, 2, 3$, represent the elastic principal stretches, viscous principal stretches, and principle directions related to the intermediate configuration, respectively. Following \eqref{eq:Hill_strain}, we define two generalized strains $\hat{\bm E}^{\mathrm{e}}$ and $\hat{\bm E}^{\mathrm{v}}$ as
\begin{align*}
\hat{\bm E}^{\mathrm{e}}:= \sum_{a=1}^{3} E (\lambda_a^{\mathrm{e}}) \hat{\bm{N}}_a \otimes \hat{\bm{N}}_a \quad \mbox{and} \quad \hat{\bm E}^{\mathrm{v}}:= \sum_{a=1}^{3} E^{\mathrm v} (\lambda_a^{\mathrm{v}}) \bm{N}_a \otimes \bm{N}_a.
\end{align*}
In the above, we use hats to distinguish them from their counterparts in the additive theory. In the following, we investigate the connections between the two kinematic assumptions. The first result is related to the equivalence of the strain invariants when the strain is of the Green-Lagrange type.

\begin{proposition} 
\label{prop:basic-invariants}
If we interpret the internal state variable $\bm{\Gamma}$ as $\bm{C}^{\mathrm{v}}$, the basic invariants of $\hat{\bm{E}}^{\mathrm{e}\:(2)}_{\mathrm{SH}}$ and $\bm E^{\mathrm{e}\:(2)}_{\mathrm{SH}} \bm{\Gamma}^{-1}$ coincide, i.e.
\begin{align*}
J_{i} \left( \hat{\bm{E}}^{\mathrm{e}\:(2)}_{\mathrm{SH}} \right) = J_{i} \left(\bm{E}^{\mathrm{e}\:(2)}_{\mathrm{SH}} \bm \Gamma^{-1} \right), \quad \mbox{for} \quad i = 1, 2, 3.
\end{align*}
In the above, the basic invariants are defined as the traces of tensor powers, i.e., $J_{i} \left( \cdot \right) := \mathrm{tr} \left( ( \cdot )^{i} \right)$, where $(\cdot )$ is a rank-two tensor. 
\end{proposition}

\begin{proof}
The first basic invariant of $\hat{\bm{E}}^{\mathrm{e}\:(2)}_{\mathrm{SH}}$ is
\begin{align*}
J_{1} \left( \hat{\bm{E}}^{\mathrm{e}\:(2)}_{\mathrm{SH}} \right) &= \mathrm{tr} \left( \hat{\bm{E}}^{\mathrm{e}\:(2)}_{\mathrm{SH}} \right) = \frac{1}{2} \left( \mathrm{tr} \left( \bm{F}^{\mathrm{v}\:-T} \bm{C} \bm{F}^{\mathrm{v}\:-1} \right) - 3 \right) = \frac{1}{2} \left( \mathrm{tr}\left( \bm{C} \bm{C}^{\mathrm{v}\:-1}\right) - 3 \right).
\end{align*}
With the interpretation of $\bm{\Gamma}$ as $\bm{C}^{\mathrm{v}}$, we have
\begin{align*}
\bm{E}_{\mathrm{e}\:\mathrm{SH}}^{(2)} \bm{C}^{\mathrm{v}\:-1} = \frac{1}{2} \left( \bm{C} \bm{C}^{\mathrm{v}\:-1} - \bm{\Gamma} \bm{C}^{\mathrm{v}\:-1} \right) = \frac{1}{2} \left( \bm{C} \bm{C}^{\mathrm{v}\:-1} - \bm{I} \right).
\end{align*}
Subsequently, we have
\begin{align*}
J_{1}\left(\bm{E}_{\mathrm{e}\:\mathrm{SH}}^{(2)} \bm{C}^{\mathrm{v}\:-1} \right) = \mathrm{tr} \left(\bm{E}_{\mathrm{e}\:\mathrm{SH}}^{(2)} \bm{C}^{\mathrm{v}\:-1} \right) = \frac12 \left( \mathrm{tr} \left( \bm{C} \bm{C}^{\mathrm{v}\:-1} \right) - 3 \right) = J_{1} \left( \hat{\bm{E}}^{\mathrm{e}\:(2)}_{\mathrm{SH}} \right).
\end{align*}
From the derivation above, we establish the identity $\mathrm{tr}\left( \bm{C} \bm{C}^{\mathrm{v}\:-1} \right) = \mathrm{tr}\left( \bm{C}^{\mathrm{e}} \right)$. After straightforward computations, we may further obtain
\begin{align*}
\mathrm{tr}\left( \left( \bm{C} \bm{C}^{\mathrm{v}\:-1} \right)^2 \right) &= \left( \bm{C} \bm{C}^{\mathrm{v}\:-1} \right) : \left( \bm{C}^{\mathrm{v}\:-1} \bm{C} \right) \displaybreak[2] \\  
&= \left( \bm{F}^{\mathrm{v}\:-T} \bm{C} \bm{F}^{\mathrm{v}\:-1} \right) : \left( \bm{F}^{\mathrm{v}\:-T} \bm{C} \bm{F}^{\mathrm{v}\:-1} \right) \displaybreak[2] \\
&= \bm{C}^{\mathrm{e}} : \bm{C}^{\mathrm{e}} \displaybreak[2] \\
&= \mathrm{tr} \left( \bm{C}^{\mathrm{e}\:2} \right)
\end{align*}
and
\begin{align*}
\mathrm{tr}\left( \left( \bm{C} \bm{C}^{\mathrm{v}\:-1} \right)^3 \right) &= \left( \bm{C} \bm{C}^{\mathrm{v}\:-1} \bm{C} \bm{C}^{\mathrm{v}\:-1} \right) : \left( \bm{C}^{\mathrm{v}\:-1} \bm{C} \right) \displaybreak[2] \\
&= \left( \bm{F}^{\mathrm{v}\:-T} \bm{C} \bm{C}^{\mathrm{v}\:-1} \bm{C} \bm{F}^{\mathrm{v}\:-1} \right) : \left( \bm{F}^{\mathrm{v}\:-T} \bm{C} \bm{F}^{\mathrm{v}\:-1} \right) \displaybreak[2] \\
&= \bm{C}^{\mathrm{e}\:2} : \bm{C}^{\mathrm{e}} \displaybreak[2] \\
&= \mathrm{tr} \left( \bm{C}^{\mathrm{e}\:3} \right).
\end{align*}
Regarding the second and third basic invariants, we have
\begin{align*}
J_{2} \left( \hat{\bm{E}}^{\mathrm{e}\:(2)}_{\mathrm{SH}} \right) = \mathrm{tr} \left( \left( \hat{\bm{E}}^{\mathrm{e}\:(2)}_{\mathrm{SH}} \right)^2  \right) = \frac{1}{4} \left( \mathrm{tr} \left( \left(\bm{C} \bm{C}^{\mathrm{v}\:-1} \right)^2  \right) - 2\mathrm{tr}\left(\bm{C} \bm{C}^{\mathrm{v}\:-1} \right) + 3 \right),
\end{align*}
and
\begin{align*}
J_{3} \left( \hat{\bm{E}}^{\mathrm{e}\:(2)}_{\mathrm{SH}} \right) = \mathrm{tr} \left( \left( \hat{\bm{E}}^{\mathrm{e}\:(2)}_{\mathrm{SH}} \right)^3 \right) = \frac{1}{8} \left( \mathrm{tr} \left( \left(\bm{C} \bm{C}^{\mathrm{v}\:-1} \right)^3  \right) - 3\mathrm{tr} \left( \left(\bm{C} \bm{C}^{\mathrm{v}\:-1} \right)^2 \right) + 3\mathrm{tr}\left(\bm{C} \bm{C}^{\mathrm{v}\:-1} \right) - 3 \right).
\end{align*}
With the following relations 
\begin{align*}
\bm{C} \bm{C}^{\mathrm{v}\:-1} =& 2 \bm{E}_{\mathrm{SH}}^{(2)} \bm{C}^{\mathrm{v}\:-1} + \bm{I}, \displaybreak[2] \\
\left( \bm{C} \bm{C}^{\mathrm{v}\:-1} \right)^2 =& 4 \left( \bm{E}_{\mathrm{SH}}^{(2)} \bm{C}^{\mathrm{v}\:-1} \right)^2 + 4 \bm{E}_{\mathrm{SH}}^{(2)} \bm{C}^{\mathrm{v}\:-1} + \bm{I}, \displaybreak[2] \\
\left( \bm{C} \bm{C}^{\mathrm{v}\:-1} \right)^3 =& 8 \left( \bm{E}_{\mathrm{SH}}^{(2)} \bm{C}^{\mathrm{v}\:-1} \right)^3 + 12 \left( \bm{E}_{\mathrm{SH}}^{(2)} \bm{C}^{\mathrm{v}\:-1} \right)^2 + 6 \bm{E}_{\mathrm{SH}}^{(2)} \bm{C}^{\mathrm{v}\:-1} + \bm{I},
\end{align*}
we get the results regarding the second and third basic invariants,
\begin{align*}
J_{2} \left( \hat{\bm{E}}^{\mathrm{e}\:(2)}_{\mathrm{SH}} \right) =& \mathrm{tr} \left( \left( \bm{E}_{\mathrm{SH}}^{(2)} \bm{C}^{\mathrm{v}\:-1} \right)^{2} \right) = J_{2} \left( \bm{E}_{\mathrm{SH}}^{(2)} \bm{C}^{\mathrm{v}\:-1} \right), \displaybreak[2] \\
J_{3} \left( \hat{\bm{E}}^{\mathrm{e}\:(2)}_{\mathrm{SH}} \right) =& \mathrm{tr} \left( \left( \bm{E}_{\mathrm{SH}}^{(2)} \bm{C}^{\mathrm{v}\:-1} \right)^{3} \right) = J_{3} \left( \bm{E}_{\mathrm{SH}}^{(2)} \bm{C}^{\mathrm{v}\:-1}\right),
\end{align*}
which completes the proof.
\end{proof}

In the second result, we consider the Hencky strain and assume the deformation gradient $\bm F^{\mathrm v}$ is a pure stretch tensor and coaxial with $\bm C$. This assumption suggests that $\bm F^{\mathrm v}$ commutes with $\bm C$, which eventually leads to an additive split structure of the strain within the multiplicative decomposition framework. 

\begin{proposition}
If $\bm F^{\mathrm{v}}$ is a pure stretch tensor coaxial with $\bm C$, we have the following additive split structure of strains
\begin{align*}
\bm E^{(0)}_{\mathrm{SH}} = \hat{\bm{E}}^{\mathrm{e}\:(0)}_{\mathrm{SH}} + \hat{\bm{E}}^{\mathrm{v}\:(0)}_{\mathrm{SH}}.
\end{align*} 
\end{proposition}
\begin{proof}
Given the multiplicative decomposition $\bm F = \bm{F}^{\mathrm{e}} \bm{F}^{\mathrm{v}}$, the elastic Hencky strain is given by
\begin{align*}
\hat{\bm{E}}^{\mathrm{e}\:(0)}_{\mathrm{SH}} = \frac12 \mathrm{ln} \left( \bm{C}^{\mathrm{e}} \right) = \frac12 \mathrm{ln} \left( \bm{F}^{\mathrm{v}\:-T} \bm{C} \bm{F}^{\mathrm{v}\:-1} \right).
\end{align*}
Since $\bm{F}^{\mathrm{v}}$ is a pure stretch tensor coaxial with $\bm{C}$, we have
\begin{align*}
\hat{\bm{E}}^{\mathrm{e}\:(0)}_{\mathrm{SH}} = \frac12 \mathrm{ln} \left( \bm{C} \bm{F}^{\mathrm{v}\:-T} \bm{F}^{\mathrm{v}\:-1} \right)= \frac12 \sum_{a=1}^3 \mathrm{ln} \left( \lambda_a^2 \lambda^{\mathrm{v}\:-2}_{a} \right) \bm{N}_a \otimes \bm{N}_a = \bm{E}_{\mathrm{SH}}^{(0)} - \hat{\bm{E}}^{\mathrm{v}\:(0)}_{\mathrm{SH}},
\end{align*}
which completes the proof.
\end{proof}

\section{Glossary of terms}
\label{ap:Glossary}
\renewcommand{\arraystretch}{1.3}
\begin{longtable}{p{3.5cm} p{6cm} p{3.5cm}}
\hline
Symbol & Name or description & Place of definition or first occurrence\\
\hline
$\bm{E}$ & The generalized strain & \eqref{eq:Hill_strain}\\
$\tilde{\bm{E}}$ & The generalized strain associated with isochoric deformation & \eqref{eq:iso_Hill_strain} \\
$\mathbb{Q}$, $\mathbb{L}$ & The first and second derivatives of the strain $\bm{E}$ with respect to $\bm{C}$ & \eqref{eq:Q}, \eqref{eq:L}\\
$\tilde{\mathbb{Q}}$, $\tilde{\mathbb{L}}$, & The first and second derivatives of the strain $\tilde{\bm{E}}$ with respect to $\tilde{\bm{C}}$ & \eqref{eq:tilde_Q_L}\\
$\bm{\Gamma}$, $\bm Q$, $\bm E^{\mathrm{v}}$, & The deformation-like internal state variable, the stress-like internal state variable, and the viscous strain& \eqref{eq:Gamma}, \eqref{eq:def_Q}, \eqref{eq:Ev-def}\\
$\mathbb{Q}^{\mathrm{v}}$, $\mathbb{L}^{\mathrm{v}}$, &The first and second derivatives of the strain $\bm E^{\mathrm{v}}$ with respect to $\bm{\Gamma}$ & \eqref{eq:QQ_LL_v}\\
$\Psi$, $\Psi^{\infty}$, $\Upsilon$ & The Helmholz free energy, its equilibrium part, and the configurational free energy & \eqref{eq:Helmholtz_eq_neq}\\
$\Psi_{\mathrm{vol}}^{\infty}$, $\Psi_{\mathrm{iso}}$, $\Psi^{\infty}_{\mathrm{iso}}$ & The volumetric part of $\Psi$, the isochoric part of $\Psi$, and the isochoric part of $\Psi^{\infty}$ & \eqref{eq:vol_iso-Helmholtz-free-energy}, \eqref{eq:decomposition_of_Psi_iso} \\
$G$, $G^{\infty}_{\mathrm{vol}}$, $G_{\mathrm{iso}}$, $G^{\infty}_{\mathrm{iso}}$ & The Gibbs free energy, its volumetric part, its isochoric part, and the equilibrium part of  $G_{\mathrm{iso}}$ & \eqref{eq:Legendre_transformation}, \eqref{eq:Gibbs} \\
$\mathcal D$ & The internal dissipation & \eqref{eq:Clausius_Plank_inequality}\\
$\mathbb V$ & The viscosity tensor & \eqref{eq:constitutive_S_Q}\\
 $\mu^{\infty}$, $\mu^{\mathrm{neq}}$ & The shear moduli for the equilibrium part and non-equilibrium part & \eqref{eq:G_iso_inf_Upsilon}\\
$\bm S$, $\bm S^{\infty}$, $\bm S^{\mathrm{neq}}$ & The second Piola-Kirchhoff stress, its equilibrium part, and its non-equilibrium part & \eqref{eq:constitutive_S_Q}, \eqref{eq:def_S_T_eq_neq}$_{1}$, \eqref{eq:def_S_T_eq_neq}$_{2}$ \\
$\bm{T}^{\infty}$, $\bm{T}^{\mathrm{neq}}$ & Stresses conjugate to the generalized strain $\bm{E}$ with respect to $\Psi^{\infty}$ and $\Upsilon$, respectively & \eqref{eq:def_S_T_eq_neq}$_{3}$, \eqref{eq:def_S_T_eq_neq}$_{4}$ \\
$\tilde{\bm{S}}^{\infty}_{\mathrm{iso}}$, $\tilde{\bm{S}}_{\mathrm{iso}}^{\mathrm{neq}}$ & Fictitious stresses & \eqref{eq:tilde_stress}$_1$, \eqref{eq:tilde_stress}$_{2}$ \\
$\tilde{\bm{T}}^{\infty}$, $\tilde{\bm{T}}^{\mathrm{neq}} $ & The stress conjugate to the generalized strain $\tilde{\bm{E}}$ with respect to $G^{\infty}_{\mathrm{iso}}$ and $\Upsilon$, respectively & \eqref{eq:tilde_T_equilibrated} \\
$\mathbb{C}_{\mathrm{iso}}$, $\mathbb{C}_{\mathrm{iso}}^{\infty}$, $\mathbb{C}_{\mathrm{iso}}^{\mathrm{neq}}$ & The isochoric elasticity tensor, its equilibrium part, and non-equilibrium part & \eqref{eq:definition_CC_iso} \\
$\tilde{\mathbb{C}}_{\mathrm{iso} }^{\infty}$, $\tilde{\mathbb{C}}_{\mathrm{iso}}^{\mathrm{neq}}$ & The fictitious elasticity tensors & \eqref{eq:tilde_CC_iso_equi}$_{1}$, \eqref{eq:elements_in_CC_neq}$_{1}$ \\
$\Psi^{\infty}_{\beta}$ & The $\beta$-th quadratic term of $\Psi^{\infty}$ & \eqref{eq:Hill_Psi_Upsilon}$_1$\\
$\bm{E}^{\infty}_{\beta}$ & The generalized strain used to define the quadratic energy $\Psi^{\infty}_{\beta}$ & \eqref{eq:Hill_Psi_Upsilon} \\
$\bm{T}^{\infty}_{\beta}$ & The stresses conjugate to $\bm{E}_{\beta}$ with respect to $\Psi^{\infty}_{\beta}$ & \eqref{eq:T_eq_neq_beta}$_1$ \\
$\mathbb{Q}_{\beta}^{\infty}$ & The first derivative of the strain $\bm{E}^{\infty}_{\beta}$ with respect of $\bm{C}$ & \eqref{eq:Q_inf_beta_alpha}$_{1}$ \\
$\bm{\Gamma}^{\alpha}$, $\bm Q^{\alpha}$, $\mathbb{V}^{\alpha}$ & The internal state variables and viscosity tensors associated with the $\alpha$-th relaxation processes & Section \eqref{sec:Kinematic assumption}, \eqref{eq:gen_S_inf_S_neq_Q_alpha}$_3$, \eqref{eq:Q_alpha_V_alpha_evo_eqn}$_1$\\
$\Upsilon^{\alpha}$ & The configurational free energy associated with the $\alpha$-th relaxation process & \eqref{eq:Hill_Psi_Upsilon}$_2$ \\
$\bm{E}^{\alpha}$, $\bm{E}^{\mathrm{v} \: \alpha}$ & The generalized strains that defines $\Upsilon^{\alpha}$ & Section \ref{sec:helmholtz-gen-hyperelasticity} \\
$\bm{T}^{\alpha}$ & The stress conjugate to $\bm{E}^{\alpha}$ with respect to $\Upsilon^{\alpha}$  & \eqref{eq:T_eq_neq_beta}$_3$ \\
$\mathbb Q^{\alpha}$, $\mathbb Q^{\mathrm v \: \alpha}$ & The derivative of $\bm{E}^{\alpha}$ with respect to $\bm{C}$ and the derivative of $\bm{E}^{\mathrm{v}\:\alpha}$ with respect to $\bm{\Gamma}^{\alpha}$ & \eqref{eq:Q_inf_beta_alpha}$_2$, \eqref{eq:Q_inf_beta_alpha}$_3$ \\
$G^{\infty}_{\mathrm{iso} \: \beta}$ & The $\beta$-th quadratic term of $G^{\infty}_{\mathrm{iso}}$ & \eqref{eq:Hill_Gibbs_Upsilon}$_1$ \\
$\tilde{\bm{E}}^{\infty}_{\beta}$ & The generalized strains used to define the quadratic energy $G^{\infty}_{\mathrm{iso} \: \beta}$ & \eqref{eq:Hill_Gibbs_Upsilon}$_1$ \\
$\mu^{\infty}_{\beta}$, $\mu^{\alpha}$ & The shear moduli used in $G^{\infty}_{\mathrm{iso} \: \beta}$ and $\Upsilon^{\alpha}$ & \eqref{eq:Hill_Gibbs_Upsilon} \\
$\tilde{\bm{T}}^{\infty}_{\beta}$, $\tilde{\bm{T}}^{\alpha}$ & The stresses conjugate to $\tilde{\bm{E}}^{\infty}_{\beta}$ with respect to $G^{\infty}_{\mathrm{iso} \: \beta}$ and the stresses conjugate to $\tilde{\bm E}^{\alpha}$ with respect to $\Upsilon^{\alpha}$, respectively & \eqref{eq:tilde-stress-like-projection-tensors}$_1$, \eqref{eq:tilde-stress-like-projection-tensors}$_2$ \\
$\tilde{\mathbb Q}^{\infty}_{\beta}$, $\tilde{\mathbb L}^{\infty}_{\beta}$ & The first and second derivatives of $\tilde{\bm{E}}^{\infty}_{\beta}$ with respect to $\tilde{\bm{C}}$ & \eqref{eq:tilde-stress-like-projection-tensors-2}$_1$, \eqref{eq:tilde_L_beta} \\
$\tilde{\bm{E}}^{\alpha}$, $\mu^{\mathrm{neq}\:\alpha}$ & The strain and shear moduli used in the definition of the quadratic energy $\Upsilon^{\alpha}$ & \eqref{eq:Hill_Gibbs_Upsilon} \\
$\tilde{\bm{T}}^{\alpha}$ & The stresses conjugate to $\tilde{\bm{E}}^{\alpha}$ with respect to $\Upsilon^{\alpha}$ & \eqref{eq:tilde-stress-like-projection-tensors}$_2$ \\
$\tilde{\mathbb Q}^{\alpha}$, $\tilde{\mathbb L}^{\alpha}$ & The first and second derivatives of $\tilde{\bm{E}}^{\alpha}$ with respect to $\tilde{\bm{C}}$ & \eqref{eq:tilde-stress-like-projection-tensors-2}$_2$, \eqref{eq:elements_in_CC_neq}$_2$ \\
$\mathbb{K}^{\alpha}$ & The derivative of $\bm{Q}^{\alpha}$ with respect to $\bm{\Gamma}^{\alpha}$ & \eqref{eq:definition_K} \\
$\mathbb L^{\mathrm{v} \: \alpha}$ & The second derivative of $\bm{E}^{\mathrm{v}\:\alpha}$ with respect to $\bm{\Gamma}^{\alpha}$ & \eqref{eq:Lv_alpha_beta} \\
\hline
\end{longtable} 

\bibliographystyle{elsarticle-num}
\bibliography{viscoelasticity-theory.bib}

\end{document}